\documentclass[11pt,a4paper]{article}
\usepackage{amsthm}
\usepackage{lmodern}
\usepackage{amsmath}
\usepackage{amssymb}
\usepackage{amsthm}
\usepackage{cite}
\usepackage{doi}
\usepackage{enumerate}
\usepackage{epsfig}
\usepackage{float}
\usepackage{geometry}
\usepackage{graphicx}
\usepackage{subcaption}
\usepackage{hyperref}
\usepackage[mathscr]{eucal}
\usepackage{mathtools}
\usepackage{stackengine,scalerel}
\usepackage{tikz}
\usepackage{newtxtext,newtxmath}   
\usepackage{microtype}
\stackMath

\newtheorem{proposition}{Proposition}[section]
\newtheorem{corollary}{Corollary}[section]
\newtheorem{lemma}{Lemma}[section]
\newtheorem{definition}{Definition}[section]
\newtheorem{discussion}{Discussion}[section]
\newtheorem{example}{Example}[section]
\newtheorem{remark}{Remark}[section]

\newcommand{\pmfsub}[1]{{\smash{#1}\vphantom{XYZ}}}
\newcommand{\pmf}[1]{\mathsf{#1}}
\newcommand{\pmfOf}[1]{\pmf{P}_{\pmfsub{#1}}}  
\newcommand{\CondpmfOf}[2]{\pmf{P}_{\pmfsub{#1} \mid \pmfsub{#2}}}

\DeclareMathOperator{\Vertex}{\mathsf{V}}
\DeclareMathOperator{\Edge}{\mathsf{E}}
\DeclareMathOperator{\Cycle}{\mathsf{C}}

\DeclareMathOperator{\Complete}{\mathsf{K}}
\DeclareMathOperator{\Entr}{H}
\DeclareMathOperator{\Clique}{\omega}
\DeclareMathOperator{\Chromatic}{\chi}
\DeclareMathOperator{\Path}{\mathsf{P}}
\DeclareMathOperator{\Star}{\mathsf{S}}

\newcommand{\Gr}[1]{\mathsf{#1}}                          
\newcommand{\V}[1]{\Vertex(#1)}                           
\newcommand{\E}[1]{\Edge(#1)}                             
\newcommand{\CoG}[1]{\Complete_{#1}}                      
\newcommand{\CoBG}[2]{\Complete_{#1,#2}}                  
\newcommand{\CG}[1]{\Cycle_{#1}}                          
\newcommand{\card}[1]{|#1|}                               
\newcommand{\bigcard}[1]{\bigl|#1\bigr|}
\newcommand{\PathG}[1]{\Path_{#1}}                        
\newcommand{\StarG}[1]{\Star_{#1}}                        

\newcommand{\OneTo}[1]{[#1]}
\newcommand{\set}{\ensuremath{\mathcal}}
\newcommand{\es}{\varnothing}                             
\newcommand{\naturals}{\ensuremath{\mathbb{N}}}           
\newcommand{\midnew}{\hspace*{-0.05cm} \mid \hspace*{-0.05cm}}

\newcommand{\clnum}[1]{\Clique(#1)}                       

\newcommand{\chrnum}[1]{\Chromatic(#1)}                   

\newcommand{\Hom}[2]{\mathrm{Hom}(#1,#2)}                 
\newcommand{\homcount}[2]{\mathrm{hom}(#1,#2)}            

\newcommand{\Ent}[1]{\Entr(#1)}

\newcommand{\EntCond}[2]{\Entr(#1 \midnew \kern0.1em #2)}
\newcommand{\bigEntCond}[2]{\Entr\bigl(#1 | \kern0.1em #2\bigr)}
\newcommand{\BigEntCond}[2]{\Entr\Bigl(#1 \kern-0.1em \bigm| \kern-0.1em #2 \Bigr)}
\newcommand{\biggEntCond}[2]{\Entr\biggl(#1 \kern-0.1em \Bigm| \kern-0.1em #2 \biggr)}
\newcommand{\BiggEntCond}[2]{\Entr\Biggl(#1 \kern-0.1em \biggm| \kern-0.1em #2 \Biggr)}

\newcommand{\Expecwrt}[2]{\mathbb{E}_{#1}[#2]}

\newcommand{\ExpecCondwrt}[3]{\mathbb{E}_{#1}[#2 \kern0.1em|\kern0.1em #3]}
\newcommand{\bigExpecCondwrt}[3]{\mathbb{E}_{#1}\bigl[#2 \kern-0.1em \bigm| \kern-0.1em #3\bigr]}
\newcommand{\BigExpecCondwrt}[3]{\mathbb{E}_{#1}\Bigl[#2 \kern-0.1em \Bigm| \kern-0.1em #3\Bigr]}
\newcommand{\biggExpecCondwrt}[3]{\mathbb{E}_{#1}\biggl[#2 \kern-0.1em \biggm| \kern-0.1em #3\biggr]}
\newcommand{\BiggExpecCondwrt}[3]{\mathbb{E}_{#1}\Biggl[#2 \kern-0.1em \Biggm| \kern-0.1em #3\Biggr]}

\hypersetup{
  colorlinks = true,
  urlcolor   = red,
  linkcolor  = blue,
  citecolor  = red
}

\def\@seccntformat#1{\csname the#1\endcsname\quad}
\makeatother

\begin{document}
\setlength{\baselineskip}{1\baselineskip}

\setcounter{page}{1}

\title{\huge{Counting Graph Homomorphisms in Bipartite Settings}}

\author{Igal Sason
\thanks{
Igal Sason is with the Viterbi Faculty of Electrical and Computer Engineering and the Department of Mathematics
(secondary affiliation) at the Technion --- Israel Institute of Technology, Haifa 3200003, Israel. Email: eeigal@technion.ac.il.
This work was presented in part at the {\em 13th European Conference on Combinatorics,
Graph Theory and Applications (EuroComb '25)}, Budapest, Hungary, August 25--29, 2025.}}

\maketitle

\thispagestyle{empty}
\setcounter{page}{1}

\vspace*{-0.5cm}
\begin{abstract}
This paper studies the problem of counting homomorphisms from a bipartite source graph to a bipartite target graph.
An exact formula is first derived for the number of homomorphisms from a complete bipartite graph to a general
bipartite graph. Although exact, its evaluation is typically computationally intensive, and a computationally tractable 
combinatorial lower bound is derived. When the target graph contains no 4-cycles, the lower bound simplifies 
and becomes exact. Two additional lower bounds on the number of homomorphisms from a complete bipartite graph to an 
arbitrary bipartite graph are derived using properties of Shannon entropy. The first depends only on the sizes of the 
partite sets in the source and target graphs, together with the edge density of the target graph.
The second further incorporates the degree profiles of the partite sets of the target graph, thereby strengthening the
first bound. Both entropy-based bounds improve upon the inequality implied by the validity of Sidorenko’s conjecture
for complete bipartite source graphs.
The lower bounds for complete bipartite source graphs are combined with new auxiliary results to derive general lower
bounds on homomorphism counts between arbitrary bipartite graphs. Furthermore, a known reverse Sidorenko inequality is
employed to derive a corresponding upper bound. This upper bound is attained when the source graph is a disjoint union of
complete bipartite graphs, and admits a simple closed-form expression when the target graph contains no 4-cycles.
Numerical results compare the new computationally tractable bounds with exact homomorphism counts in cases where exact 
computation is feasible.
\end{abstract}

\noindent {\bf Keywords.}
Graph homomorphisms; bipartite graphs; homomorphism counts; Shannon entropy; Sidorenko’s conjecture.

\noindent {\bf 2020 Mathematics Subject Classification (MSC).} 05A15, 05C30, 05C60, 05C75, 05C80, 94A17.

\section{Introduction}
\label{section: introduction}

Graph homomorphisms, defined as adjacency-preserving mappings from the vertices of one graph to those of another, are
fundamental constructs in graph theory. This notion unifies various concepts such as graph colorings, independent sets,
and clique covers, providing a versatile framework for studying combinatorial properties of graphs (see, e.g.,
\cite{FederVardi98,HellN26,HellN21,Lovasz12,Zhao17}).

Formally, a homomorphism from a source graph $\Gr{F}$ to a target graph $\Gr{G}$ is a vertex mapping that maps adjacent vertices in $\Gr{F}$
to adjacent vertices in $\Gr{G}$. The set of all such graph homomorphisms is denoted by $\Hom{\Gr{F}}{\Gr{G}}$, and its cardinality
is denoted by $\homcount{\Gr{F}}{\Gr{G}} \triangleq \bigcard{\Hom{\Gr{F}}{\Gr{G}}}$.
For a fixed graph $\Gr{G}$ and an instance graph $\Gr{F}$, determining the existence of a homomorphism from $\Gr{F}$ to $\Gr{G}$ is generally
computationally hard. A study of the existence of graph homomorphisms via semidefinite programming can be found in \cite{GodsilRRSV19}.
Conversely, no-homomorphism lemmas are results that provide necessary conditions for the existence of a graph homomorphism from
$\Gr{F}$ to $\Gr{G}$ (see, e.g., \cite{AlbertsonC85,NaserasrSZ21}).
If $\Gr{G}$ is bipartite or contains a loop, then the problem of deciding if there exists a homomorphism from $\Gr{F}$ to the fixed $\Gr{G}$
is solvable in polynomial time; otherwise, it is NP-complete \cite{HellN26,HellN90}.
Thus, unless $\text{P} = \text{NP}$, the homomorphism problem exhibits a complexity dichotomy \cite{DyerG00,DyerG04,Hell06,CaiCL15,CurticapeanDM17}.

A homomorphism from a bipartite graph $\Gr{F}$ to a graph $\Gr{G}$ that contains at least one edge always exists: one may map all vertices
in one partite set of $\Gr{F}$ to one endpoint of an edge of $\Gr{G}$, and all vertices in the other partite set to the other endpoint. In
contrast, if $\Gr{F}$ is non-bipartite and $\Gr{G}$ is bipartite, then no homomorphism from $\Gr{F}$ to $\Gr{G}$ exists.
Indeed, the existence of such a homomorphism would imply that $\chi(\Gr{F}) \leq \chi(\Gr{G})$, where $\chi(\cdot)$ denotes the chromatic number.
This is impossible because every non-bipartite graph has chromatic number at least~3, whereas every bipartite graph has chromatic number at most~2.
Since the identity mapping is a homomorphism from a graph to itself (i.e., an endomorphism), every non-bipartite graph $\Gr{G}$ admits a non-bipartite
graph $\Gr{F}$ such that $\homcount{\Gr{F}}{\Gr{G}} \geq 1$. Consequently, bipartiteness admits the following characterization in terms of graph
homomorphisms: a graph $\Gr{G}$ is bipartite if and only if $\homcount{\Gr{F}}{\Gr{G}} = 0$ for every non-bipartite graph $\Gr{F}$, including graphs
with loops.

A fundamental result by Lov\'{a}sz \cite{Lovasz67}, together with subsequent extensions in \cite[Section~5.4]{Lovasz12}, \cite[Theorem~4]{Sernau18},
and \cite{GarijoGN11}, asserts that for every finite graph $\Gr{F}$ (possibly with loops), the vector $\bigl(\homcount{\Gr{F}}{\Gr{G}}\bigr)_{\Gr{G}}$,
indexed by all finite graphs $\Gr{G}$, uniquely determines $\Gr{F}$ up to isomorphism. Likewise, for every finite graph $\Gr{G}$, the vector
$\bigl(\homcount{\Gr{F}}{\Gr{G}}\bigr)_{\Gr{F}}$, indexed by all finite connected graphs $\Gr{F}$, uniquely determines $\Gr{G}$ up to isomorphism.
The problem of determining graphs $\Gr{G}$, up to isomorphism, from subvectors of the homomorphism counts $\bigl(\homcount{\Gr{F}}{\Gr{G}}\bigr)_{\Gr{F}}$,
where $\Gr{F}$ ranges over restricted classes of graphs, was studied in \cite{Dvorak10}.

Counting graph homomorphisms is a central problem in combinatorics \cite{Borgs06,HellN26}, having a wide range of applications and interconnections with
extremal graph theory, theoretical computer science, communication networks, group theory, statistical physics, and graph-based machine learning
(see, e.g., \cite{BaoJBCL25,Borgs06,KoppartyR11,HellN26,HellN21,Lovasz12,BrightwellW99,Kahn02,GalvinT04,GarijoNR09,CsikvariRS22,ShamsRC19,Zhao17}).
These counts reflect structural regularities in networks, encode combinatorial constraints, and arise naturally in contexts ranging from subgraph
enumeration to the modeling of physical systems.

Counting homomorphisms from a bipartite graph to a target graph underlies several classical results and applications. A notable example is the theorem 
of \cite{Sidorenko94}, which states that among all trees $\Gr{F}$ on $n$ vertices, the star graph $\Gr{F}=\CoBG{1}{n-1}$ maximizes the number of homomorphisms 
to every graph $\Gr{G}$. Alternative proofs of this theorem were subsequently given in \cite{CsikvariL14,LevinP17}.
Beyond extremal graph theory, this result has also been used in the study of the degree and sensitivity of Boolean functions \cite{GopalanSW16}.

Another important line of research concerns Sidorenko's conjecture (see, e.g., \cite[Conjecture~5.0.5]{Zhao23}).
It conjectures that, for every bipartite graph $\Gr{F}$, the probability that a uniformly selected random mapping from the vertices of $\Gr{F}$ to those
of a graph $\Gr{G}$ is a homomorphism is at least the product, over all edges of $\Gr{F}$, of the probabilities that each edge is mapped to an edge of
$\Gr{G}$. This conjecture provides a lower bound on the number of homomorphisms from a bipartite source graph $\Gr{F}$ to a general target graph $\Gr{G}$.
A necessary condition for $\Gr{F}$ to satisfy Sidorenko's conjecture is that it be bipartite. Moreover, several important classes of bipartite graphs,
including trees and complete bipartite graphs, have been shown to satisfy the conjecture \cite[Theorems~5.5.11--5.5.12]{Zhao23}.

In a related algorithmic direction, the tree-walk algorithm introduced in \cite{CsikvariL14} provides an efficient recursive method for computing
the number of homomorphisms from a tree to an arbitrary graph.

In theoretical computer science, counting homomorphisms is closely connected to the evaluation of conjunctive queries in relational databases and to
counting solutions of constraint satisfaction problems (CSPs). In particular, homomorphisms from $\CoBG{p}{q}$ to a target graph $\Gr{G}$ correspond 
to satisfying assignments of a CSP with $p+q$ variables, each with domain $\V{\Gr{G}}$, and $pq$ binary adjacency constraints, one for each edge of 
$\CoBG{p}{q}$. Equivalently, these assignments are the answers to the conjunctive query formed by these adjacency constraints, with all variables free. 
The following example illustrates this correspondence when the target graph is a cycle.
\begin{example}
{\em Consider the problem of counting homomorphisms from the complete bipartite graph $\CoBG{p}{q}$, where $p, q \in \naturals$, to
a target cycle graph $\Gr{G} \cong \CG{n}$, where $n \geq 3$, with vertex set $\V{\Gr{G}} = \{a_1, \ldots, a_n\}$ and edge set
\[
\E{\Gr{G}} = \bigl\{\{a_1, a_2\}, \ldots, \{a_{n-1}, a_n\}, \{a_n, a_1\} \bigr\}.
\]
The source graph $\CoBG{p}{q}$ has partite sets $\set{U} = \{u_1, \ldots, u_p\}$ and $\set{V} = \{v_1, \ldots, v_q\}$,
and the edge set
\[
\E{\CoBG{p}{q}} = \bigl\{\{u_k,v_\ell\}: k \in \OneTo{p}, \; \ell \in \OneTo{q} \bigr\}.
\]
To express $\Hom{\CoBG{p}{q}}{\CG{n}}$ as a CSP, we define:
\begin{itemize}
  \item Variables: $x_{u_1}, \ldots , x_{u_p}, x_{v_1}, \ldots, x_{v_q}$,
  \item Domain: $\set{D} = \{a_1, \ldots, a_n \}$,
  \item Constraints: For each edge $\{u_k, v_\ell\} \in \E{\CoBG{p}{q}}$ (i.e., for every $(k,\ell) \in \OneTo{p} \times \OneTo{q}$) impose the constraint
  that the pair $(x_{u_k}, x_{v_\ell})$ corresponds to an edge in $\E{\Gr{G}}$, i.e.,
  \[
  (x_{u_k}, x_{v_\ell}) \in \bigl\{(a_1,a_2), \ldots, (a_{n-1}, a_n), (a_n,a_1), (a_2,a_1), \ldots, (a_n, a_{n-1}), (a_1, a_n)\bigr\}, \, \forall \, (k,\ell) \in \OneTo{p} \times \OneTo{q}.
  \]
\end{itemize}
Solutions to this CSP are in one-to-one correspondence with homomorphisms $\phi \colon \V{\CoBG{p}{q}} \to \V{\CG{n}}$, whose
total number is given by (see Corollary~\ref{corollary: K_{p,q} --> cycle})
\[
\homcount{\CoBG{p}{q}}{\CG{n}} =
\begin{dcases}
n(2^p + 2^q-2),     \quad & \mbox{if $n=3$ or $n \geq 5$,} \\
2^{p+q+1},          \quad & \mbox{if $n=4$.}
\end{dcases}
\]}
\end{example}

In extremal graph theory and the study of graph limits, homomorphism counts from a source graph $\Gr{F}$ to a target
graph $\Gr{G}$ are used to formulate and prove inequalities, such as those related to Sidorenko's conjecture, and to
define graph homomorphism densities. These quantities play a central role in the analytic theory of dense graph sequences,
and are used in the study of graphons and convergence notions in graph limit theory (see, e.g., the textbooks
\cite{Lovasz12,Zhao23}). Applications also appear in random graph models, where the asymptotic behavior of
$\homcount{\Gr{F}}{\Gr{G}_n}$ for a sequence of graphs $\{\Gr{G}_n\}$ helps characterize structural
properties of random graphs.

Applying information-theoretic tools to obtain combinatorial results has proven to be fruitful. Notably,
Shannon entropy has significantly deepened the understanding of the structural and quantitative properties of combinatorial
objects by enabling concise and often elegant proofs of classical results in combinatorics (see, e.g.,
\cite[Chapter~5]{Zhao23}, \cite[Chapter~37]{AignerZ18}, \cite[Chapter~22]{Jukna11}, and \cite{ChungGFS86,ConlonKLL18,FriedgutK98,Friedgut04,
Galvin14,HatamiJS18,Kahn01,Kahn02,GalvinT04,MadimanT_IT10,Radhakrishnan97,Radhakrishnan01,Sason21,Sason21b,Sason25,Szegedy15a,Szegedy15b,WangTL23}).
Part of this work employs such information-theoretic tools, specifically Shannon entropy,
to derive lower bounds on the number of homomorphisms from a complete bipartite graph to an arbitrary bipartite graph.
This approach differs from the use of Shearer's entropy lemma, which leads to the derivation of upper bounds on graph
homomorphisms (see \cite{GalvinT04,Galvin14,MadimanT_IT10}). The entropy method has been used to prove
that Sidorenko's conjecture holds for some bipartite graphs (see, e.g., \cite[Section~5.5]{Zhao23}), and the derivation of
entropy-based lower bounds in this work offers further insight into the interplay between combinatorial structures and
information-theoretic principles.

This paper studies the problem of counting homomorphisms from a bipartite source graph to a bipartite target graph,
and is organized as follows. Section~\ref{section: preliminaries} presents essential preliminaries.
Section~\ref{section: exact expressions and comb. bounds} first derives an exact expression for the number of
homomorphisms from a complete bipartite graph to an arbitrary bipartite graph (Proposition~\ref{prop: exact counting: CoBG --> BG});
while exact, this expression is generally computationally intensive. This section then presents a computationally tractable combinatorial
lower bound for this quantity (Proposition~\ref{proposition: LB and exact hom(CoBG,G)}), which becomes exact when the target
bipartite graph contains no 4-cycles.
Section~\ref{section: IT bounds} derives two additional lower bounds on the number of homomorphisms from a complete bipartite
graph to an arbitrary bipartite graph, using properties of Shannon entropy. The first bound depends only on the sizes of the partite
sets of the source and target graphs and the edge density of the target graph (Proposition~\ref{prop.: IT-LB}). The second is a
refinement that also incorporates the degree profiles within the partite sets of the target graph (Proposition~\ref{prop.: refined IT-LB}),
yielding a strengthening of the first. Both entropy-based bounds are demonstrated to improve upon the inequality implied by the validity
of Sidorenko’s conjecture for complete bipartite graphs (Discussion~\ref{discussion: Comparison of the 1st IT LB to Sidorenko's lower bound}).
Section~\ref{section: boounds on homomorphism counts between bipartite graphs} introduces upper and lower bounds on homomorphism counts
between arbitrary bipartite graphs (Propositions~\ref{proposition: UB on hom(BG, BG)} and~\ref{proposition: LB on hom(BG, BG)}), based
on the results of Sections~\ref{section: exact expressions and comb. bounds} and~\ref{section: IT bounds}, a reverse Sidorenko inequality
from \cite{SahSSZ20}, and new auxiliary results derived in the same section.
Finally, Section~\ref{section: numerical results} presents numerical comparisons between the proposed easy-to-compute bounds and exact
homomorphism counts, whenever the latter are computationally feasible.

\section{Preliminaries}
\label{section: preliminaries}

Let $\V{\Gr{G}}$ and $\E{\Gr{G}}$ denote the vertex set and edge set of a graph $\Gr{G}$, respectively.
For adjacent vertices $u,v \in \V{\Gr{G}}$, let $e = \{u,v\} \in \E{\Gr{G}}$ denote the edge connecting them.
Throughout this paper, $\naturals = \{1, 2, 3, \ldots \}$ denotes the set of natural numbers, 
the shorthand $\OneTo{n} \triangleq \{1, \ldots, n\}$ is used for all $n \in \naturals$,
the empty set is denoted by $\es$, and a graph $\Gr{G}$ is called nonempty if $ \E{\Gr{G}} \neq \es$.

Let $\Gr{F}$ and $\Gr{G}$ be finite and undirected graphs with no multiple edges between any pair of distinct vertices.
A {\em homomorphism} from $\Gr{F}$ to $\Gr{G}$, denoted by $\Gr{F} \to \Gr{G}$, is a mapping $\phi \colon \V{\Gr{F}} \to \V{\Gr{G}}$
(not necessarily injective) such that every edge of $\Gr{F}$ is mapped to an edge of $\Gr{G}$, i.e., $\{\phi(u), \phi(v)\} \in \E{\Gr{G}}$
whenever $\{u,v\} \in \E{\Gr{F}}$. No condition is imposed on non-edges of $\Gr{F}$: a pair of non-adjacent vertices in $\Gr{F}$
may be mapped to a non-edge, to an edge, or even to the same vertex of $\Gr{G}$.
The graphs $\Gr{F}$ and $\Gr{G}$ are typically referred to as the source and target graphs, respectively. The source graph
$\Gr{F}$ is usually a simple graph (with no loops or multiple edges), whereas the target $\Gr{G}$ may be allowed to contain loops.
Unless specified otherwise, all graphs considered in this paper are assumed to be simple, finite, and undirected.

A {\em copy} of $\Gr{F}$ in $\Gr{G}$ is a subgraph of $\Gr{G}$ that is isomorphic to $\Gr{F}$. In contrast, the
image of a graph homomorphism $\Gr{F} \to \Gr{G}$ need not correspond to such a copy since the homomorphism is not
required to be injective on vertices.

Let $\Hom{\Gr{F}}{\Gr{G}}$ denote the set of all homomorphisms from $\Gr{F}$ to $\Gr{G}$, and let
$\homcount{\Gr{F}}{\Gr{G}} \triangleq \bigcard{\Hom{\Gr{F}}{\Gr{G}}}$
be the number of such graph homomorphisms (using lowercase letters to denote cardinalities). These
are called {\em homomorphism numbers}. As illustrative examples, for every $n \in \naturals$,
$\homcount{\CoG{n}}{\Gr{G}}$ is equal to $n!$ times the number of complete subgraphs of $\Gr{G}$ on
$n$ vertices, and $\homcount{\Gr{G}}{\CoG{n}}$ is equal to the number of proper $n$-colorings of $\Gr{G}$,
in which adjacent vertices are assigned distinct colors from a set of $n$ colors.
Let $\clnum{\Gr{G}}$ and $\chrnum{\Gr{G}}$ denote the clique number and chromatic number of a graph
$\Gr{G}$, respectively. Consequently, $\clnum{\Gr{G}}$ equals the largest $n \in \naturals$ for which
there exists a homomorphism $\CoG{n} \to \Gr{G}$, and $\chrnum{\Gr{G}}$ equals the smallest $n \in \naturals$
such that a homomorphism $\Gr{G} \to \CoG{n}$ exists. In this way, graph homomorphisms provide a unifying 
framework for characterizing classical graph invariants, including the independence, clique, and chromatic 
numbers, each of which is NP-hard to compute \cite{GareyJ79}.

Some simple and useful identities, which are used to simplify the computation of homomorphism numbers, are next stated
(see, e.g., \cite[Section~5.2.3]{Lovasz12} and \cite{Borgs06}).
\begin{enumerate}[(1)]
\item If $\Gr{F}$ is a disjoint union of $r$ subgraphs $\Gr{F}_1, \ldots, \Gr{F}_r$, for some $r \geq 1$, then
\begin{align}
\label{eq1: homomorphism numbers}
\homcount{\Gr{F}}{\Gr{G}} = \prod_{j=1}^r \homcount{\Gr{F}_j}{\Gr{G}}.
\end{align}
\item If $\Gr{F}$ is a connected graph, and $\Gr{G}$ is a disjoint union of $r$ subgraphs $\Gr{G}_1, \ldots, \Gr{G}_r$,
for some $r \geq 1$, then
\begin{align}
\label{eq2: homomorphism numbers}
\homcount{\Gr{F}}{\Gr{G}} = \sum_{j=1}^r \homcount{\Gr{F}}{\Gr{G}_j}.
\end{align}
\item For two simple graphs $\Gr{G}_1$ and $\Gr{G}_2$, their {\em categorical product} $\Gr{G}_1 \times \Gr{G}_2$
(a.k.a. their {\em tensor product} or {\em Kronecker product}) is defined to be a graph whose vertex set is
the Cartesian product $\V{\Gr{G}} = \V{\Gr{G}_1} \times \V{\Gr{G}_2}$, in which two vertices $(i_1, i_2), \,
(j_1, j_2) \in \V{\Gr{G}}$ are adjacent if $\{i_1, j_1\} \in \E{\Gr{G}_1}$ and $\{i_2, j_2\} \in \E{\Gr{G}_2}$.
For $r \geq 3$, the categorical product $\Gr{G}_1 \times \ldots \times \Gr{G}_r$ of $r$ simple graphs is recursively
defined by $(\Gr{G}_1 \times \ldots \times \Gr{G}_{r-1}) \times \Gr{G}_r$, and this product is commutative and
associative (up to isomorphism). For all $r \geq 1$,
\begin{align}
\label{eq3: homomorphism numbers}
\homcount{\Gr{F}}{\Gr{G}_1 \times \ldots \times \Gr{G}_r} = \prod_{j=1}^r \homcount{\Gr{F}}{\Gr{G}_j}.
\end{align}
By introducing the operation of exponentiation, there is also an identity for $\homcount{\Gr{F}_1 \times \ldots \times \Gr{F}_r}{\Gr{G}}$ \cite{Lovasz67}.
\end{enumerate}

Let ${\bf{A}}_{\Gr{G}} = (A_{\Gr{G}}(i,j))_{i,j \in \OneTo{n}}$ denote the adjacency matrix of a graph $\Gr{G}$ of order $n$, assumed to have
no multiple edges but possibly containing loops.
By definition, the number of homomorphisms from a simple source graph $\Gr{F}$ to the target graph $\Gr{G}$ is given by
\begin{align}
\label{eq: 13.08.25}
\homcount{\Gr{F}}{\Gr{G}} = \sum_{\phi \colon \V{\Gr{F}} \to \V{\Gr{G}}} \; \prod_{\{u,v\} \in \E{\Gr{F}}} A_{\Gr{G}}\bigl(\phi(u),\phi(v)\bigr).
\end{align}

In addition to homomorphism numbers, we now introduce {\em homomorphism densities}, which are closely related.
\begin{definition}[Homomorphism densities]
\label{definition: Homomorphism densities}
{\em Let $\Gr{F}$ and $\Gr{G}$ be graphs, $v(\Gr{F}) \triangleq \card{ \hspace*{-0.05cm} \V{\Gr{F}}}$, and
$v(\Gr{G}) \triangleq \card{\hspace*{-0.05cm} \V{\Gr{G}}}$.
The {\em $\Gr{F}$-homomorphism density in $\Gr{G}$} (or simply $\Gr{F}$-density in $\Gr{G}$) is
the probability that a uniformly random vertex mapping $\phi \colon \V{\Gr{F}} \to \V{\Gr{G}}$ satisfies $\phi \in \Hom{\Gr{F}}{\Gr{G}}$,
i.e., it is given by
\begin{align}
\label{eq: 16.06.2025}
t(\Gr{F}, \Gr{G}) \triangleq \frac{\homcount{\Gr{F}}{\Gr{G}}}{v(\Gr{G})^{\, v(\Gr{F})}}.
\end{align}}
\end{definition}

\begin{example}[Homomorphism densities]
{\em By Definition~\ref{definition: Homomorphism densities}, we have $t(\CoG{1}, \Gr{G}) = 1$, and
\begin{align}
t(\CoG{2}, \Gr{G}) = \frac{2 \, e(\Gr{G})}{v(\Gr{G})^2}, \label{eq: edges}
\end{align}
where $e(\Gr{G}) \triangleq \card{\hspace*{-0.05cm} \E{\Gr{G}}}$. Moreover, for all $n \in \naturals$,
\begin{align}
t(\Gr{G}, \CoG{n}) = \frac{P(\Gr{G},n)}{n^{\, v(\Gr{G})}},
\end{align}
where  $P(\Gr{G}, n)$ denotes the number of proper $n$-colorings of the vertices of $\Gr{G}$, and it is a
polynomial in $n$ (depending only on $\Gr{G}$). For this reason, $P(\Gr{G}, \cdot)$ is called the
{\em chromatic polynomial of $\Gr{G}$}. By definition, if $n < \chi(\Gr{G})$, then $P(\Gr{G},n) = 0$ and
$t(\Gr{G}, \CoG{n}) = 0$.}
\end{example}

This paper relies in part on properties of the Shannon entropy for the derivation of bounds on
the number of homomorphisms in bipartite settings. We briefly provide here the standard notation
and background, following \cite{CoverT06}.

Let $X$ be a discrete random variable that takes values in a set $\set{X}$,
and let $\pmfOf{X}$ be the probability mass function (PMF) of $X$. The
{\em Shannon entropy} of $X$ is given by
\begin{align}
\label{eq: entropy}
\Ent{X} \triangleq -\sum_{x \in \set{X}} \pmfOf{X}(x) \, \log \pmfOf{X}(x),
\end{align}
where all logarithms are taken to a fixed base (e.g., base~2 or base~$\mathrm{e}$).
The exponential function, denoted by $\exp(\cdot)$, is defined as the inverse of the logarithm with
respect to the same base.

Let $X$ and $Y$ be discrete random variables with joint PMF $\pmfOf{XY}$, and conditional
PMF of $X$ given $Y$ denoted by $\CondpmfOf{X}{Y}$. The {\em conditional entropy} of $X$ given $Y$
is defined as
\begin{subequations} \label{eq: conditional entropy}
\begin{align}
\label{eq: conditional entropy 1}
\EntCond{X}{Y} & \triangleq -\sum_{(x,y) \in \set{X} \times \set{Y}}
\pmfOf{XY}(x,y) \log \CondpmfOf{X}{Y}(x|y) \\[0.1cm]
\label{eq: conditional entropy 1.5}
&= \sum_{y \in \set{Y}} \pmfOf{Y}(y) \, \EntCond{X}{Y=y}.
\end{align}
\end{subequations}
We make use of the following two basic properties of the Shannon entropy.
\begin{itemize}
\item {\em Uniform bound}: If $X$ is a discrete random variable supported on a finite set
$\set{X}$, then
\begin{align}
\label{eq: uniform bound}
\Ent{X} \leq \log \card{\set{X}},
\end{align}
with equality if and only if $X$ is equiprobable over $\set{X}$, i.e., $\pmfOf{X}(x) = \frac1{\card{\set{X}}}$ for
all $x \in \set{X}$.
\item {\em Chain rule}: The Shannon entropy of a discrete random vector $X^n \triangleq (X_1, \ldots, X_n)$ satisfies
\begin{eqnarray}
\label{eq: chain rule}
\Ent{X^n} = \sum_{i=1}^n \EntCond{X_i}{X^{i-1}},
\end{eqnarray}
where $X^{0}$ indicates the empty vector and $\EntCond{X_1}{X^{0}} = \Ent{X_1}$.
\end{itemize}

\section{Exact Expressions and Combinatorial Lower Bounds}
\label{section: exact expressions and comb. bounds}

This section begins by presenting an exact expression for the number of
homomorphisms from a complete bipartite graph to an arbitrary bipartite
graph. To this end, we first define Stirling numbers of the second kind.

\begin{definition}
\label{def: Stirling number of 2nd type}
{\em Let $n, k \in \naturals$. The Stirling number of the second kind,
denoted by $S(n,k)$, is the number of partitions of the set $\OneTo{n}$
into $k$ nonempty blocks. We set $S(n,k) \triangleq 0$ for $k>n$,
and adopt the conventions $S(0,0) \triangleq 1$ and
$S(n,0)=S(0,n) \triangleq 0$ for $n \in \naturals$.}
\end{definition}

By Definition~\ref{def: Stirling number of 2nd type},
$S(n, n) = 1$ since each element must be placed in its 
own singleton set, and $S(n, 1) = 1$ as the entire set 
forms a single block.
Moreover, these numbers satisfy the recurrence relation
(see \cite[Eq.~(6.3)]{GrahamKP94})
\begin{align}
\label{eq: recursive equation S}
S(n,k) = k\,S(n-1,k) + S(n-1,k-1),
\quad n,k \in \naturals,\; k \leq n.
\end{align}
This recurrence equation yields the following closed-form 
expression (see \cite[Eq.~(6.19)]{GrahamKP94}):
\begin{align}
\label{eq: closed-form for Stirling}
S(n,k) = \frac{1}{k!}
\sum_{j=0}^{k} \biggl\{ (-1)^{k-j} \binom{k}{j} \, j^n \biggr\},
\quad 1 \leq k \leq n.
\end{align}

\begin{proposition}
\label{prop: exact counting: CoBG --> BG}
{\em Let $\Gr{G}$ be a bipartite graph with partite sets $\set{L}$ and $\set{R}$, and let $p, q \in \naturals$. Then,
\begin{align}
\label{eq: exact counting: CoBG --> BG}
\homcount{\CoBG{p}{q}}{\Gr{G}} = \sum_{k=1}^p \sum_{\ell=1}^q  \Bigl\{ k! \, \ell! \, S(p,k) \, S(q,\ell)
\, \bigl[N_{k,\ell}(\Gr{G}) + N_{\ell,k}(\Gr{G}) \bigr] \Bigr\},
\end{align}
where $S(\cdot,\cdot)$ denotes the Stirling number of the second kind (see Definition~\ref{def: Stirling number of 2nd type}), 
and $N_{k,\ell}(\Gr{G})$ denotes the number of pairs of subsets $A \subseteq \set{L}$ and $B \subseteq \set{R}$ with
$|A|=k$ and $|B|=\ell$ such that every vertex in $A$ is adjacent in $\Gr{G}$ to every vertex in $B$.}
\end{proposition}

\begin{proof}
A homomorphism in $\Hom{\CoBG{p}{q}}{\Gr{G}}$, from the complete bipartite graph $\CoBG{p}{q}$ to a bipartite
graph $\Gr{G}$, maps the vertices of $\CoBG{p}{q}$ to those of a bipartite clique in $\Gr{G}$, preserving adjacency 
and either preserving or interchanging the two partite sets. Specifically:
\begin{enumerate}
\item A homomorphism $\phi \colon \V{\CoBG{p}{q}} \to \V{\Gr{G}}$ maps each edge in $\CoBG{p}{q}$ to an edge in $\Gr{G}$.
\item Vertices in the same partite set of $\CoBG{p}{q}$ must be mapped to vertices in the same partite set of $\Gr{G}$.
Otherwise, there exist two vertices $u$ and $w$ in the same partite set of $\CoBG{p}{q}$ that are mapped to different
partite sets of $\Gr{G}$, say $\phi(u) \in \set{L}$ and $\phi(w) \in \set{R}$. Let $v$ be a neighbor of $u$ in $\CoBG{p}{q}$.
Then $\phi(v) \in \set{R}$, so both $v$ and $w$ are mapped into $\set{R}$. But $v$ and $w$ are vertices in different partite 
sets of $\CoBG{p}{q}$ and hence are adjacent, whereas $\{\phi(v), \phi(w)\} \notin \E{\Gr{G}}$. This leads to a contradiction 
since $\phi \in \Hom{\CoBG{p}{q}}{\Gr{G}}$.
\item All vertices in the same partite set of $\CoBG{p}{q}$ must be mapped to non-isolated vertices in the same partite
set of $\Gr{G}$. Indeed, if some image were isolated in $\Gr{G}$, say $\phi(u)$, then for a neighbor $v$ of $u$ in $\CoBG{p}{q}$,
the edge $\{u,v\} \in \E{\CoBG{p}{q}}$ would be mapped to a non-edge $\{\phi(u), \phi(v)\} \notin \E{\Gr{G}}$, contradicting the
homomorphism property.
\item The subgraph of $\Gr{G}$ induced by the image of a homomorphism $\phi \in \Hom{\CoBG{p}{q}}{\Gr{G}}$ is a complete bipartite 
graph in $\Gr{G}$, i.e., a bipartite clique. Suppose not. Then there exist two non-adjacent vertices in the image, one from each 
partite set. By Item~2, these correspond to vertices in opposite partite sets of $\CoBG{p}{q}$, which are adjacent there. Since 
homomorphisms preserve adjacency, their images should be adjacent, a contradiction.
\end{enumerate}
Let $\set{U}$ and $\set{V}$ denote the partite sets of $\CoBG{p}{q}$ of sizes $p$ and $q$, respectively.
The image of a homomorphism in $\Hom{\CoBG{p}{q}}{\Gr{G}}$ is the vertex set of a bipartite clique in $\Gr{G}$
that is isomorphic to the complete bipartite graph $\CoBG{k}{\ell}$, for some $k \in \OneTo{p}$ and $\ell \in \OneTo{q}$,
since the homomorphism need not be injective. 
Here, $k$ and $\ell$ are the sizes of the images of $\set{U}$ and $\set{V}$, respectively.
Concretely, this corresponds to one of the following two possibilities:
\begin{itemize}
\item The $p$ vertices of $\set{U}$ are mapped onto a subset of size $k$ of the partite set $\set{L}$
of $\Gr{G}$, while the $q$ vertices of $\set{V}$ are mapped onto a subset of size $\ell$ of the
opposite partite set $\set{R}$ of $\Gr{G}$.
\item Conversely, the $p$ vertices of $\set{U}$ are mapped onto a subset of size $k$ of $\set{R}$, 
while the $q$ vertices of $\set{V}$ are mapped onto a subset of size $\ell$ of $\set{L}$.
\end{itemize}
Hence, the number of choices of image sets together with their orientation relative to the source partite sets
is $h_{k,\ell} \triangleq N_{k,\ell}(\Gr{G}) + N_{\ell,k}(\Gr{G})$. When $k=\ell$, each image clique is counted 
twice, once for each orientation. These orientations give distinct homomorphisms.

By Definition~\ref{def: Stirling number of 2nd type}, the number of surjective functions
$f \colon \OneTo{n} \to \OneTo{k}$, where $n,k \in \naturals$ and $k \in \OneTo{n}$, equals $k! \, S(n,k)$.
Thus, for fixed $k \in \OneTo{p}$ and $\ell \in \OneTo{q}$, the number of homomorphisms
$\phi \in \Hom{\CoBG{p}{q}}{\Gr{G}}$ satisfying $|\phi(\set{U})|=k$ and $|\phi(\set{V})|=\ell$ equals
\begin{align}
& \bigl(k! \, S(p,k)\bigr)\cdot\bigl(\ell! \, S(q,\ell)\bigr)\, h_{k,\ell} \nonumber \\
\label{eq1: 09.07.25}
&= k! \,\ell! \, S(p,k)\, S(q,\ell)\,\bigl[N_{k,\ell}(\Gr{G}) + N_{\ell,k}(\Gr{G})\bigr].
\end{align}
Here, the factors $k!\,S(p,k)$ and $\ell!\,S(q,\ell)$ count surjective mappings from the $p$ 
(respectively, $q$) vertices of $\CoBG{p}{q}$ onto $k$ (respectively, $\ell$) distinct
vertices in the two partite sets of $\Gr{G}$. Specifically, the Stirling numbers $S(p,k)$ and 
$S(q,\ell)$ count partitions of the source vertices into $k$ and $\ell$ nonempty sets, while 
the factorials $k!$ and $\ell!$ give the number of assignments of these nonempty, disjoint sets
to the actual target vertices of a bipartite clique in $\Gr{G}$ that is isomorphic to $\CoBG{k}{\ell}$.
For each choice of image sets and orientation, every pair of these surjections defines a homomorphism, 
since all edges between the two image sets are present. Consequently, summing the expression in 
\eqref{eq1: 09.07.25} over all $k \in \OneTo{p}$ and $\ell \in \OneTo{q}$ yields the total number of 
homomorphisms from $\CoBG{p}{q}$ to $\Gr{G}$, giving equality~\eqref{eq: exact counting: CoBG --> BG}.
\end{proof}

\begin{remark}
\label{rem: connectivity in G not needed}
{\em The bipartite graph $\Gr{G}$ in Proposition~\ref{prop: exact counting: CoBG --> BG}
is not required to be connected. Indeed, every homomorphism from the complete
bipartite graph $\CoBG{p}{q}$ to $\Gr{G}$ maps the two partite classes of
$\CoBG{p}{q}$ to subsets of vertices in $\Gr{G}$ that induce a bipartite clique.
Since such a bipartite clique is necessarily contained in a single (connected)
component of $\Gr{G}$, the quantities $N_{k,\ell}(\Gr{G})$ and
$N_{\ell,k}(\Gr{G})$ automatically account for all components of $\Gr{G}$.}
\end{remark}

\begin{remark}
{\em Let $k,\ell \geq 2$. In the setting of Proposition~\ref{prop: exact counting: CoBG --> BG}, if the bipartite
graph $\Gr{G}$ contains no $4$-cycle, then $N_{k,\ell}(\Gr{G})=0$ for all such pairs $(k,\ell)$.
Indeed, every complete bipartite graph $\CoBG{k}{\ell}$ with $k,\ell \geq 2$ contains a $4$-cycle. In particular, 
this conclusion holds if $\Gr{G}$ has girth greater than $4$ or is acyclic. Moreover, $N_{k,\ell}(\Gr{G})=0$ trivially if
$k>\card{\set{L}}$ or $\ell>\card{\set{R}}$. In general, however, evaluating the right-hand side of
\eqref{eq: exact counting: CoBG --> BG} may remain computationally demanding for large graphs, since counts that 
are not known {\em a priori} to vanish must still be determined.
For a bipartite graph $\Gr{G}$ with partite sets $\set{L}$ and $\set{R}$ of sizes $n_1$ and $n_2$, respectively, 
a direct enumeration of $N_{k,\ell}(\Gr{G})$ involves examining $\binom{n_1}{k}\binom{n_2}{\ell}$ pairs of vertex 
subsets and checking whether all edges between the two selected subsets are present.}
\end{remark}

\begin{proposition}
\label{proposition: LB and exact hom(CoBG,G)}
{\em Let $\Gr{G}$ be a bipartite graph, and let $p, q \in \naturals$. Let $d_{\Gr{G}}(w)$ denote the degree
of a vertex $w \in \V{\Gr{G}}$. Then,
\begin{align}
\label{eq1: 01.07.25}
\homcount{\CoBG{p}{q}}{\Gr{G}} \geq \sum_{w \in \V{\Gr{G}}} \bigl\{d_{\Gr{G}}(w)^p
+ d_{\Gr{G}}(w)^q \bigr\} - 2 \, \card{\hspace*{-0.05cm} \E{\Gr{G}}},
\end{align}
and equality holds in \eqref{eq1: 01.07.25} if and only if $\Gr{G}$ is $\CG{4}$-free, or $p=1$, or $q=1$.}
\end{proposition}

\begin{proof}
Let $\Gr{G}$ be a bipartite graph.
By Proposition~\ref{prop: exact counting: CoBG --> BG},
\begin{align}
\label{eq3a: 09.07.25}
\homcount{\CoBG{p}{q}}{\Gr{G}} &= \sum_{k=1}^p \sum_{\ell=1}^q  k! \, \ell! \, S(p,k) \, S(q,\ell)
\, \bigl[N_{k,\ell}(\Gr{G}) + N_{\ell,k}(\Gr{G}) \bigr] \\
\label{eq3b: 09.07.25}
&\geq \sum_{k=1}^p \bigl\{ k! \, S(p,k) \, \bigl[N_{k,1}(\Gr{G}) + N_{1,k}(\Gr{G}) \bigr] \bigr\}
+ \sum_{\ell=1}^q \bigl\{ \ell! \, S(q,\ell) \, \bigl[N_{1,\ell}(\Gr{G}) + N_{\ell,1}(\Gr{G}) \bigr] \bigr\} 
-2 N_{1,1}(\Gr{G}) \\
\label{eq3c: 09.07.25}
&= \sum_{k=1}^p \bigl\{ k! \, S(p,k) \, \bigl[N_{k,1}(\Gr{G}) + N_{1,k}(\Gr{G}) \bigr] \bigr\}
+ \sum_{\ell=1}^q \bigl\{ \ell! \, S(q,\ell) \, \bigl[N_{1,\ell}(\Gr{G}) + N_{\ell,1}(\Gr{G}) \bigr] \bigr\} 
-2 \, \card{\hspace*{-0.05cm} \E{\Gr{G}}},
\end{align}
where equality~\eqref{eq3a: 09.07.25} relies on \eqref{eq: exact counting: CoBG --> BG};
inequality \eqref{eq3b: 09.07.25} is obtained by omitting the nonnegative terms in the double sum on the right-hand
side of \eqref{eq3a: 09.07.25} corresponding to indices $k, \ell \geq 2$, and using the identity 
$S(n,1)=1$ for all $n \in \naturals$ (see Definition~\ref{def: Stirling number of 2nd type});
finally, equality~\eqref{eq3c: 09.07.25} holds since by definition $N_{1,1}(\Gr{G}) = \card{\hspace*{-0.05cm} \E{\Gr{G}}}$.

\noindent 
Define a function $g \colon \naturals \to \naturals \cup \{0\}$ by
\begin{align}
\label{eq4: 09.07.25}
g(m) \triangleq \sum_{k=1}^m k! \, S(m,k) \, \bigl[N_{k,1}(\Gr{G}) + N_{1,k}(\Gr{G}) \bigr], \quad m \in \naturals.
\end{align}
Combining \eqref{eq3c: 09.07.25} and \eqref{eq4: 09.07.25} gives
\begin{align}
\label{eq5: 09.07.25}
\homcount{\CoBG{p}{q}}{\Gr{G}} &\geq g(p) + g(q) - 2 \, \card{\hspace*{-0.05cm} \E{\Gr{G}}}.
\end{align}
Proving that
\begin{align}
\label{eq6: 09.07.25}
g(m) = \sum_{w \in \V{\Gr{G}}} d_{\Gr{G}}(w)^m, \quad m \in \naturals
\end{align}
completes the proof of \eqref{eq1: 01.07.25} by substituting \eqref{eq6: 09.07.25} into \eqref{eq5: 09.07.25}.
Equality~\eqref{eq6: 09.07.25} indeed holds, as Proposition~\ref{prop: exact counting: CoBG --> BG} implies that $g(m)$
in \eqref{eq4: 09.07.25} is equal to the number of homomorphisms from the star graph $\StarG{m+1} = \CoBG{1}{m}$ to $\Gr{G}$,
which is known to be equal to the right-hand side of \eqref{eq6: 09.07.25} (see, e.g., \cite[Example 5.10]{Lovasz12}).

We next show that inequality~\eqref{eq1: 01.07.25} holds with equality if and only if the bipartite graph $\Gr{G}$
is $\CG{4}$-free, or $p=1$, or $q=1$. If $p=1$, then 
\[
\homcount{\CoBG{p}{q}}{\Gr{G}} = \homcount{\StarG{q+1}}{\Gr{G}} = g(q) = g(p) + g(q) - 2 \, \card{\hspace*{-0.05cm} \E{\Gr{G}}},
\]
and similarly, by replacing $p=1$ with $q=1$, the same conclusion holds by symmetry. Otherwise, if $p,q \geq 2$, we 
show that inequality~\eqref{eq1: 01.07.25} holds with equality if and only if the bipartite graph $\Gr{G}$ is $\CG{4}$-free,
which in other words means that $\Gr{G}$ does not contain a bipartite clique isomorphic to $\CoBG{2}{2}$. By the first part 
of the proof, this is equivalent to equality in inequality~\eqref{eq3b: 09.07.25}. Since inequality~\eqref{eq3b: 09.07.25} 
is obtained by omitting the nonnegative terms with $2 \leq k \leq p$ and $2 \leq \ell \leq q$, and since $S(p,2)S(q,2)>0$, 
equality implies that $N_{2,2}(\Gr{G})=0$.
Conversely, if $N_{2,2}(\Gr{G})=0$, then $N_{k,\ell}(\Gr{G})=0$ for every $k,\ell \geq 2$, because every complete bipartite 
graph with at least two vertices in each partite set contains a copy of $\CoBG{2}{2}$. Thus, all the omitted terms vanish, 
and equality holds. Consequently, if $p,q \geq 2$, then equality holds in \eqref{eq1: 01.07.25} if and only if $\Gr{G}$ is 
$\CG{4}$-free.
\end{proof}

\begin{remark}
{\em Let $\Gr{G}$ be a $\CG{4}$-free graph with $n$ vertices and $m$ edges. By \cite[p.~200]{AignerZ18}, we have
\begin{align}
\label{eq3: 15.07.25}
m \leq \bigl\lfloor \tfrac{1}{4} \, n \, \bigl(1 + \sqrt{4n-3}\bigr) \bigr\rfloor.
\end{align}
If $\Gr{G}$ is also bipartite, having partite sets of sizes $n_1$ and $n_2$, then the upper bound \eqref{eq3: 15.07.25}
on the size of $\Gr{G}$ can be refined to
\begin{align}
\label{eq4: 15.07.25}
m \leq \biggl\lfloor \tfrac{1}{4} \, n \, \biggl(1 + \sqrt{4n-3-\frac{8n_1 n_2}{n}} \, \biggr) \biggr\rfloor.
\end{align}
This follows by a similar line of analysis to that in \cite[p.~200]{AignerZ18}, with the modification at the outset that,
under the assumption that $\Gr{G}$ is bipartite with partite sets of sizes $n_1$ and $n_2$, we have
\begin{align}
\label{eq5: 15.07.25}
\sum_{u \in \V{\Gr{G}}} \binom{d_{\Gr{G}}(u)}{2} \leq \binom{n_1}{2} + \binom{n_2}{2}.
\end{align}
Note that the term inside the square root on the right-hand side of \eqref{eq4: 15.07.25} ranges between $2n-3$ and $4n-11 + \tfrac8n$,
thereby yielding an improvement over the term $4n - 3$ that appears under the square root on the right-hand side of \eqref{eq3: 15.07.25}.
It can be verified that inequality \eqref{eq4: 15.07.25} holds with equality if the bipartite and $\CG{4}$-free graph $\Gr{G}$
is a path on either~4 or 5 vertices.}
\end{remark}

\begin{definition}
\label{definition: edge density in bipartite graph}
{\em Let $\Gr{G}$ be a bipartite graph with partite sets of sizes $n_1$ and $n_2$. The {\em edge density} of $\Gr{G}$, denoted by
$\delta(\Gr{G})$ (or simply $\delta$ when the graph is clear from context), is the fraction of present edges in $\Gr{G}$ out of all
possible edges between the two partite sets, i.e.,
\begin{align}
\label{eq: edge density in bipartite graph}
\delta(\Gr{G}) \triangleq \frac{\card{\E{\Gr{G}}}}{n_1 n_2} \in [0,1].
\end{align}}
\end{definition}

\begin{remark}
\label{remark: edge-density of C4-free bipartite graph}
{\em Let $\Gr{G}$ be a $\CG{4}$-free bipartite graph with partite sets of sizes $n_1$ and $n_2$. By \eqref{eq4: 15.07.25},
the edge density of $\Gr{G}$ satisfies
\begin{align}
\label{eq1: 31.07.25}
\delta(\Gr{G}) \leq \frac1{n_1 n_2} \, \biggl\lfloor \tfrac{1}{4} \, (n_1+n_2) \, \biggl(1 + \sqrt{4(n_1+n_2)-3
- \frac{8n_1 n_2}{n_1+n_2}} \; \biggr) \biggr\rfloor.
\end{align}
In particular, if $n_1 = n_2 = \tfrac12 n$, then the edge density of such a $\CG{4}$-free bipartite graph $\Gr{G}$ satisfies
\begin{align}
\label{eq2: 31.07.25}
\delta(\Gr{G}) \leq \frac{1+\sqrt{2n-3}}{n},
\end{align}
showing that it scales at most like the inverse of the square-root of $n$.
This relates to our discussion in Section~\ref{section: numerical results}.}
\end{remark}

\begin{remark}
\label{remark: the LB is loose as alpha tends to 1}
{\em In light of Proposition~\ref{proposition: LB and exact hom(CoBG,G)}
and inequality \eqref{eq2: 31.07.25}, if $p,q\geq 2$, the edge
density of a bipartite graph $\Gr{G}$ on $n$ vertices with
equal-sized partite sets, for which the lower bound on the
right-hand side of \eqref{eq1: 01.07.25} holds with equality,
necessarily tends to zero as $n\to\infty$.
In contrast, if $p=1$ or $q=1$, equality in
\eqref{eq1: 01.07.25} holds for every bipartite graph $\Gr{G}$,
including complete bipartite graphs, which have edge density~$1$.

At the opposite extreme, consider the case where $\Gr{G}$ is
a complete bipartite graph with equal-sized partite sets.
Then, for every integer $p\geq 2$, the exact value of
$\homcount{\CoBG{p}{p}}{\Gr{G}}$ is $2^{1-2p}n^{2p}$,
whereas the lower bound on the right-hand side of
\eqref{eq1: 01.07.25} equals
$2^{1-p}n^{p+1}-\tfrac12 n^2$.
The ratio of the exact value to the lower bound is
\[
\frac{2^{1-2p}n^{2p}}
     {2^{1-p}n^{p+1}-\tfrac12 n^2}
=2^{-p}n^{p-1}+ \tfrac14 + O(n^{1-p}),
\]
which tends to infinity as $n\to\infty$.
This contrasts with the lower bounds derived in
Section~\ref{section: IT bounds}, for which the corresponding
ratio tends, as desired, to~$1$.
This observation also relates to the discussion in
Section~\ref{section: numerical results}.}
\end{remark}

\begin{corollary}
\label{corollary: K_{p,q} --> tree}
{\em Let $\Gr{T}$ be a tree on $n$ vertices, and $p, q \in \naturals$. Then,
\begin{align}
\label{eq1: 01.07.25 - tree}
\homcount{\CoBG{p}{q}}{\Gr{T}} = \sum_{w \in \V{\Gr{T}}} \bigl\{d_{\Gr{T}}(w)^p
+ d_{\Gr{T}}(w)^q \bigr\} - 2(n-1).
\end{align}}
\end{corollary}
\begin{proof}
A tree $\Gr{T}$ on $n$ vertices is a connected and acyclic graph of size $\card{\hspace*{-0.05cm} \E{\Gr{T}}}=n-1$.
It is in particular a $\CG{4}$-free bipartite graph, so it satisfies the necessary and sufficient condition
for equality in~\eqref{eq1: 01.07.25} (see Proposition~\ref{proposition: LB and exact hom(CoBG,G)}). Accordingly,
equality~\eqref{eq1: 01.07.25 - tree} holds as a special case.
\end{proof}

\begin{corollary}
\label{corollary: C4-free and regular bipartite graph}
{\em Let $\Gr{G}$ be a bipartite graph with $n$ vertices and $m$ edges. Then
\begin{align}
\label{eq1: 15.07.25}
\homcount{\CoBG{p}{q}}{\Gr{G}}
\geq 2^p n^{1-p} m^p + 2^q n^{1-q} m^q - 2m.
\end{align}
The equality conditions in \eqref{eq1: 15.07.25} are as follows:
\begin{itemize}
\item If $p=q=1$, equality holds for every bipartite $\Gr{G}$.
\item If exactly one of $p,q$ equals~$1$, equality holds
if and only if $\Gr{G}$ is regular.
\item If $p,q\geq 2$, equality holds if and only if
$\Gr{G}$ is $\CG{4}$-free and regular.
\end{itemize}}
\end{corollary}
\begin{proof}
Applying Jensen's inequality to the right-hand side of \eqref{eq1: 01.07.25} gives
\begin{align}
& \sum_{w \in \V{\Gr{G}}} \bigl\{d_{\Gr{G}}(w)^p + d_{\Gr{G}}(w)^q \bigr\} -2m \nonumber \\
\label{eq2: 15.07.25}
& \geq n \, \biggl( \frac1n \, \sum_{w \in \V{\Gr{G}}} d_{\Gr{G}}(w) \biggr)^p
+ n \, \biggl( \frac1n \, \sum_{w \in \V{\Gr{G}}} d_{\Gr{G}}(w) \biggr)^q -2m \\
&= 2^p n^{1-p} m^p + 2^q n^{1-q} m^q - 2m,  \nonumber
\end{align}
where the last equality follows from the identity $\sum_{w \in \V{\Gr{G}}} d_{\Gr{G}}(w) = 2m$. Inequality~\eqref{eq1: 15.07.25} holds
with equality if and only if both inequalities \eqref{eq1: 01.07.25} and \eqref{eq2: 15.07.25} hold with equality. 
By Proposition~\ref{proposition: LB and exact hom(CoBG,G)}, inequality \eqref{eq1: 01.07.25} holds with equality if and only if 
the bipartite graph $\Gr{G}$ is $\CG{4}$-free or at least one of $p$ and $q$ equals~1.
If $p=q=1$, then \eqref{eq2: 15.07.25} always holds with
equality. Otherwise, strict convexity of $x \mapsto x^r$ on
$[0,\infty)$ for $r>1$ implies that \eqref{eq2: 15.07.25} holds with equality
if and only if $\Gr{G}$ is regular.
Combining these conditions gives the stated characterization of equality in \eqref{eq1: 15.07.25}. 
\end{proof}

\begin{corollary}
\label{corollary: K_{p,q} --> cycle}
{\em Let $p, q, n \in \naturals$ with $n \geq 3$. Then, the number of homomorphisms from the complete bipartite graph
$\CoBG{p}{q}$ to the cycle graph $\CG{n}$ is given by
\begin{align}
\label{eq: hom CoBG to Cn}
\homcount{\CoBG{p}{q}}{\CG{n}} =
\begin{dcases}
n(2^p + 2^q-2),     \quad & \mbox{if $n=3$ or $n \geq 5$,} \\
2^{p+q+1},          \quad & \mbox{if $n=4$.}
\end{dcases}
\end{align}}
\end{corollary}
\begin{proof}
We prove \eqref{eq: hom CoBG to Cn} by distinguishing between three cases.
\begin{enumerate}[(1)]
\item If $n \geq 6$ is even, then the cycle graph $\CG{n}$ is a bipartite, 2-regular, and $\CG{4}$-free graph.
The equality $\homcount{\CoBG{p}{q}}{\CG{n}} = n(2^p + 2^q-2)$ follows from 
Corollary~\ref{corollary: C4-free and regular bipartite graph} with $\Gr{G} = \CG{n}$ and $m=n$.
\item If $n=4$, then it can be verified that the bipartite cycle graph $\CG{4}$ satisfies
\begin{align}
\label{eq: values of N for C4}
N_{1,1}(\CG{4}) = 4, \quad N_{1,2}(\CG{4}) = 2 = N_{2,1}(\CG{4}), \quad N_{2,2}(\CG{4}) = 1, \quad N_{k, \ell}(\CG{4}) = 0, \; \text{otherwise}.
\end{align}
Furthermore, the recursive equation in Definition~\ref{def: Stirling number of 2nd type} yields (see, e.g., \cite[Table~264]{GrahamKP94})
\begin{align}
\label{eq: S(n,2)}
S(n,1)=1, \quad S(n,2) = 2^{n-1}-1, \quad \forall n \in \naturals.
\end{align}
Substituting \eqref{eq: values of N for C4} and \eqref{eq: S(n,2)} into \eqref{eq: exact counting: CoBG --> BG} gives,
after a short algebraic calculation, the equality
\[
\homcount{\CoBG{p}{q}}{\CG{4}} = 2^{p+q+1}.
\]
\item If $n \geq 3$ is odd, then the cycle graph $\CG{n}$ is not bipartite, so Proposition~\ref{prop: exact counting: CoBG --> BG}
does not apply. Let $\set{U} = \{u_1, \ldots, u_p\}$ and $\set{V} = \{v_1, \ldots, v_q\}$ denote the partite sets of $\CoBG{p}{q}$,
of sizes $p$ and $q$, respectively, and let $\V{\CG{n}} = \{c_1, \ldots, c_n\}$ be the vertex set of $\CG{n}$. Consider
an arbitrary homomorphism $\phi \in \Hom{\CoBG{p}{q}}{\CG{n}}$. 

Since $\CG{n}$ has no loops, the sets $\phi(\set{U})$ and $\phi(\set{V})$ are disjoint. Moreover, since $\CG{n}$
is $\CG{4}$-free, at least one of these two sets must be a singleton. Otherwise, choosing two distinct vertices
from each image would give a copy of $\CoBG{2}{2}$ in $\CG{n}$. Thus, two structural configurations for $\phi$ are possible:
(i)~all vertices in $\set{U}$ are mapped to a single vertex $c_\ell \in \V{\CG{n}}$, and each vertex in $\set{V}$ is mapped
independently to one of the two neighbors of $c_\ell$ in $\CG{n}$, or (ii)~all vertices in $\set{V}$ are mapped to a single 
vertex $c_\ell$, and each vertex in $\set{U}$ is mapped independently to one of the two neighbors of $c_\ell$ in $\CG{n}$.

The numbers of homomorphisms satisfying configurations~(i) and (ii) are $n2^q$ and $n2^p$, respectively. Their intersection
consists of the $2n$ homomorphisms mapping each partite set to a single vertex, with the two image vertices adjacent.
Therefore, 
\[
\homcount{\CoBG{p}{q}}{\CG{n}}
= n2^q+n2^p-2n
= n(2^p+2^q-2).
\]
This proves the claimed formula for odd $n \geq 3$. 
\end{enumerate}
\end{proof}

\begin{corollary}
\label{corollary: exact counting: CoBG to a path}
{\em Let $\PathG{\ell}$ be a path graph on $\ell$ vertices, and let $\ell \geq 2$ and $p, q \in \naturals$. Then,
\begin{align}
\label{eq1b: 01.07.25}
\homcount{\CoBG{p}{q}}{\PathG{\ell}} = (\ell-2) \, (2^p + 2^q - 2) + 2.
\end{align}
More generally, if $\Gr{G} \cong \PathG{\ell_1} \times \ldots \times \PathG{\ell_k}$ is isomorphic
to a direct product (a.k.a. categorical or tensor product) of path graphs, where $\ell_j \geq 2$ 
for all $j \in \OneTo{k}$, then
\begin{align}
\label{eq1: 04.07.25}
\homcount{\CoBG{p}{q}}{\Gr{G}} = \prod_{j=1}^{k} \, \Bigl( (\ell_j-2) \, (2^p + 2^q - 2) + 2 \Bigr).
\end{align}}
\end{corollary}
\begin{proof}
Equality~\eqref{eq1b: 01.07.25} follows as a special case of \eqref{eq1: 01.07.25 - tree} since, for $\ell \geq 2$, 
a path graph $\PathG{\ell}$ is a tree on $\ell$ vertices with $\ell-2$ vertices of degree~2, and two vertices of degree~1.
Equality~\eqref{eq1: 04.07.25} relies on \eqref{eq3: homomorphism numbers} and \eqref{eq1b: 01.07.25}.
\end{proof}

\section{Entropy-Based Lower Bounds}
\label{section: IT bounds}

This section relies on properties of the Shannon entropy to derive lower bounds on the number of homomorphisms
from complete bipartite graphs to bipartite graphs, and examines the tightness of these bounds. Familiarity
with Shannon entropy and its basic properties is assumed, following standard notation (see \cite[Chapter~2]{CoverT06}
and the brief overview in Section~\ref{section: preliminaries}).

We begin by deriving a lower bound on the number of homomorphisms from a complete bipartite graph to a given
bipartite graph, based solely on the sizes of the partite sets of the source and target graphs and the edge
density of the target graph. This bound, which is the first of two entropy-based lower bounds derived in
this section, strengthens the inequality implied by the validity of Sidorenko's
conjecture for complete bipartite graphs, and it also improves our earlier result in \cite{Sason25}. We then derive a refined
and strengthened entropy-based lower bound by additionally incorporating the degree profiles of both partite sets of
the target bipartite graph $\Gr{G}$. It is worth noting that the first entropy-based bound (see Proposition~\ref{prop.: IT-LB})
serves not only as a pedagogical precursor to the refined and improved lower bound on the number of graph homomorphisms (see
Proposition~\ref{prop.: refined IT-LB}), but also holds universally for all bipartite graphs $\Gr{G}$ with prescribed
partite sets and edge density, irrespective of the degree profiles within the partite sets of $\Gr{G}$.

\subsection{Lower bound on the number of homomorphisms}
\label{subsection: Lower bound on the number of homomorphisms}

\begin{proposition}
\label{prop.: IT-LB}
{\em Let $\Gr{G}$ be a simple bipartite graph with partite sets of sizes $n_1$ and $n_2$, and let
$\delta = \delta(\Gr{G})$ be the edge density of $\Gr{G}$. Then, for all positive integers $p, q \in \naturals$,
\begin{align}
\label{eq4b: 16.09.2024}
\homcount{\CoBG{p}{q}}{\Gr{G}} & \geq \delta^{pq} \bigl(n_1^p n_2^q + n_1^q n_2^p \bigr) \\
\label{eq4c: 16.09.2024}
&= \delta^{pq} \, \homcount{\CoBG{p}{q}}{\CoBG{n_1}{n_2}},
\end{align}
with equality in \eqref{eq4b: 16.09.2024} attained at both $\delta=0$ and $\delta=1$, corresponding to the
cases where $\Gr{G}$ is an empty graph and $\Gr{G} = \CoBG{n_1}{n_2}$, respectively.}
\end{proposition}

\begin{proof}
If $\delta=0$, then $\Gr{G}$ is an empty (edgeless) bipartite graph, so $\homcount{\CoBG{p}{q}}{\Gr{G}} = 0$,
and \eqref{eq4b: 16.09.2024} trivially holds with equality. Suppose $\delta \in (0,1]$.

Let $\set{U}$ and $\set{V}$ denote the partite vertex sets of the bipartite graph $\Gr{G}$, where
$\card{\set{U}} = n_1$ and $\card{\set{V}} = n_2$.
Let $(U,V)$ be a random vector taking values in $\set{U} \times \set{V}$, and assume that $\{U,V\}$ is
distributed uniformly at random over the edge set of $\Gr{G}$. Then, by the condition for equality in
\eqref{eq: uniform bound}, the joint entropy of $(U,V)$ is given by
\begin{align}
\Ent{U,V} &= \log \, \bigcard{\E{\Gr{G}}} \nonumber \\
&=\log(\delta n_1 n_2).  \label{eq: Ent U,V}
\end{align}

Let $\pmfOf{U}$ denote the marginal probability mass function (PMF) of $U$, and let $\CondpmfOf{V}{U}$ denote
the conditional PMF of $V$ given $U$. Define the random vectors ${\bf{U}}^p \triangleq (U_1, \ldots, U_p)$ and
${\bf{V}}^q \triangleq (V_1, \ldots, V_q)$, and let $\underline{u} = (u_1, \ldots, u_p)$
and $\underline{v} = (v_1, \ldots, v_q)$ denote their respective realizations, where $p, q \in \naturals$.

Throughout this section, formulas expressing conditional PMFs as ratios apply
only when the conditioning event has positive probability.
On conditioning events of probability zero, these ratios are undefined.
Whenever conditional-PMF notation is used on such events, it denotes
an arbitrarily chosen PMF on the corresponding alphabet.

In particular, since
$\pmfOf{U}(u)=d_{\Gr{G}}(u)/\card{\E{\Gr{G}}}$,
the zero-probability values of $U$ are precisely the isolated
vertices in $\set{U}$. For each such vertex $u$, extend the notation
$\CondpmfOf{V}{U}(\cdot \midnew u)$ by assigning an arbitrary PMF
on $\set{V}$.

For each tuple $\underline{v}$ such that
$\pmfOf{{\bf V}^q}(\underline{v})=0$, extend the notation
$\CondpmfOf{U_i}{{\bf V}^q}(\cdot \midnew \underline{v})$
in \eqref{eq2b: cond. PMF} by assigning an arbitrary PMF
on $\set{U}$, identical for all $i \in \OneTo{p}$.
For each such tuple, define
$\CondpmfOf{{\bf U}^p}{{\bf V}^q}(\cdot \midnew \underline{v})$
as the product of the assigned PMFs, as in
\eqref{eq2a: cond. PMF}.
These extensions do not affect the joint distributions or entropies.

We now state two lemmas that will be used in the sequel for the proof of Proposition~\ref{prop.: IT-LB}.
\begin{lemma}
\label{lemma 1: 08.08.25}
{\em Let $({\bf U}^p, {\bf V}^q)$ be constructed as follows:
\begin{enumerate}[a)]
\item $\{V_j\}_{j=1}^q$ are conditionally independent and identically distributed (i.i.d.) given $U$, with
\begin{align}
\label{eq1: cond. PMF}
\CondpmfOf{{\bf V}^q}{U}(\underline{v} \midnew u)
= \prod_{j=1}^q \CondpmfOf{V}{U}(v_j \midnew u),
\quad u \in \set{U}, \; \underline{v} \in \set{V}^q.
\end{align}
\item $\{U_i\}_{i=1}^p$ are conditionally i.i.d. given ${\bf V}^q$, with
\begin{align}
\label{eq2a: cond. PMF}
\CondpmfOf{{\bf U}^p}{{\bf V}^q}(\underline{u} \midnew \underline{v})
= \prod_{i=1}^p \CondpmfOf{U_i}{{\bf V}^q}(u_i \midnew \underline{v}),
\end{align}
where each of the $p$ identical conditional PMFs on the right-hand side of \eqref{eq2a: cond. PMF} is given by
\begin{align}
\label{eq2b: cond. PMF}
\CondpmfOf{U_i}{{\bf V}^q}(u \midnew \underline{v})
= \frac{\pmfOf{U}(u) \, \overset{q}{\underset{j=1}{\prod}} \CondpmfOf{V}{U}(v_j \midnew u)}
{\underset{u' \in \set{U}}{\sum} \biggl\{ \pmfOf{U}(u') \, \overset{q}{\underset{j=1}{\prod}} \CondpmfOf{V}{U}(v_j \midnew u') \biggr\} },
\quad u \in \set{U}, \; \underline{v} \in \set{V}^q, \; i \in \OneTo{p}.
\end{align}
\end{enumerate}
Then,
\begin{enumerate}[i)]
\item $U_i \sim U$ for all $i \in \OneTo{p}$.
\item $(U_i,{\bf V}^q) \sim (U,{\bf V}^q)$ and $(U_i,V_j) \sim (U,V)$ for all $i \in \OneTo{p}$, $j \in \OneTo{q}$.
\end{enumerate}}
\end{lemma}
\begin{proof}
From \eqref{eq1: cond. PMF},
\begin{align}
\label{eq2c: cond. PMF}
\pmfOf{{\bf V}^q}(\underline{v})
= \sum_{u \in \set{U}} \biggl\{ \pmfOf{U}(u) \, \prod_{j=1}^q \CondpmfOf{V}{U}(v_j \midnew u) \biggr\}.
\end{align}
For every $i \in \OneTo{p}$, $u \in \set{U}$, and
$\underline{v} \in \set{V}^q$, we have
\[
\CondpmfOf{U_i}{{\bf V}^q}(u \midnew \underline{v})
\, \pmfOf{{\bf V}^q}(\underline{v})
= \pmfOf{U}(u) \, \prod_{j=1}^q
\CondpmfOf{V}{U}(v_j \midnew u).
\]
For positive-probability tuples, this follows from
\eqref{eq2b: cond. PMF} and \eqref{eq2c: cond. PMF}.
For null tuples, both sides are zero, since every nonnegative
summand in \eqref{eq2c: cond. PMF} is zero.
Hence, for $i \in \OneTo{p}$,
\begin{align}
\pmfOf{U_i}(u)
&= \sum_{\underline{v} \in \set{V}^q} \biggl\{ \CondpmfOf{U_i}{{\bf V}^q}(u \midnew \underline{v}) \, \pmfOf{{\bf V}^q}(\underline{v}) \biggr\} \label{eq1: 17.3.25} \\
&= \sum_{\underline{v} \in \set{V}^q} \biggl\{ \pmfOf{U}(u) \, \prod_{j=1}^q \CondpmfOf{V}{U}(v_j \midnew u) \biggr\} \label{eq2: 17.3.25} \\
&= \pmfOf{U}(u) \, \prod_{j=1}^q \sum_{v_j \in \set{V}} \CondpmfOf{V}{U}(v_j \midnew u) \label{eq3: 17.3.25} \\
&= \pmfOf{U}(u),  \label{eq4: 17.3.25}
\end{align}
where \eqref{eq1: 17.3.25} holds by the law of total probability,
\eqref{eq2: 17.3.25} follows from the preceding product identity,
\eqref{eq3: 17.3.25} follows from the factorization of the summand, and \eqref{eq4: 17.3.25}
holds since the conditional probability masses in each inner summation equal~1.
This proves~(i). Moreover,
for all $i \in \OneTo{p}$,
\begin{align}
\label{eq1: 25.6.25}
\pmfOf{U_i,{\bf V}^q}(u,\underline{v})
&= \pmfOf{U}(u) \, \prod_{j=1}^q \CondpmfOf{V}{U}(v_j \midnew u) \\
\label{eq5: 17.3.25}
&= \pmfOf{U,{\bf V}^q}(u,\underline{v}),
\end{align}
where \eqref{eq1: 25.6.25} follows from the same product identity,
and \eqref{eq5: 17.3.25} holds by \eqref{eq1: cond. PMF}.
Hence, $(U_i,{\bf V}^q) \sim (U,{\bf V}^q)$, and marginalization over $\{v_k\}_{k \in \OneTo{q} \setminus \{j\}}$ yields
\begin{align}
\pmfOf{U_i,V_j}(u,v) = \pmfOf{U,V}(u,v),
\end{align}
which proves~(ii).
\end{proof}

\begin{lemma}
\label{lemma 2: 08.08.25}
{\em In the setting of Lemma~\ref{lemma 1: 08.08.25}, the joint entropies satisfy
\begin{align}
\label{eq1: 09.08.25}
& \Ent{U_1, {\bf{V}}^q} \geq \log(\delta^q n_1 n_2^q), \\
\label{eq2: 09.08.25}
& \Ent{{\bf{U}}^p, {\bf{V}}^q} \geq \log(\delta^{pq} n_1^p n_2^q).
\end{align}}
\end{lemma}
\begin{proof}
The joint entropy of the random subvector $(U_1, {\bf{V}}^q)$ satisfies
\begin{align}
\Ent{U_1, {\bf{V}}^q} &= \Ent{U, {\bf{V}}^q} \label{eq0d: 16.03.2025} \\
&=\Ent{U} + \sum_{j=1}^q \EntCond{V_j}{U} \label{eq0: 16.03.2025} \\
&= \Ent{U} + q \EntCond{V}{U} \label{eq0a: 16.03.2025} \\
&= q \Ent{U,V} - (q-1) \Ent{U} \label{eq0b: 16.03.2025} \\
&= q \log(\delta n_1 n_2) - (q-1) \Ent{U} \label{eq0c: 16.03.2025} \\
&\geq q \log(\delta n_1 n_2) - (q-1) \log n_1 \label{eq1: 16.03.2025} \\
&= \log(\delta^q n_1 n_2^q),  \label{eq5: 16.09.2024}
\end{align}
where \eqref{eq0d: 16.03.2025} holds since
$(U_1, {\bf{V}}^q) \sim (U, {\bf{V}}^q)$ by Item~(ii) of Lemma~\ref{lemma 1: 08.08.25};
\eqref{eq0: 16.03.2025} follows from the chain rule in \eqref{eq: chain rule} and the fact
that, by \eqref{eq1: cond. PMF}, the entries of ${\bf{V}}^q$ are conditionally independent
given $U$; \eqref{eq0a: 16.03.2025} holds since $(U, V_j) \sim (U,V)$ by \eqref{eq1: cond. PMF};
\eqref{eq0b: 16.03.2025} is a second application of \eqref{eq: chain rule}; \eqref{eq0c: 16.03.2025}
holds by \eqref{eq: Ent U,V}, and \eqref{eq1: 16.03.2025} follows from the uniform bound in
\eqref{eq: uniform bound}, which yields $\Ent{U} \leq \log \card{\set{U}} = \log n_1$.
Consequently, the joint entropy of $({\bf{U}}^p, {\bf{V}}^q)$ satisfies
\begin{align}
\hspace*{-0.4cm} \Ent{{\bf{U}}^p, {\bf{V}}^q}
&= \Ent{{\bf{V}}^q} + \sum_{i=1}^p \EntCond{U_i}{{\bf{V}}^q} \label{eq2a: 16.03.2025} \\
&= \Ent{{\bf{V}}^q} + p \EntCond{U_1}{{\bf{V}}^q} \label{eq2b: 16.03.2025} \\
&= p \Ent{U_1, {\bf{V}}^q} - (p-1) \Ent{{\bf{V}}^q} \label{eq2c: 16.03.2025} \\
&\geq p \log(\delta^q n_1 n_2^q) - (p-1) \log(n_2^q) \label{eq3: 16.03.2025} \\
&= \log(\delta^{pq} n_1^p n_2^q),  \label{eq6: 16.09.2024}
\end{align}
where \eqref{eq2a: 16.03.2025} holds by the chain rule in \eqref{eq: chain rule} and since, by
\eqref{eq2a: cond. PMF}, the random variables $U_1, \ldots, U_p$ are conditionally independent given ${\bf{V}}^q$;
\eqref{eq2b: 16.03.2025} holds since, by \eqref{eq2b: cond. PMF}, all the $U_i$'s ($i \in \OneTo{p}$)
are identically distributed given ${\bf{V}}^q$;
\eqref{eq2c: 16.03.2025} holds by another use of the chain rule; finally, \eqref{eq3: 16.03.2025} holds
by \eqref{eq0: 16.03.2025}--\eqref{eq5: 16.09.2024} and the uniform bound in \eqref{eq: uniform bound},
which implies that $\Ent{{\bf{V}}^q} \leq \log(\card{\set{V}}^{\, q}) = \log(n_2^q)$.
\end{proof}

We proceed with the proof of Proposition~\ref{prop.: IT-LB}.
Each $(\underline{u}, \underline{v})$ in the support of the random vector
$({\bf{U}}^p, {\bf{V}}^q)$ in Lemma~\ref{lemma 1: 08.08.25}
determines a unique homomorphism
$\phi \in \Hom{\CoBG{p}{q}}{\Gr{G}}$. To this end, label the vertices of the complete bipartite graph $\CoBG{p}{q}$ by the
elements of the set $\OneTo{p+q}$, assigning the labels $1, \ldots, p$ to the vertices in one partite set of size $p$, and
the labels $p+1, \ldots, p+q$ to those in the second partite set of size $q$.
For every $i \in \OneTo{p}$, map vertex $i \in \V{\CoBG{p}{q}}$ to vertex $u_i \in \set{U}$ from one partite set of $\Gr{G}$.
Likewise, for every $j \in \OneTo{q}$, map vertex $p+j \in \V{\CoBG{p}{q}}$ to vertex $v_j \in \set{V}$ from the second partite set
of $\Gr{G}$.
By the construction of $(U,V)$ and Item~(ii) of
Lemma~\ref{lemma 1: 08.08.25}, we have
$\{U_i,V_j\} \in \E{\Gr{G}}$ almost surely for each
$i \in \OneTo{p}$ and $j \in \OneTo{q}$.
Since there are finitely many such pairs, all these edge conditions
hold simultaneously almost surely, and hence for every tuple in the support.
Thus, the mapping $\phi \colon \V{\CoBG{p}{q}} \to \V{\Gr{G}}$
maps every edge $\{i,p+j\}$ to an edge $\{u_i,v_j\}$ of $\Gr{G}$,
and therefore defines a homomorphism in $\Hom{\CoBG{p}{q}}{\Gr{G}}$. Let $\set{H}_1$ be the subset of $\Hom{\CoBG{p}{q}}{\Gr{G}}$
in which the vertices in the partite sets of sizes $p$ and $q$ of $\CoBG{p}{q}$ are mapped, respectively, to vertices in the
partite sets of sizes $n_1$ and $n_2$ of $\Gr{G}$.
The suggested correspondence is injective since distinct realizations $(\underline{u}, \underline{v})$ of the random vector
$({\bf{U}}^p, {\bf{V}}^q)$ yield distinct homomorphisms in $\set{H}_1 \subseteq \Hom{\CoBG{p}{q}}{\Gr{G}}$.
By the uniform bound in \eqref{eq: uniform bound}, it follows that
\begin{align}
\label{eq7: 16.09.2024}
\Ent{{\bf{U}}^p, {\bf{V}}^q} \leq \log \card{\set{H}_1}.
\end{align}
Combining inequalities \eqref{eq2: 09.08.25} and \eqref{eq7: 16.09.2024} yields
\begin{align}
\label{eq8: 16.09.2024}
\card{\set{H}_1} \geq \delta^{pq} n_1^p n_2^q.
\end{align}
Likewise, let $\set{H}_2$ be the subset of $\Hom{\CoBG{p}{q}}{\Gr{G}}$ in which the vertices in the partite sets of sizes $p$ and $q$
of $\CoBG{p}{q}$ are mapped, respectively, to vertices in the partite sets of sizes $n_2$ and $n_1$ of $\Gr{G}$.
Analogously to \eqref{eq8: 16.09.2024}, interchanging $n_1$ and $n_2$ yields
\begin{align}
\label{eq8b: 16.09.2024}
\card{\set{H}_2} \geq \delta^{pq} n_1^q n_2^p.
\end{align}
The subsets $\set{H}_1$ and $\set{H}_2$ of
$\Hom{\CoBG{p}{q}}{\Gr{G}}$ are, by definition, disjoint.
Since $p,q \geq 1$, the graph $\CoBG{p}{q}$ is connected.
Every homomorphism from a connected bipartite graph to a bipartite graph
either preserves or interchanges the two partite sets.
Hence, $\set{H}_1$ and $\set{H}_2$ form a partition of
$\Hom{\CoBG{p}{q}}{\Gr{G}}$.
Consequently, it follows from \eqref{eq8: 16.09.2024} and
\eqref{eq8b: 16.09.2024} that
\begin{align}
\homcount{\CoBG{p}{q}}{\Gr{G}} &= \card{\set{H}_1} + \card{\set{H}_2} \label{eq8c: 16.09.2024} \\
&\geq \delta^{pq} \bigl(n_1^p n_2^q + n_1^q n_2^p \bigr), \label{eq9: 16.09.2024}
\end{align}
which proves inequality \eqref{eq4b: 16.09.2024}.
Equality \eqref{eq4c: 16.09.2024} holds by the identity
\begin{align}
\label{eq1: 29.07.25}
\homcount{\CoBG{p}{q}}{\CoBG{n_1}{n_2}} = n_1^p n_2^q + n_1^q n_2^p.
\end{align}
Finally, inequality \eqref{eq4b: 16.09.2024} holds with equality at $\delta=1$, yielding in this case that $\Gr{G} = \CoBG{n_1}{n_2}$.
\end{proof}

\begin{corollary}
\label{corollary: homomorphism densities}
{\em Let $\Gr{G}$ be a simple bipartite graph with partite sets of sizes $n_1$ and $n_2$,
and $\delta n_1 n_2$ edges for some $\delta \in (0,1]$. Then, for all $p,q, n_1, n_2 \in \naturals$,
\begin{align}
\label{eq3: 16.06.2025}
t(\CoBG{p}{q}, \Gr{G}) \geq \delta^{pq} \; t(\CoBG{p}{q},\CoBG{n_1}{n_2}).
\end{align}
In other words, the $\CoBG{p}{q}$-homomorphism density in a bipartite graph $\Gr{G}$, with an edge density $\delta \in [0,1]$,
is at least $\delta^{\, \card{\E{\CoBG{p}{q}}}}$ times the $\CoBG{p}{q}$-homomorphism density in the
complete bipartite graph with the same partite vertex sets as $\Gr{G}$.
In particular, inequality \eqref{eq3: 16.06.2025} holds with equality at $\delta=1$.}
\end{corollary}
\begin{proof}
By \eqref{eq4c: 16.09.2024},
\begin{align}
\homcount{\CoBG{p}{q}}{\Gr{G}} \geq \delta^{pq} \, \homcount{\CoBG{p}{q}}{\CoBG{n_1}{n_2}},
\end{align}
and $\V{\Gr{G}} = \V{\CoBG{n_1}{n_2}}$ since, by assumption, $\Gr{G}$ is a spanning subgraph of the complete bipartite
graph $\CoBG{n_1}{n_2}$. The result in \eqref{eq3: 16.06.2025} then follows from Definition~\ref{definition: Homomorphism densities}.
\end{proof}

\begin{definition}[Sidorenko graph]
\label{definition: Sidorenko graph}
{\em A graph $\Gr{H}$ is said to be Sidorenko if, for every graph $\Gr{G}$,
\begin{align}
\label{eq2: 16.06.2025}
t(\Gr{H}, \Gr{G}) \geq t(\CoG{2}, \Gr{G})^{\, e(\Gr{H})},
\end{align}
where $e(\Gr{H}) \triangleq \card{ \hspace*{-0.05cm} \E{\Gr{H}}}$. Equivalently,
by \eqref{eq: 16.06.2025} and \eqref{eq: edges},
\begin{align}
\label{def: Sidorenko graph}
\frac{\homcount{\Gr{H}}{\Gr{G}}}{v(\Gr{G})^{\; v(\Gr{H})}} \geq
\biggl( \frac{2 \, e(\Gr{G})}{v(\Gr{G})^2} \biggr)^{\, e(\Gr{H})}.
\end{align}
In words, a graph $\Gr{H}$ is Sidorenko if the probability that a random uniform mapping from
$\V{\Gr{H}}$ to the vertex set of any graph $\Gr{G}$ forms a homomorphism is at least the product,
over all edges in $\Gr{H}$, of the probabilities that these edges are mapped to edges of $\Gr{G}$.
Inequality~\eqref{eq2: 16.06.2025} is referred to as Sidorenko's lower bound for homomorphism densities.}
\end{definition}

Sidorenko's conjecture states that every bipartite graph is Sidorenko \cite{Sidorenko93}. Related forms of this conjecture
appeared earlier in a work by Erd\H{o}s and Simonovits~\cite{Simonovits84}. While the conjecture remains an open problem
in its full generality, it is known that every bipartite graph containing a vertex adjacent to all vertices in its other
part is Sidorenko (see, e.g., \cite[Theorem~5.5.14]{Zhao23}, originally proved in \cite{ConlonFS10}, and simplified in
\cite{ConlonFS10b}). Additional classes of bipartite graphs that are Sidorenko have been established in \cite{ConlonKLL18}.

\begin{discussion}[Comparison to Sidorenko's lower bound]
\label{discussion: Comparison of the 1st IT LB to Sidorenko's lower bound}
{\em Every complete bipartite graph is known to be Sidorenko (see \cite[Theorem~5.5.12]{Zhao23}).
Specializing \eqref{def: Sidorenko graph} to a complete bipartite graph $\Gr{H} = \CoBG{p}{q}$, where $p,q \in \naturals$,
yields inequality \eqref{def: Sidorenko graph} with $v(\Gr{H}) = p+q$ and $e(\Gr{H}) = pq$.
Let us now further specialize it to the case where $\Gr{G}$ is a simple bipartite graph with partite sets of sizes
$n_1$ and $n_2$, has no isolated vertices, and contains $\delta n_1 n_2$ edges for some $\delta \in (0,1]$.
In this specialized setting, \eqref{def: Sidorenko graph} gives
\begin{align}
\label{eq2: 21.05.2025}
\homcount{\CoBG{p}{q}}{\Gr{G}} \geq (2 \delta)^{pq} (n_1+n_2)^{p+q-2pq} (n_1 n_2)^{pq} \triangleq \mathrm{LB}_1.
\end{align}
This lower bound on $\homcount{\CoBG{p}{q}}{\Gr{G}}$ is compared to the bound $\mathrm{LB}_2 \triangleq \delta^{pq} \, \bigl(n_1^p n_2^q + n_1^q n_2^p)$,
which appears as the right-hand side of \eqref{eq4b: 16.09.2024}.
To compare these two lower bounds, which are symmetric in $n_1$ and $n_2$ and also in $p$ and $q$, we examine the ratio $\frac{\mathrm{LB}_2}{\mathrm{LB}_1}$.
Without loss of generality, assume that $p \geq q$, and let $r \triangleq \frac{\max\{n_1,n_2\}}{\min\{n_1,n_2\}} \geq 1$.
By straightforward algebra, we get
\begin{align}
\label{eq5: 21.05.2025}
\frac{\mathrm{LB}_2}{\mathrm{LB}_1} &= 2^{-p} \, \Biggl( \frac{(1+r)^2}{2 r} \Biggr)^{p(q-1)} \; (1+r)^{p-q} \; \Bigl(1 + r^{-(p-q)} \Bigr) \\
&\geq 2^{-p} \, 2^{p(q-1)} \, 2^{p-q} \; \bigl(1 + r^{-(p-q)} \bigr) \nonumber \\
\label{eq6: 21.05.2025}
&= 2^{pq-(p+q)} \; \bigl(1 + r^{-|p-q|} \bigr).
\end{align}
By the symmetry of the right-hand side of \eqref{eq6: 21.05.2025} in $p$ and $q$, the earlier assumption that $p \geq q$ can be dropped. Consequently, the following cases hold:
\begin{enumerate}[(1)]
\item If $p=q$, then it follows from \eqref{eq6: 21.05.2025} that $\mathrm{LB}_2 \geq 2^{(p-1)^2} \, \mathrm{LB}_1$, and in particular, $\mathrm{LB}_2 \geq \mathrm{LB}_1$.
\item Else, if $p>1$ and $q=1$ (i.e., $\CoBG{p}{q}$ is a star graph), then by Jensen's inequality
\begin{align}
\frac{\mathrm{LB}_2}{\mathrm{LB}_1} = \frac{\tfrac12 \bigl(n_1^{1-p} + n_2^{1-p}\bigr)}{\Bigl(\frac{n_1+n_2}{2}\Bigr)^{1-p}} \geq 1,
\end{align}
so $\mathrm{LB}_2 \geq \mathrm{LB}_1$. Due to symmetry in $p$ and $q$, it also holds if $p=1$ and $q>1$.
\item Otherwise (i.e., if $p,q \geq 2$ and $p \neq q$), we get from \eqref{eq6: 21.05.2025} that $\mathrm{LB}_2 > 2^{pq-(p+q)} \, \mathrm{LB}_1 \geq 2 \, \mathrm{LB}_1$.
\end{enumerate}
To conclude, our lower bound on $\homcount{\CoBG{p}{q}}{\Gr{G}}$ on the right-hand side of \eqref{eq4b: 16.09.2024}
compares favorably to Sidorenko's lower bound given in \eqref{eq2: 21.05.2025}. Equivalently, inequality \eqref{eq3: 16.06.2025}
compares favorably to Sidorenko's lower bound expressed in terms of homomorphism densities (see \eqref{eq2: 16.06.2025} with
$\Gr{H} = \CoBG{p}{q}$).}
\end{discussion}

\subsection{A refined and strengthened entropy-based lower bound}
\label{subsection: refined entropy-based bound}
We next present a refinement of the lower bound on $\homcount{\CoBG{p}{q}}{\Gr{G}}$
in Proposition~\ref{prop.: IT-LB}. To this end, we define the degree profile (or degree
sequence) of the bipartite graph $\Gr{G}$ with partite sets $\set{U}$ and $\set{V}$ of sizes
$n_1$ and $n_2$, respectively, as the vectors
\begin{align}
\label{eq1: degree profile U}
& {\bf{d}}^{(\set{U})}(\Gr{G}) = \Bigl(d_1^{(\set{U})}, \ldots, d_{n_1}^{(\set{U})} \Bigr), \\
\label{eq1: degree profile V}
& {\bf{d}}^{(\set{V})}(\Gr{G}) = \Bigl(d_1^{(\set{V})}, \ldots, d_{n_2}^{(\set{V})} \Bigr),
\end{align}
where $d_i^{(\set{U})}$ and $d_j^{(\set{V})}$, for all $i \in \OneTo{n_1}$ and $j \in \OneTo{n_2}$, denote, respectively,
the degrees of the $i$-th vertex in $\set{U}$ and the $j$-th vertex in $\set{V}$.
The normalized degree profiles per edge in $\Gr{G}$ are given by
\begin{align}
\overline{\bf{d}}^{(\set{U})}(\Gr{G}) &= \Bigl(\overline{d}_1^{(\set{U})}, \ldots, \overline{d}_{n_1}^{(\set{U})} \Bigr)
\label{eq1: normalized degree profile U}
\triangleq \biggl(\frac{d_1^{(\set{U})}}{\card{\E{\Gr{G}}}}, \ldots, \frac{d_{n_1}^{(\set{U})}}{\card{\E{\Gr{G}}}} \biggr), \\[0.1cm]
\overline{\bf{d}}^{(\set{V})}(\Gr{G}) &= \Bigl(\overline{d}_1^{(\set{V})}, \ldots, \overline{d}_{n_2}^{(\set{V})} \Bigr)
\label{eq1: normalized degree profile V}
\triangleq \biggl(\frac{d_1^{(\set{V})}}{\card{\E{\Gr{G}}}}, \ldots, \frac{d_{n_2}^{(\set{V})}}{\card{\E{\Gr{G}}}} \biggr),
\end{align}
so $\overline{\bf{d}}^{(\set{U})}(\Gr{G})$ and $\overline{\bf{d}}^{(\set{V})}(\Gr{G})$ form probability vectors since their entries are
nonnegative and satisfy
\begin{align}
\label{eq2: 25.6.2025}
\sum_{k=1}^{n_1} \overline{d}_k^{(\set{U})} = 1 = \sum_{k=1}^{n_2} \overline{d}_k^{(\set{V})}.
\end{align}

The following refined lower bound on the number of homomorphisms from a complete bipartite graph to a simple bipartite graph
improves upon the earlier bound in Proposition~\ref{prop.: IT-LB}; see Section~\ref{section: numerical results}.
\begin{proposition}
\label{prop.: refined IT-LB}
{\em Let $\Gr{G}$ be a simple bipartite graph with partite sets
$\set{U}$ and $\set{V}$ of sizes $n_1$ and $n_2$, respectively.
Suppose that $\E{\Gr{G}}\neq\varnothing$, and let $\delta$ be
the edge density defined in
\eqref{eq: edge density in bipartite graph}.
Let $\overline{\bf{d}}^{(\set{U})}(\Gr{G})$ and
$\overline{\bf{d}}^{(\set{V})}(\Gr{G})$ denote the normalized
degree profiles of $\Gr{G}$ on the partite sets $\set{U}$ and
$\set{V}$, respectively, as defined in
\eqref{eq1: normalized degree profile U} and
\eqref{eq1: normalized degree profile V}.
Then, for all $p,q\in\naturals$,
\begin{align}
\homcount{\CoBG{p}{q}}{\Gr{G}}
&\geq \max \Bigl\{
(\delta n_1 n_2)^{pq}
\, \exp \bigl(-p(q-1)x-q(p-1)y\bigr), \nonumber\\
&\hspace*{1.3cm}
(\delta n_1 n_2)^q \, \exp \bigl(-(q-1)x\bigr), \;
(\delta n_1 n_2)^p \, \exp \bigl(-(p-1)y\bigr)
\Bigr\} \nonumber\\
&\quad + \max \Bigl\{
(\delta n_1 n_2)^{pq}
\, \exp \bigl(-q(p-1)x-p(q-1)y\bigr), \nonumber\\
\label{eq3: 16.07.2025}
&\hspace*{1.3cm}
(\delta n_1 n_2)^p \, \exp \bigl(-(p-1)x\bigr), \;
(\delta n_1 n_2)^q \, \exp \bigl(-(q-1)y\bigr)
\Bigr\},
\end{align}
where
\begin{align}
\label{eq1: 5.10.26}
x &\triangleq
-\sum_{k=1}^{n_1}
\overline{d}_k^{(\set{U})}
\, \log \overline{d}_k^{(\set{U})},
\qquad
y \triangleq
-\sum_{k=1}^{n_2}
\overline{d}_k^{(\set{V})}
\, \log \overline{d}_k^{(\set{V})}.
\end{align}
Here, $0\log 0=0$ by continuous extension, so the definitions
of $x$ and $y$ allow isolated vertices.
If $\E{\Gr{G}}=\varnothing$, then 
$\homcount{\CoBG{p}{q}}{\Gr{G}}=0$ for all $p,q\in\naturals$.}
\end{proposition}
\begin{proof}
By revisiting the proof of Proposition~\ref{prop.: IT-LB}, let $\{U,V\}$ be uniformly distributed over the edges of $\Gr{G}$,
where $U \in \set{U}$ and $V \in \set{V}$.
It then follows that the marginal PMFs of the random variables $U$ and $V$ are equal to the normalized degree profiles
$\overline{\bf{d}}^{(\set{U})}(\Gr{G})$ and $\overline{\bf{d}}^{(\set{V})}(\Gr{G})$, respectively. Hence, by \eqref{eq1: 5.10.26},
the entropies of $U$ and $V$ satisfy 
\begin{align}
\label{eq: Ent U}
& \Ent{U} = x, \\
\label{eq: Ent V}
& \Ent{V} = y.
\end{align}
By \eqref{eq0d: 16.03.2025}--\eqref{eq0c: 16.03.2025},
whose derivation uses the conditional independence in
\eqref{eq1: cond. PMF}, we obtain
\begin{align}
\label{eq2: 22.05.2025}
\Ent{U_1, V_1, \ldots, V_q} &= q \log(\delta n_1 n_2) - (q-1) \Ent{U},
\end{align}
which replaces inequality \eqref{eq5: 16.09.2024} by equality \eqref{eq2: 22.05.2025},
with $\Ent{U}$ as given in \eqref{eq: Ent U}. Then,
\begin{align}
& \hspace*{-0.3cm} \Ent{{\bf{U}}^p, {\bf{V}}^q} \nonumber \\[0.1cm]
\label{eq4: 22.05.2025}
&= p \Ent{U_1, {\bf{V}}^q} - (p-1) \Ent{{\bf{V}}^q} \\[0.1cm]
\label{eq5: 22.05.2025}
&= p \bigl[ q \log(\delta n_1 n_2) - (q-1) \Ent{U} \bigr] - (p-1) \Ent{{\bf{V}}^q} \\[0.1cm]
\label{eq6: 22.05.2025}
&\geq pq \log(\delta n_1 n_2) - p(q-1) \, \Ent{U} - (p-1) \, \min \bigl\{ q \Ent{V}, \, \Ent{U_1, {\bf{V}}^q} \bigr\} \\[0.1cm]
&= \max \Bigl\{ pq \log(\delta n_1 n_2) - p(q-1) \Ent{U} - (p-1)q \Ent{V}, \nonumber \\
\label{eq1a: 26.06.2025}
& \hspace*{1.3cm} pq \log(\delta n_1 n_2) - p(q-1) \Ent{U} - (p-1) \, \Ent{U_1, {\bf{V}}^q} \Bigr\} \\[0.1cm]
&= \max \Bigl\{ pq \log(\delta n_1 n_2) - p(q-1) \Ent{U} - (p-1)q \Ent{V}, \nonumber \\
\label{eq1b: 26.06.2025}
& \hspace*{1.3cm} pq \log(\delta n_1 n_2) - p(q-1) \Ent{U} - (p-1) \, \bigl[q \log(\delta n_1 n_2) - (q-1) \, \Ent{U} \bigr] \Bigr\} \\[0.1cm]
\label{eq1c: 26.06.2025}
&= \max \Bigl\{ pq \log(\delta n_1 n_2) - p(q-1) \Ent{U} - (p-1)q \Ent{V}, \;
q \log(\delta n_1 n_2) - (q-1) \, \Ent{U} \Bigr\}.
\end{align}
The justifications for \eqref{eq4: 22.05.2025}--\eqref{eq1c: 26.06.2025} are given as follows:
\begin{itemize}
\item \eqref{eq4: 22.05.2025} holds by \eqref{eq2c: 16.03.2025};
\item \eqref{eq5: 22.05.2025} is due to \eqref{eq0c: 16.03.2025};
\item \eqref{eq6: 22.05.2025} holds by the subadditivity of the Shannon entropy and since
$V_j \sim V$ for all $j \in \OneTo{q}$,
implying
\[
\Ent{{\bf{V}}^q} \leq q \Ent{V},
\]
and also due to the chain rule and nonnegativity of the Shannon entropy for discrete random variables,
implying
\[
\Ent{{\bf{V}}^q} \leq \Ent{U_1, {\bf{V}}^q}.
\]
\item Equality \eqref{eq1a: 26.06.2025} follows by the identity $x - \min\{y,z\} = \max\{x-y, x-z\}$ for all $x,y,z \in \mathbb{R}$;
\item Equality \eqref{eq1b: 26.06.2025} holds by \eqref{eq0c: 16.03.2025};
\item Equality \eqref{eq1c: 26.06.2025} follows by simplifying the second term on the right-hand side of \eqref{eq1b: 26.06.2025}.
\end{itemize}
Consequently, combining \eqref{eq7: 16.09.2024} with \eqref{eq6: 22.05.2025}, we conclude that
\begin{align}
\card{\set{H}_1} & \geq \max \Bigl\{ (\delta n_1 n_2)^{pq} \, \exp \bigl( - p(q-1) \, \Ent{U} - q(p-1) \, \Ent{V} \bigr), \nonumber \\
\label{eq0: LB on H1}
& \hspace*{1.3cm} (\delta n_1 n_2)^q \, \exp \bigl( -(q-1) \, \Ent{U} \bigr) \Bigr\}.
\end{align}
By the definition of $\set{H}_1$ as the set of all homomorphisms from the complete bipartite graph $\CoBG{p}{q}$ to the
bipartite graph $\Gr{G}$, where the vertices in the partite set of size $p$ in $\CoBG{p}{q}$ are mapped to the partite
set $\set{U}$ of size $n_1$ in $\Gr{G}$, and the vertices in the partite set of size $q$ in $\CoBG{p}{q}$ are mapped to
the partite set $\set{V}$ of size $n_2$ in $\Gr{G}$, we obtain, by applying a symmetry argument to the second term on
the right-hand side of \eqref{eq0: LB on H1} (interchanging $p$ with $q$ and $U$ with $V$), the additional bound
$\card{\set{H}_1} \geq (\delta n_1 n_2)^p \, \exp \bigl( -(p-1) \, \Ent{V} \bigr)$. Combining this inequality with
\eqref{eq0: LB on H1} yields the improved lower bound
\begin{align}
\card{\set{H}_1} & \geq \max \Bigl\{ (\delta n_1 n_2)^{pq} \, \exp \bigl( - p(q-1) \, \Ent{U} - q(p-1) \, \Ent{V} \bigr), \nonumber \\
\label{eq: LB on H1}
& \hspace*{1.3cm} (\delta n_1 n_2)^q \, \exp \bigl( -(q-1) \, \Ent{U} \bigr), \; (\delta n_1 n_2)^p \, \exp \bigl( -(p-1) \, \Ent{V} \bigr) \Bigr\}.
\end{align}
Similarly to the transition from \eqref{eq8: 16.09.2024} to \eqref{eq8b: 16.09.2024}, we get from \eqref{eq: LB on H1} that
\begin{align}
\card{\set{H}_2} & \geq \max \Bigl\{ (\delta n_1 n_2)^{pq} \, \exp \bigl( - q(p-1) \, \Ent{U} - p(q-1) \, \Ent{V} \bigr), \nonumber \\
\label{eq: LB on H2}
& \hspace*{1.3cm} (\delta n_1 n_2)^p \, \exp \bigl( -(p-1) \, \Ent{U} \bigr), \; (\delta n_1 n_2)^q \, \exp \bigl( -(q-1) \, \Ent{V} \bigr) \Bigr\}.
\end{align}
Hence, combining \eqref{eq8c: 16.09.2024}, \eqref{eq: LB on H1}, and \eqref{eq: LB on H2}, we obtain the lower bound
\begin{align}
\homcount{\CoBG{p}{q}}{\Gr{G}}
&\geq \max \Bigl\{ (\delta n_1 n_2)^{pq} \, \exp \bigl( - p(q-1) \, \Ent{U} - q(p-1) \, \Ent{V} \bigr), \nonumber \\
& \hspace*{1.3cm} (\delta n_1 n_2)^q \, \exp \bigl( -(q-1) \, \Ent{U} \bigr), \; (\delta n_1 n_2)^p \, \exp \bigl( -(p-1) \, \Ent{V} \bigr) \Bigr\} \nonumber \\
& \hspace*{0.1cm} + \max \Bigl\{ (\delta n_1 n_2)^{pq} \, \exp \bigl( - q(p-1) \, \Ent{U} - p(q-1) \, \Ent{V} \bigr), \nonumber \\
\label{eq1: 23.06.2025}
& \hspace*{1.3cm} (\delta n_1 n_2)^p \, \exp \bigl( -(p-1) \, \Ent{U} \bigr), \; (\delta n_1 n_2)^q \, \exp \bigl( -(q-1) \, \Ent{V} \bigr) \Bigr\}.
\end{align}
This provides a refinement and strengthening of the lower bound on the right-hand side of \eqref{eq4b: 16.09.2024}, since we have
\begin{align}
\homcount{\CoBG{p}{q}}{\Gr{G}}
&\geq (\delta n_1 n_2)^{pq} \Bigl[ \exp \bigl( - p(q-1) \, \Ent{U} - q(p-1) \, \Ent{V} \bigr) \nonumber \\
\label{eq1: loosening}
& \hspace*{2cm} + \exp \bigl( - q(p-1) \, \Ent{U} - p(q-1) \, \Ent{V} \bigr) \Bigr] \\
&\geq (\delta n_1 n_2)^{pq} \, \Bigl[ \exp\Bigl( -p(q-1) \log n_1 - q(p-1) \log n_2 \Bigr) \nonumber \\
\label{eq2: loosening}
& \hspace*{2.2cm} + \exp\Bigl( -q(p-1) \log n_1 - p(q-1) \log n_2 \Bigr) \Bigr] \\
&= \delta^{pq} \, \bigl(n_1^p n_2^q + n_1^q n_2^p \bigr) \label{eq8a: 22.06.2025} \\[0.1cm]
&= \delta^{pq} \, \homcount{\CoBG{p}{q}}{\CoBG{n_1}{n_2}}, \label{eq8b: 22.06.2025}
\end{align}
where \eqref{eq1: loosening} is a relaxed form of \eqref{eq1: 23.06.2025}, obtained by taking the first term in each
maximum on the right-hand side of~\eqref{eq1: 23.06.2025};
\eqref{eq2: loosening} follows from the upper bound on the Shannon entropy, namely, $\Ent{U} \leq \log n_1$ and
$\Ent{V} \leq \log n_2$; finally, \eqref{eq8a: 22.06.2025} validates inequality \eqref{eq4b: 16.09.2024}, while
\eqref{eq8b: 22.06.2025} validates inequality \eqref{eq4c: 16.09.2024} by relying on equality \eqref{eq1: 29.07.25}.
\end{proof}

\begin{remark}
\label{remark: specialization of the refined LB}
{\em If $\Gr{G}$ is a nonempty bipartite graph in which all vertices within each partite set have equal degree, then the lower
bound on the right-hand side of \eqref{eq4b: 16.09.2024} coincides with that of \eqref{eq1: loosening}. This follows
from \eqref{eq: Ent U} and \eqref{eq: Ent V}, as in this case we have $\Ent{U} = \log n_1$ and $\Ent{V} = \log n_2$.}
\end{remark}

\section{Bounds on Homomorphism Counts Between Bipartite Graphs}
\label{section: boounds on homomorphism counts between bipartite graphs}

The present section introduces upper and lower bounds on the number of homomorphisms between bipartite graphs.
It relies in part on our earlier bounds on the number of homomorphisms from a complete bipartite graph to an arbitrary
bipartite graph, as introduced in Sections~\ref{section: exact expressions and comb. bounds} and~\ref{section: IT bounds}.

\subsection{Upper bounds on homomorphism numbers}
\label{subsection: Upper bounds on homomorphism numbers}

The following upper bounds rely on the earlier findings in Section~\ref{section: exact expressions and comb. bounds}
and a reverse Sidorenko inequality in \cite{SahSSZ20}.
\begin{proposition}
\label{proposition: UB on hom(BG, BG)}
{\em Let $\Gr{F}$ be a bipartite graph with $s \geq 0$ isolated vertices. Then, using the notation in
Proposition~\ref{prop: exact counting: CoBG --> BG},
\begin{enumerate}
\item For every bipartite graph $\Gr{G}$,
\begin{align}
\label{eq1: 08.07.25}
\hspace*{-0.9cm} \homcount{\Gr{F}}{\Gr{G}} \leq \card{\V{\Gr{G}}}^s \prod_{\{u,v\} \in \E{\Gr{F}}}
\Biggl( \sum_{k=1}^{d_{\Gr{F}}(u)} \sum_{\ell=1}^{d_{\Gr{F}}(v)} k! \, \ell! \, S(d_{\Gr{F}}(u),k) \, S(d_{\Gr{F}}(v),\ell)
\, \bigl[N_{k, \ell}(\Gr{G}) + N_{\ell, k}(\Gr{G}) \bigr] \Biggr)^{\frac1{d_{\Gr{F}}(u) \, d_{\Gr{F}}(v)}},
\end{align}
with equality holding in \eqref{eq1: 08.07.25} if $\Gr{F}$ is a disjoint union of isolated vertices and complete bipartite graphs.
\item In particular, for every bipartite graph $\Gr{G}$ that contains no 4-cycles,
\begin{align}
\label{eq2: 01.07.25}
\homcount{\Gr{F}}{\Gr{G}} \leq \card{\V{\Gr{G}}}^s \prod_{\{u,v\} \in \E{\Gr{F}}} \Biggl( \,
\sum_{w \in \V{\Gr{G}}} \, \bigl\{d_{\Gr{G}}(w)^{d_{\Gr{F}}(u)} + d_{\Gr{G}}(w)^{d_{\Gr{F}}(v)}
\, \bigr\} - 2 \, \card{\hspace*{-0.05cm} \E{\Gr{G}}} \Biggr)^{\frac1{d_{\Gr{F}}(u) \, d_{\Gr{F}}(v)}},
\end{align}
with equality holding in \eqref{eq2: 01.07.25} if $\Gr{F}$ is a disjoint union of isolated vertices and star graphs.
\end{enumerate}}
\end{proposition}
\begin{proof}
By \cite[Theorem~1.9]{SahSSZ20}, if $\Gr{F}$ is a triangle-free graph with no isolated vertices, and $\Gr{H}$
is any graph (possibly with loops), then
\begin{align}
\label{eq: SahSSZ20}
\homcount{\Gr{F}}{\Gr{H}} \leq \prod_{\{u, v\} \in \E{\Gr{F}}} \homcount{\CoBG{d_{\Gr{F}}(u)}{d_{\Gr{F}}(v)}}{\Gr{H}}^{\frac1{d_{\Gr{F}}(u) d_{\Gr{F}}(v)}}.
\end{align}
Since there are no homomorphisms from a non-bipartite graph into a bipartite graph $\Gr{G}$, inequality \eqref{eq: SahSSZ20} is used
to obtain an upper bound on $\homcount{\Gr{F}}{\Gr{G}}$ only in the special case where $\Gr{F}$ is bipartite; otherwise,
$\homcount{\Gr{F}}{\Gr{G}}=0$ and no bound is needed.

By \eqref{eq1: homomorphism numbers}, we may assume without any loss of generality that $s=0$.
Inequality~\eqref{eq1: 08.07.25} follows from a combination of Proposition~\ref{prop: exact counting: CoBG --> BG}
and \eqref{eq: SahSSZ20}. Suppose that $\Gr{F}$ is the disjoint union of connected components $\Gr{F}_1, \ldots, \Gr{F}_r$. In this case, not only does
\eqref{eq1: homomorphism numbers} hold, but the upper bound on the right-hand side of \eqref{eq: SahSSZ20} also equals the product of the corresponding
upper bounds on $\{\homcount{\Gr{F}_j}{\Gr{G}}\}_{j=1}^r$. Therefore, to establish the sufficient condition for equality, it suffices to show that 
\eqref{eq1: 08.07.25} holds with equality when $\Gr{F} = \CoBG{p}{q}$ for arbitrary $p,q \in \naturals$.
In that case, the right-hand side of \eqref{eq1: 08.07.25} equals
\begin{align}
& \hspace*{-0.3cm} \prod_{\{u,v\} \in \E{\Gr{F}}}
\Biggl( \, \sum_{k=1}^{d_{\Gr{F}}(u)} \sum_{\ell=1}^{d_{\Gr{F}}(v)} k! \, \ell! \, S(d_{\Gr{F}}(u),k) \, S(d_{\Gr{F}}(v),\ell)
\, \bigl[N_{k, \ell}(\Gr{G}) + N_{\ell, k}(\Gr{G}) \bigr] \Biggr)^{\frac1{d_{\Gr{F}}(u) \, d_{\Gr{F}}(v)}} \nonumber \\[0.1cm]
\label{eq1: 19.07.2025}
&= \sum_{k=1}^p \sum_{\ell=1}^q  k! \, \ell! \, S(p,k) \, S(q,\ell)
\, \bigl[N_{k, \ell}(\Gr{G}) + N_{\ell, k}(\Gr{G}) \bigr] \\[0.1cm]
\label{eq2: 19.07.2025}
&= \homcount{\Gr{F}}{\Gr{G}},
\end{align}
where \eqref{eq1: 19.07.2025} holds because every edge of $\CoBG{p}{q}$ has endpoint degrees $p$ and $q$, so all $pq$
factors in the product are identical and each has exponent $\frac{1}{pq}$; equality \eqref{eq2: 19.07.2025} 
holds by Proposition~\ref{prop: exact counting: CoBG --> BG}.

We next prove \eqref{eq2: 01.07.25} when $\Gr{G}$ contains no $4$-cycles, together with its sufficient condition for equality.
Under this assumption, inequality \eqref{eq1: 01.07.25} holds with equality.
Substituting this equality into \eqref{eq: SahSSZ20}, with $p=d_{\Gr{F}}(u)$ and $q=d_{\Gr{F}}(v)$ for each edge
$\{u,v\}\in\E{\Gr{F}}$, and restoring the factor $\card{\V{\Gr{G}}}^s$ for the isolated vertices proves \eqref{eq2: 01.07.25}.
Since star graphs are complete bipartite graphs with no $4$-cycles, the equality argument above applies to each star component. 
Thus, \eqref{eq2: 01.07.25} holds with equality if $\Gr{F}$ is a disjoint union of isolated vertices and star graphs.
\end{proof}

\begin{remark}
\label{remark: SahSSZ20}
{\em \cite[Proposition~1.10]{SahSSZ20} demonstrates that the triangle-free condition for $\Gr{F}$ in \cite[Theorem~1.9]{SahSSZ20}
is essential for the validity of \eqref{eq: SahSSZ20}, in the sense that for every graph~$\Gr{F}$ containing a triangle, there exists a
graph~$\Gr{H}$ for which inequality~\eqref{eq: SahSSZ20} fails to hold. This interesting observation from \cite{SahSSZ20}, however, is
less relevant in the context of Proposition~\ref{proposition: UB on hom(BG, BG)} because $\Gr{F}$ is a bipartite graph, so it contains no
odd cycles.}
\end{remark}

Application of Item~2 in Proposition~\ref{proposition: UB on hom(BG, BG)}
readily yields the following upper bounds on homomorphism counts.
\begin{corollary}
\label{corollary: upper bounds on homomorphism counts}
{\em Let $\Gr{F}$ be an $r$-regular bipartite graph on $n_{\Gr{F}}$ vertices, where $r \geq 1$,
and let $\Gr{G}$ be a bipartite graph that contains no 4-cycles. Then,
\begin{align}
\label{eq1: 18.07.25}
\homcount{\Gr{F}}{\Gr{G}} \leq \Biggl( \, 2 \sum_{w \in \V{\Gr{G}}} \, d_{\Gr{G}}(w)^{r}
- 2 \, \card{\hspace*{-0.05cm} \E{\Gr{G}}} \Biggr)^{\frac{n_{\vspace*{-0.15cm} \Gr{F}}}{2r}}.
\end{align}
In particular, if $\Gr{T}$ is a tree on $n_{\Gr{T}}$ vertices, then
\begin{align}
\label{eq2: 18.07.25}
\homcount{\Gr{F}}{\Gr{T}} \leq \Biggl( \, 2 \sum_{w \in \V{\Gr{T}}} \, d_{\Gr{T}}(w)^{r}
- 2(n_{\Gr{T}}-1) \Biggr)^{\frac{n_{\vspace*{-0.15cm} \Gr{F}}}{2r}}.
\end{align}}
\end{corollary}

\begin{remark}
{\em A specialization of \cite[Theorem~1.9]{SahSSZ20} to the setting where $\Gr{F}$ is a regular bipartite
graph recovers the earlier result in \cite{GalvinT04}. This special case, combined with the equality condition 
in \eqref{eq1: 01.07.25}, suffices to derive Corollary~\ref{corollary: upper bounds on homomorphism counts}.
The proof of the result in \cite{GalvinT04} and that of its extension in \cite{WangTL23} to bipartite graphs 
that are regular on one side rely on Shearer's entropy lemma, whereas the proof of \cite[Theorem~1.9]{SahSSZ20} is algebraic.}
\end{remark}

\subsection{Lower bounds on homomorphism numbers}
\label{subsection: Lower bounds on homomorphism numbers}

This section derives lower bounds on the number of homomorphisms from one simple bipartite graph to another.
\begin{lemma}
\label{lemma: auxiliary result}
{\em Let $\Gr{F}$ be a bipartite graph obtained by removing an edge from the complete bipartite graph $\CoBG{p}{q}$,
where $p,q \geq 2$, and let $\Gr{G}$ be a simple bipartite graph. Define the set
\begin{align}
\label{eq5: 24.07.25}
\set{D} \triangleq \Bigl\{ (u,v) \in \V{\Gr{G}} \times \V{\Gr{G}}:  \{u,v\} \not\in \E{\Gr{G}}, \; \exists \, w \in \set{N}_{\Gr{G}}(u) \; \text{such that} \;
\set{N}_{\Gr{G}}(v) \cap \set{N}_{\Gr{G}}(w) \neq \es \Bigr\},
\end{align}
where $\set{N}_{\Gr{G}}(\cdot)$ denotes the set of neighbors of the specified vertex in $\Gr{G}$. Then, 
\begin{align}
\label{eq: 21.07.25}
\homcount{\Gr{F}}{\Gr{G}} &\geq \homcount{\CoBG{p}{q}}{\Gr{G}} + \eta_{p,q}(\Gr{G}) \\
\label{eq2: 21.07.25}
& \geq \homcount{\CoBG{p}{q}}{\Gr{G}},
\end{align}
where
\begin{align}
\label{eq6: 24.07.25}
\eta_{p,q}(\Gr{G}) \triangleq \sum_{(u,v) \in \set{D}} \; \sum_{w \in \set{N}_{\Gr{G}}(u): \;
\set{N}_{\Gr{G}}(v) \cap \set{N}_{\Gr{G}}(w) \neq \es} \,
\bigcard{\set{N}_{\Gr{G}}(v) \cap \set{N}_{\Gr{G}}(w)}^{\, \max\{p,q\}-1}.
\end{align}
Furthermore, the following equality statements hold:
\begin{enumerate}[(1)]
\item $\homcount{\Gr{F}}{\Gr{G}} = \homcount{\CoBG{p}{q}}{\Gr{G}}$
if and only if $\eta_{p,q}(\Gr{G})=0$ (equivalently, $\set{D}=\es$).
\item
Inequality \eqref{eq: 21.07.25} holds with equality whenever either $p=2$ or $q=2$, regardless of $\Gr{G}$.
\end{enumerate}}
\end{lemma}
\begin{proof}
The bipartite graph $\Gr{F}$ is, by assumption, a spanning subgraph of the complete bipartite graph $\CoBG{p}{q}$.
Consequently, $\Hom{\CoBG{p}{q}}{\Gr{G}} \subseteq \Hom{\Gr{F}}{\Gr{G}}$, since every mapping that maps each edge 
of $\CoBG{p}{q}$ to an edge of $\Gr{G}$ maps, in particular, each edge of $\Gr{F}$ to an edge of $\Gr{G}$. This 
implies that $\homcount{\Gr{F}}{\Gr{G}} \geq \homcount{\CoBG{p}{q}}{\Gr{G}}$. Let
\begin{align}
\label{eq4: 23.07.25}
\Delta(\Gr{F}, \Gr{G}) \triangleq \homcount{\Gr{F}}{\Gr{G}} - \homcount{\CoBG{p}{q}}{\Gr{G}} \geq 0.
\end{align}
Let $\set{A}=\{a_1,\ldots,a_p\}$ and $\set{B}=\{b_1,\ldots,b_q\}$ be the two partite sets of
$\CoBG{p}{q}$, and let $\set{D}$ be defined as in \eqref{eq5: 24.07.25}.
Without loss of generality, suppose that the removed edge is $\{a_1,b_1\}$.

Fix $(u,v)\in\set{D}$. Construct a mapping $\phi \colon \V{\CoBG{p}{q}} \to \V{\Gr{G}}$ as follows:
\begin{enumerate}[(1)]
\item Let $(\phi(a_1),\phi(b_1))=(u,v)$.
\item Select $w\in\set{N}_{\Gr{G}}(u)$ such that $\set{N}_{\Gr{G}}(v)\cap\set{N}_{\Gr{G}}(w)\neq\es$,
and let $\phi(b_\ell)=w$ for every $\ell\in\{2,\ldots,q\}$. Such a vertex $w$ exists by the definition 
of $\set{D}$.
\item Map each vertex $a_k$, $k\in\{2, \ldots, p\}$, independently to any vertex in 
$\set{N}_{\Gr{G}}(v) \cap \set{N}_{\Gr{G}}(w)$.
\end{enumerate}
Since $\{\phi(a_1),\phi(b_1)\}=\{u,v\}\notin\E{\Gr{G}}$,
we have $\phi\notin\Hom{\CoBG{p}{q}}{\Gr{G}}$.
On the other hand,
\begin{itemize}
\item For every $\ell\in\{2,\ldots,q\}$, the edge
$\{a_1,b_\ell\}$ is mapped to $\{u,w\}\in\E{\Gr{G}}$.
\item For every $k\in\{2,\ldots,p\}$, the edge
$\{a_k,b_1\}$ is mapped to $\{\phi(a_k),v\}$,
and each edge $\{a_k,b_\ell\}$ with $\ell\geq2$
is mapped to $\{\phi(a_k),w\}$.
All these image pairs are edges of $\Gr{G}$ by Item~(3).
\end{itemize}
Thus, $\phi \in \Hom{\Gr{F}}{\Gr{G}}$, so 
\begin{equation}
\label{eq1: 08.10.26}
\phi \in \Hom{\Gr{F}}{\Gr{G}} \setminus \Hom{\CoBG{p}{q}}{\Gr{G}}.
\end{equation}
For the fixed pair $(u,v) \in \set{D}$, this construction gives
\begin{equation}
\label{eq2: 08.10.26}
\sum_{\substack{w \in \set{N}_{\Gr{G}}(u):\\
\set{N}_{\Gr{G}}(v) \cap \set{N}_{\Gr{G}}(w) \neq \es}}
\bigcard{\set{N}_{\Gr{G}}(v) \cap \set{N}_{\Gr{G}}(w)}^{p-1}
\end{equation}
distinct mappings. Indeed, different choices of $w$ in the sum \eqref{eq2: 08.10.26}
give different images of $b_2$, and, for a fixed $w$, different choices in 
Item~(3) give different images of at least one vertex in $\{a_2,\ldots,a_p\}$.

\smallskip 
Analogously, for each $(u,v)\in\set{D}$, construct a mapping
$\phi$ as follows:
\begin{enumerate}[(a)]
\item Let $(\phi(a_1),\phi(b_1))=(v,u)$.
\item Select $w\in\set{N}_{\Gr{G}}(u)$ such that $\set{N}_{\Gr{G}}(v)\cap\set{N}_{\Gr{G}}(w)\neq\es$,
and let $\phi(a_k)=w$ for every $k\in\{2,\ldots,p\}$.
\item Map each vertex $b_\ell$, $\ell\in\{2,\ldots,q\}$, independently to any vertex in 
$\set{N}_{\Gr{G}}(v) \cap \set{N}_{\Gr{G}}(w)$.
\end{enumerate}
By the same reasoning as above, each resulting mapping belongs to
$\Hom{\Gr{F}}{\Gr{G}}\setminus\Hom{\CoBG{p}{q}}{\Gr{G}}$,
and, for the fixed pair $(u,v)$, there are
\[
\sum_{\substack{w\in\set{N}_{\Gr{G}}(u):\\
\set{N}_{\Gr{G}}(v)\cap\set{N}_{\Gr{G}}(w)\neq\es}}
\bigcard{\set{N}_{\Gr{G}}(v)\cap\set{N}_{\Gr{G}}(w)}^{q-1}
\]
distinct mappings.

If $p\geq q$, use the construction in Items~(1)--(3)
for every $(u,v)\in\set{D}$.
If $q>p$, use the construction in Items~(a)--(c)
for every $(u,v)\in\set{D}$.
Within either construction, distinct ordered pairs give
distinct images of the endpoints of the removed edge.
Therefore, no mappings are counted more than once, and
\begin{align}
\label{eq1: 25.07.25}
\Delta(\Gr{F},\Gr{G})
&\geq
\sum_{(u,v)\in\set{D}}
\sum_{\substack{w\in\set{N}_{\Gr{G}}(u):\\
\set{N}_{\Gr{G}}(v)\cap\set{N}_{\Gr{G}}(w)\neq\es}}
\bigcard{\set{N}_{\Gr{G}}(v)\cap\set{N}_{\Gr{G}}(w)}
^{\max\{p,q\}-1} \\
\label{eq2: 25.07.25}
&=\eta_{p,q}(\Gr{G}).
\end{align}
Combining \eqref{eq4: 23.07.25} and
\eqref{eq1: 25.07.25}--\eqref{eq2: 25.07.25}
proves \eqref{eq: 21.07.25}.
Since $\eta_{p,q}(\Gr{G})\geq0$, inequality
\eqref{eq2: 21.07.25} also follows.

We next prove that
$\homcount{\Gr{F}}{\Gr{G}}
=\homcount{\CoBG{p}{q}}{\Gr{G}}$
if and only if $\eta_{p,q}(\Gr{G})=0$.
Every summand in \eqref{eq6: 24.07.25} is positive,
and each pair in $\set{D}$ has at least one admissible
choice of $w$. Hence, $\eta_{p,q}(\Gr{G})=0$
if and only if $\set{D}=\es$.
If the two homomorphism counts are equal,
\eqref{eq: 21.07.25} implies that $\eta_{p,q}(\Gr{G})=0$.

Conversely, suppose that $\eta_{p,q}(\Gr{G})=0$, so
$\set{D}=\es$.
Suppose, for contradiction, that there exists
\[
\phi\in
\Hom{\Gr{F}}{\Gr{G}}\setminus
\Hom{\CoBG{p}{q}}{\Gr{G}}.
\]
Let $u=\phi(a_1)$ and $v=\phi(b_1)$.
Then $\{u,v\}\notin\E{\Gr{G}}$.
Since $\{a_1,b_2\}\in\E{\Gr{F}}$, the vertex
$w=\phi(b_2)$ belongs to $\set{N}_{\Gr{G}}(u)$.
Moreover, since $\{a_2,b_1\}$ and $\{a_2,b_2\}$
are edges of $\Gr{F}$,
\[
\phi(a_2)\in
\set{N}_{\Gr{G}}(v)\cap\set{N}_{\Gr{G}}(w).
\]
Thus, by \eqref{eq5: 24.07.25}, $(u,v)\in\set{D}$, contradicting $\set{D}=\es$.
This proves the first equality statement.

Finally, suppose that $p=2$ and $q\geq2$. Every mapping
\[
\phi\in
\Hom{\Gr{F}}{\Gr{G}}\setminus
\Hom{\CoBG{2}{q}}{\Gr{G}}
\]
is uniquely specified by
\[
u=\phi(b_1),\qquad
v=\phi(a_1),\qquad
w=\phi(a_2),
\]
together with the images of $b_2,\ldots,b_q$.
The necessary and sufficient conditions are
\[
\{u,v\}\notin\E{\Gr{G}},\qquad
w\in\set{N}_{\Gr{G}}(u),\qquad
\phi(b_\ell)\in
\set{N}_{\Gr{G}}(v)\cap\set{N}_{\Gr{G}}(w)
\quad\text{for every }\ell\in\{2,\ldots,q\}.
\]
In particular, $(u,v)\in\set{D}$, and the common
neighborhood in this display is nonempty.
For each admissible triple $(u,v,w)$, the images of
$b_2,\ldots,b_q$ can be chosen independently.
Consequently,
\begin{equation}
\label{eq3: 27.07.25}
\begin{aligned}
\Delta(\Gr{F},\Gr{G})
&= \sum_{(u,v)\in\set{D}}
\sum_{\substack{w\in\set{N}_{\Gr{G}}(u):\\
\set{N}_{\Gr{G}}(v)\cap\set{N}_{\Gr{G}}(w)\neq\es}}
\bigcard{\set{N}_{\Gr{G}}(v)\cap\set{N}_{\Gr{G}}(w)}^{q-1} \\
&=\eta_{2,q}(\Gr{G}).
\end{aligned}
\end{equation}
Therefore, by \eqref{eq4: 23.07.25} and \eqref{eq3: 27.07.25},
\begin{align}
\label{eq4: 27.07.25}
\homcount{\Gr{F}}{\Gr{G}}
= \homcount{\CoBG{2}{q}}{\Gr{G}} + \eta_{2,q}(\Gr{G}).
\end{align}
The case $q=2$ follows by interchanging the two
partite sets and using
$\eta_{p,2}(\Gr{G})=\eta_{2,p}(\Gr{G})$,
which follows directly from \eqref{eq6: 24.07.25}.
This proves the second equality statement.
\end{proof}

\begin{remark}
\label{remark: star graph}
{\em The remaining case in Lemma~\ref{lemma: auxiliary result}, where either $p=1$ or $q=1$, is easy, and the quantity
$\homcount{\Gr{F}}{\Gr{G}}$ admits a closed-form expression. Suppose that $p=1$ and $q \geq 1$. Then, $\Gr{F}$ is a
disjoint union of a star graph $\StarG{q} = \CoBG{1}{q-1}$ and an isolated vertex. In this case, we have
\begin{align}
\label{eq3: 25.07.25}
\homcount{\Gr{F}}{\Gr{G}} &= \homcount{\CoG{1}}{\Gr{G}} \; \homcount{\StarG{q}}{\Gr{G}} \\
\label{eq4: 25.07.25}
&= \card{\V{\Gr{G}}} \, \sum_{v \in \V{\Gr{G}}} d_{\Gr{G}}(v)^{q-1}
\end{align}
where \eqref{eq3: 25.07.25} follows from \eqref{eq1: homomorphism numbers}. For $q \geq 2$,
\eqref{eq4: 25.07.25} follows from \eqref{eq6: 09.07.25} with $m=q-1$. For $q=1$, the graph 
$\Gr{F}$ consists of two isolated vertices, so both sides of \eqref{eq4: 25.07.25} equal 
$\card{\V{\Gr{G}}}^{2}$, under the convention $0^0=1$.}
\end{remark}

Let $\Gr{F}$ be a spanning subgraph of the complete bipartite graph $\CoBG{p}{q}$, where $p,q \geq 2$,
and let $\Gr{G}$ be a simple bipartite graph. Let $\ell = pq - \card{\E{\Gr{F}}}$ be the number of edges that are removed from
$\CoBG{p}{q}$ to obtain $\Gr{F}$. Define a sequence of spanning bipartite subgraphs $\{\Gr{F}_k\}_{k=0}^{\ell}$ such that
\begin{itemize}
\item $\Gr{F}_0 = \CoBG{p}{q}$,
\item For each $k \in \OneTo{\ell}$, $\Gr{F}_k$ is obtained from $\Gr{F}_{k-1}$ by removing a single edge.
\item $\Gr{F}_\ell = \Gr{F}$.
\end{itemize}
The order in which the edges are removed in the transition from $\CoBG{p}{q}$ to $\Gr{F}$ is arbitrary.
Let $\{\Delta_k\}_{k=1}^{\ell}$ be the sequence defined as
\begin{align}
\label{eq: Delta_k}
\Delta_k \triangleq \homcount{\Gr{F}_k}{\Gr{G}} - \homcount{\Gr{F}_{k-1}}{\Gr{G}}, \quad k \in \OneTo{\ell},
\end{align}
so that $\Delta_k$ represents the nonnegative change in the number of homomorphisms to $\Gr{G}$, resulting from the removal
of a single edge in the transition from $\Gr{F}_{k-1}$ to $\Gr{F}_k$.
\begin{lemma}
\label{lemma: LB on Delta_k}
{\em The sequence $\{\Delta_k\}_{k=1}^{\ell}$ satisfies
\begin{align}
\label{eq: LB on Delta_k}
\Delta_k \geq \eta_{p,q}(\Gr{G}),  \quad \forall \, k \in \OneTo{\ell},
\end{align}
where $\eta_{p,q}(\Gr{G})$ is the nonnegative integer defined in \eqref{eq6: 24.07.25}.}
\end{lemma}
\begin{proof}
Let $\set{A} = \{a_1, \ldots, a_p\}$ and $\set{B} = \{b_1, \ldots, b_q\}$ denote the two partite sets of $\CoBG{p}{q}$.
For each $k \in \OneTo{\ell}$, let $e_k = \{a_{i_k}, b_{j_k}\} \in \E{\CoBG{p}{q}}$ be the removed edge in the transition 
from the graph $\Gr{F}_{k-1}$ to its spanning subgraph $\Gr{F}_k$, and let $\widehat{\Gr{F}}_k$ be the spanning subgraph 
of $\CoBG{p}{q}$ that is obtained by removing the single edge $e_k$. Let $\phi \colon \V{\CoBG{p}{q}} \to \V{\Gr{G}}$ satisfy
$\phi \in \Hom{\widehat{\Gr{F}}_k}{\Gr{G}} \setminus \Hom{\CoBG{p}{q}}{\Gr{G}}$, provided that such a vertex mapping exists.
By definition, $\{\phi(a_{i_k}), \phi(b_{j_k})\} \not\in \E{\Gr{G}}$, but for all other edges $\{a_i, b_j\} \in \E{\CoBG{p}{q}}$
where $(i,j) \neq (i_k, j_k)$ we have $\{\phi(a_i), \phi(b_j)\} \in \E{\Gr{G}}$. In particular, the images of the edges of
the spanning subgraph $\Gr{F}_k$ are edges in $\Gr{G}$ since (by construction) $e_k \not\in \E{\Gr{F}_k}$. This means that
$\phi \in \Hom{\Gr{F}_k}{\Gr{G}} \setminus \Hom{\Gr{F}_{k-1}}{\Gr{G}}$. Consequently,
\begin{align}
\label{eq1: 26.07.25}
\Hom{\widehat{\Gr{F}}_k}{\Gr{G}} \setminus \Hom{\CoBG{p}{q}}{\Gr{G}} \, \subseteq \, \Hom{\Gr{F}_k}{\Gr{G}} \setminus \Hom{\Gr{F}_{k-1}}{\Gr{G}},
\end{align}
which yields
\begin{align}
\label{eq2: 26.07.25}
\Delta_k &= \homcount{\Gr{F}_k}{\Gr{G}} - \homcount{\Gr{F}_{k-1}}{\Gr{G}} \\
\label{eq3: 26.07.25}
&\geq \homcount{\widehat{\Gr{F}}_k}{\Gr{G}} - \homcount{\CoBG{p}{q}}{\Gr{G}} \\
\label{eq4: 26.07.25}
&\geq \eta_{p,q}(\Gr{G}),
\end{align}
where \eqref{eq2: 26.07.25} is \eqref{eq: Delta_k}; \eqref{eq3: 26.07.25} holds by \eqref{eq1: 26.07.25},
and \eqref{eq4: 26.07.25} holds by Lemma~\ref{lemma: auxiliary result} and by the construction of $\widehat{\Gr{F}}_k$
as the spanning subgraph of $\CoBG{p}{q}$ that is obtained by removing the single edge $e_k$.
\end{proof}

\begin{corollary}
\label{corollary: LB on homcount(F,G)}
{\em Let $\Gr{F}$ be a spanning subgraph of the complete bipartite graph $\CoBG{p}{q}$, where $p,q \in \naturals$,
and let $\Gr{G}$ be a simple bipartite graph.
\begin{enumerate}
\item If $p, q \geq 2$, then
\begin{align}
\label{eq5: 26.07.25}
\homcount{\Gr{F}}{\Gr{G}} \geq \homcount{\CoBG{p}{q}}{\Gr{G}} + (pq - \card{\E{\Gr{F}}}) \, \eta_{p,q}(\Gr{G}).
\end{align}
\item Else, if $p=1$ or $q=1$, then
\begin{align}
\label{eq1: 27.07.25}
\homcount{\Gr{F}}{\Gr{G}} = \card{\V{\Gr{G}}}^{pq-\card{\E{\Gr{F}}}} \, \sum_{v \in \V{\Gr{G}}} d_{\Gr{G}}(v)^{\card{\E{\Gr{F}}}}.
\end{align}
\end{enumerate}}
\end{corollary}
\begin{proof}
Suppose that $p, q \geq 2$.
Let $\{\Delta_k\}_{k=1}^{\ell}$ be the nonnegative sequence defined in \eqref{eq: Delta_k} for the sequence of bipartite graphs
$\{\Gr{F}_k\}_{k=0}^{\ell}$ constructed above, where $\Gr{F}_0 = \CoBG{p}{q}$, $\Gr{F}_\ell=\Gr{F}$, and
$\ell=pq-\card{\E{\Gr{F}}}$ is the total number of edges removed from $\CoBG{p}{q}$ to obtain $\Gr{F}$. Then, 
\begin{align}
\label{eq6: 26.07.25}
\homcount{\Gr{F}}{\Gr{G}} &= \homcount{\CoBG{p}{q}}{\Gr{G}} + \sum_{k=1}^{\ell} \Delta_k  \\
\label{eq7: 26.07.25}
&\geq \homcount{\CoBG{p}{q}}{\Gr{G}} + (pq - \card{\E{\Gr{F}}}) \, \eta_{p,q}(\Gr{G}),
\end{align}
where \eqref{eq6: 26.07.25} follows by telescoping the sum on its right-hand side using \eqref{eq: Delta_k}, 
and \eqref{eq7: 26.07.25} holds by Lemma~\ref{lemma: LB on Delta_k}. This proves \eqref{eq5: 26.07.25}.

If either $p$ or $q$ equals~1, then $\Gr{F}$ is a disjoint union of the star graph $\CoBG{1}{m}$, where $m \triangleq \card{\E{\Gr{F}}}$,
together with $r = pq-\card{\E{\Gr{F}}}$ isolated vertices. Equality~\eqref{eq1: 27.07.25} then follows from \eqref{eq1: homomorphism numbers}
and \eqref{eq6: 09.07.25}.
\end{proof}

\begin{remark}
\label{remark: p,q, number of deleted edges}
{\em If $\Gr{F}$ is a connected bipartite graph with at least one edge,
then its bipartition is unique up to exchanging the two partite sets.
Consequently, the complete bipartite graph containing $\Gr{F}$ as a
spanning subgraph is unique up to isomorphism.
If $\Gr{F}$ is disconnected, however, multiple nonisomorphic complete
bipartite graphs may contain it as a spanning subgraph.
This raises the question of how to efficiently choose
$p,q\in\naturals$ to obtain a good lower bound on
$\homcount{\Gr{F}}{\Gr{G}}$ of the form \eqref{eq5: 26.07.25},
where $\Gr{F}$ and $\Gr{G}$ are given simple bipartite graphs,
with $\Gr{G}$ nonempty.

The selection of $p$ and $q$ can significantly influence the tightness
of this lower bound. For simplicity, and to make this selection
independent of $\Gr{G}$, we choose $p$ and $q$ to minimize
$\ell \triangleq pq-\card{\E{\Gr{F}}}$, subject to the constraint
that $\Gr{F}$ is a spanning subgraph of $\CoBG{p}{q}$.
This constraint ensures that $\ell$ is nonnegative, since $\ell$
is the number of edges removed from $\CoBG{p}{q}$ to obtain $\Gr{F}$.
This choice minimizes the number of edge deletions used in the proofs
of Lemma~\ref{lemma: LB on Delta_k} and
Corollary~\ref{corollary: LB on homcount(F,G)}, where the nonnegative
difference
$\homcount{\Gr{F}}{\Gr{G}}-\homcount{\CoBG{p}{q}}{\Gr{G}}$
is expressed in \eqref{eq6: 26.07.25} as a sum of $\ell$ nonnegative
terms $\Delta_1,\ldots,\Delta_\ell$, each bounded below by
$\eta_{p,q}(\Gr{G})$ (see \eqref{eq: LB on Delta_k}).
This strategy does not necessarily maximize the resulting lower bound,
since both $\homcount{\CoBG{p}{q}}{\Gr{G}}$ and
$\eta_{p,q}(\Gr{G})$ depend on $p$ and $q$.

Numerical experiments indicate that this selection strategy is simple
and effective. For example, removing five edges from the star graph
$\CoBG{1}{10}$ yields a graph $\Gr{F}$ consisting of the star
$\CoBG{1}{5}$ and five isolated vertices.
This graph is a spanning subgraph of each of the nonisomorphic
complete bipartite graphs
$\CoBG{1}{10},\CoBG{2}{9},\ldots,\CoBG{5}{6}$,
which yield different lower bounds on $\homcount{\Gr{F}}{\Gr{G}}$.
Our strategy selects $\CoBG{1}{10}$, for which the lower bound
in \eqref{eq1: 27.07.25} is attained with equality.}
\end{remark}

Combining Corollary~\ref{corollary: LB on homcount(F,G)} with the earlier results in Propositions~\ref{prop: exact counting: CoBG --> BG},
\ref{proposition: LB and exact hom(CoBG,G)}, and \ref{prop.: refined IT-LB} gives the following lower bound on the
number of homomorphisms when the source and target graphs are arbitrary bipartite graphs (with no multiple edges).
\begin{proposition}
\label{proposition: LB on hom(BG, BG)}
{\em Let $\Gr{F}$ and $\Gr{G}$ be bipartite graphs, with $\Gr{G}$ nonempty, and let
\begin{enumerate}[(1)]
\item the partite sets of $\Gr{G}$ be denoted by $\set{U}$ and $\set{V}$, with respective sizes $n_1$ and $n_2$;
\item the edge density of $\Gr{G}$ be denoted by $\delta$, as defined in \eqref{eq: edge density in bipartite graph};
\item the degree profiles of $\Gr{G}$ with respect to the partite sets be denoted by ${\bf{d}}^{(\set{U})}(\Gr{G})$
and ${\bf{d}}^{(\set{V})}(\Gr{G})$, as given in \eqref{eq1: degree profile U} and \eqref{eq1: degree profile V}, respectively;
\item $\Gr{F}$ be the disjoint union of $k_1 \geq 0$ stars $\CoBG{1}{d_1}, \ldots, \CoBG{1}{d_{k_1}}$, where $d_i \geq 0$
(possibly including isolated vertices and isolated edges), together with $k_2 \geq 0$ additional connected bipartite graphs
$\Gr{B}_1, \ldots, \Gr{B}_{k_2}$;
\item for each $j \in \OneTo{k_2}$, the component $\Gr{B}_j$ of $\Gr{F}$ be a bipartite graph with partite sets of sizes $p_j$ and $q_j$,
where $p_j, q_j \geq 2$.
\end{enumerate}
Then,
\begin{align}
\label{eq2: 06.10.26}
\homcount{\Gr{F}}{\Gr{G}} \geq \prod_{i=1}^{k_1} \Biggl( \sum_{v \in \V{\Gr{G}}} d_{\Gr{G}}(v)^{d_i} \Biggr) \cdot \prod_{j=1}^{k_2} \mu_j,
\end{align}
where the nonnegative integers $\mu_j$, for $j \in \OneTo{k_2}$, are defined as follows. For each $j$, choose $b_j$ to be either the right-hand side of
\eqref{eq: exact counting: CoBG --> BG} or the maximum of the ceilings of the right-hand sides of \eqref{eq1: 01.07.25} and
\eqref{eq3: 16.07.2025}, with all expressions evaluated at $(p,q)=(p_j,q_j)$. Set
\begin{equation}
\label{eq1: 06.10.26}
\mu_j \triangleq b_j
+ \bigl(p_jq_j-\card{\E{\Gr{B}_j}}\bigr) \, \eta_{p_j,q_j}(\Gr{G}),
\end{equation}
where $\eta_{p_j,q_j}(\Gr{G})$ on the right-hand side of \eqref{eq1: 06.10.26} is the nonnegative integer defined in \eqref{eq6: 24.07.25}.
Empty products are understood to equal~$1$, and the convention $0^0=1$ is used in the factors corresponding to isolated vertices.}
\end{proposition}
\begin{proof}
Since homomorphism counts are multiplicative over disjoint unions (see \eqref{eq1: homomorphism numbers}), we have 
\begin{align}
\label{eq1: 24.07.25}
\homcount{\Gr{F}}{\Gr{G}}
&= \prod_{i=1}^{k_1} \homcount{\CoBG{1}{d_i}}{\Gr{G}}
\; \prod_{j=1}^{k_2} \homcount{\Gr{B}_j}{\Gr{G}} \\
\label{eq2: 24.07.25}
&= \prod_{i=1}^{k_1}
\Biggl( \sum_{v \in \V{\Gr{G}}} d_{\Gr{G}}(v)^{d_i} \Biggr)
\; \prod_{j=1}^{k_2} \homcount{\Gr{B}_j}{\Gr{G}}.
\end{align}
The second equality follows from \eqref{eq6: 09.07.25}
when $d_i \geq 1$. When $d_i=0$, the graph $\CoBG{1}{0}$
is an isolated vertex, whose homomorphism count is
$\card{\V{\Gr{G}}}$; the same formula therefore holds 
in this case as well.

For each $j \in \OneTo{k_2}$, the definition of $b_j$ and
Propositions~\ref{prop: exact counting: CoBG --> BG},
\ref{proposition: LB and exact hom(CoBG,G)}, and
\ref{prop.: refined IT-LB} give
\begin{equation}
\label{eq3: 06.10.26}
b_j \leq \homcount{\CoBG{p_j}{q_j}}{\Gr{G}}.
\end{equation}
Here, taking ceilings preserves the lower bounds because
homomorphism counts are integers.

If $\Gr{B}_j$ is complete bipartite, then
$\mu_j=b_j \leq \homcount{\Gr{B}_j}{\Gr{G}}$.
Otherwise, $\Gr{B}_j$ is a spanning subgraph of
$\CoBG{p_j}{q_j}$ obtained by removing
$p_jq_j-\card{\E{\Gr{B}_j}}$ edges. Hence, 
\begin{align}
\label{eq4a: 24.07.25}
\homcount{\Gr{B}_j}{\Gr{G}}
&\geq \homcount{\CoBG{p_j}{q_j}}{\Gr{G}}
+ \bigl(p_jq_j-\card{\E{\Gr{B}_j}}\bigr)
\, \eta_{p_j,q_j}(\Gr{G}) \\
\label{eq4b: 24.07.25}
&\geq b_j
+ \bigl(p_jq_j-\card{\E{\Gr{B}_j}}\bigr)
\, \eta_{p_j,q_j}(\Gr{G}) \\
\label{eq4c: 24.07.25}
&= \mu_j,
\end{align}
where \eqref{eq4a: 24.07.25}--\eqref{eq4c: 24.07.25} hold by 
Corollary~\ref{corollary: LB on homcount(F,G)}, \eqref{eq3: 06.10.26},
and \eqref{eq1: 06.10.26}, respectively. 
Substituting these nonnegative lower bounds into
\eqref{eq2: 24.07.25} proves \eqref{eq2: 06.10.26}.
\end{proof}

\begin{remark}
{\em For each $j \in \OneTo{k_2}$, the exact expression in
\eqref{eq: exact counting: CoBG --> BG} may be used to calculate
$b_j$ when the required bipartite clique counts can be computed
efficiently. Otherwise, $b_j$ may be taken as the maximum of
the ceilings of the right-hand sides of \eqref{eq1: 01.07.25}
and \eqref{eq3: 16.07.2025}, which can be evaluated from the
edge density and degree profiles of $\Gr{G}$. All expressions
are evaluated at $(p,q)=(p_j,q_j)$.}
\end{remark}

\section{Numerical Results}
\label{section: numerical results}

This section presents numerical results for the computationally tractable upper and lower bounds on the number of homomorphisms
between bipartite graphs established in Sections~\ref{section: exact expressions and comb. bounds}--\ref{section: boounds on homomorphism counts between bipartite graphs},
along with exact homomorphism counts in instances where exact computation is feasible.

We begin by examining the tightness of the proposed lower bounds on the number of homomorphisms from complete bipartite
graphs to bipartite target graphs, selected uniformly at random to have equal-sized partite sets of a given size, and
with a fixed edge density $\delta \in (0,1)$. The tightness of these lower bounds is studied as a function of the edge
density and the size of the partite sets. Throughout this section, the partite
sets are fixed and labeled, and uniform sampling at a prescribed edge density
$\delta$ means choosing uniformly among all subsets of $m=\delta n_1n_2$
of the $n_1n_2$ possible edges. This model is defined only when
$m$ is an integer. Accordingly, for fixed $n_1,n_2$, the admissible densities form
the grid
\[
\mathcal{A}_{n_1,n_2}
\triangleq \left\{\frac{m}{n_1n_2}:m=0,\ldots,n_1n_2\right\}.
\]
For each sampled density, a new target graph $\Gr{G}$ is drawn independently
of the graphs drawn at the other densities. The same realization is used
to compute all the compared lower bounds and, when feasible, the exact count.

The compared lower bounds include those from Propositions~\ref{proposition: LB and exact hom(CoBG,G)}, \ref{prop.: IT-LB},
and~\ref{prop.: refined IT-LB}, as well as the lower bound in \eqref{eq2: 21.05.2025}, which follows from Sidorenko's inequality.
For small complete bipartite source graphs, such as $\CoBG{3}{3}$, the number of homomorphisms is computed exactly using
Proposition~\ref{prop: exact counting: CoBG --> BG}. For larger source graphs such as $\CoBG{10}{10}$, exact computation was
not performed for the target sizes considered here, and we report only lower bounds (see Figures~\ref{figure 1} and~\ref{figure 2: two plots}, respectively).

The following observations are supported by the plots of Figures~\ref{figure 1} and~\ref{figure 2: two plots}:
\begin{enumerate}[1)]
\item The entropy-based lower bounds in Propositions~\ref{prop.: IT-LB} and~\ref{prop.: refined IT-LB} outperform
Sidorenko's lower bound in \eqref{eq2: 21.05.2025}. This is demonstrated in
Discussion~\ref{discussion: Comparison of the 1st IT LB to Sidorenko's lower bound}, which shows that the bound
in Proposition~\ref{prop.: IT-LB} (see the right-hand side of \eqref{eq4b: 16.09.2024}) improves that in
\eqref{eq2: 21.05.2025} when the target graph is bipartite. Consequently, the same
conclusion holds for the refined bound in Proposition~\ref{prop.: refined IT-LB}.
\item Both entropy-based lower bounds become relatively tight as the target
edge density approaches~1, with the source parameters $p,q$ fixed. To make
this statement precise, let $B_{\mathrm{IT}}(\Gr{G})$ and
$B_{\mathrm{IT,ref}}(\Gr{G})$ denote the first and refined entropy-based lower
bounds in Propositions~\ref{prop.: IT-LB} and~\ref{prop.: refined IT-LB}, respectively, and let
$M=\homcount{\CoBG{p}{q}}{\CoBG{n_1}{n_2}}$.
Their comparison with Sidorenko's lower bound gives
\begin{align}
\label{eq1: 04.08.25}
\delta^{pq}M
\leq B_{\mathrm{IT}}(\Gr{G})
\leq B_{\mathrm{IT,ref}}(\Gr{G})
\leq \homcount{\CoBG{p}{q}}{\Gr{G}}
\leq M.
\end{align}
For $\delta>0$, it follows that both ratios
$B_{\mathrm{IT}}(\Gr{G})/\homcount{\CoBG{p}{q}}{\Gr{G}}$
and $B_{\mathrm{IT,ref}}(\Gr{G})/\homcount{\CoBG{p}{q}}{\Gr{G}}$
lie in $[\delta^{pq},1]$. Hence both ratios tend to~1 along
any sequence of admissible densities tending to~1. For fixed target partite
sizes, the density grid is finite; at $\delta=1$, both bounds are exact.
\item The combinatorial lower bound in
Proposition~\ref{proposition: LB and exact hom(CoBG,G)} is reasonably tight
at the small densities displayed in Figure~\ref{figure 1}.
This is an empirical observation for the sampled graphs. The bound is
exact for $\CG{4}$-free bipartite targets, whose density is constrained by
Remark~\ref{remark: edge-density of C4-free bipartite graph}.
However, low density alone does not imply that a target is $\CG{4}$-free,
and this exactness result does not by itself establish tightness for
uniformly sampled sparse targets.
\item Applying the ceiling function to a positive real lower bound
$B\leq1$ gives $\lceil B\rceil=1$. This accounts for the unit-height
plateaus of the Sidorenko and first entropy-based bounds wherever their
unrounded values lie in $(0,1]$. It does not explain all the curves:
for a nonempty target, the combinatorial and refined entropy-based bounds
are at least $2\,\card{\E{\Gr{G}}}$ and therefore cannot have value~1.
At $\delta=0$, the homomorphism count from either complete bipartite source
is zero; zero values cannot be displayed on a logarithmic vertical axis.
\item Pictorially, Figures~\ref{figure 1} and~\ref{figure 2: two plots} may give the impression that the maximum of the combinatorial
lower bound in Proposition~\ref{proposition: LB and exact hom(CoBG,G)} and the first entropy-based lower bound in Proposition~\ref{prop.: IT-LB}
is greater than or equal to the refined entropy-based lower bound in Proposition~\ref{prop.: refined IT-LB}. However, this is not the case.
For instance, in the upper plot of Figure~\ref{figure 2: two plots}, at the sampled densities in the range $\delta \in [0.64,0.70]$, the refined entropy-based
lower bound in Proposition~\ref{prop.: refined IT-LB} is strictly larger than the maximum of the two other lower bounds from
Propositions~\ref{proposition: LB and exact hom(CoBG,G)} and~\ref{prop.: IT-LB}. Specifically, at $\delta = 0.64$ and $\delta = 0.70$,
the refined entropy-based lower bound in Proposition~\ref{prop.: refined IT-LB} is approximately $1.435 \cdot 10^{21}$ and $9.470 \cdot 10^{24}$,
respectively, while the corresponding maximum of the other two lower bounds is approximately $8.299 \cdot 10^{20}$ and $6.469 \cdot 10^{24}$, respectively.
\end{enumerate}

\begin{figure}[htbp]
  \centering
  \includegraphics[width=0.75\textwidth]{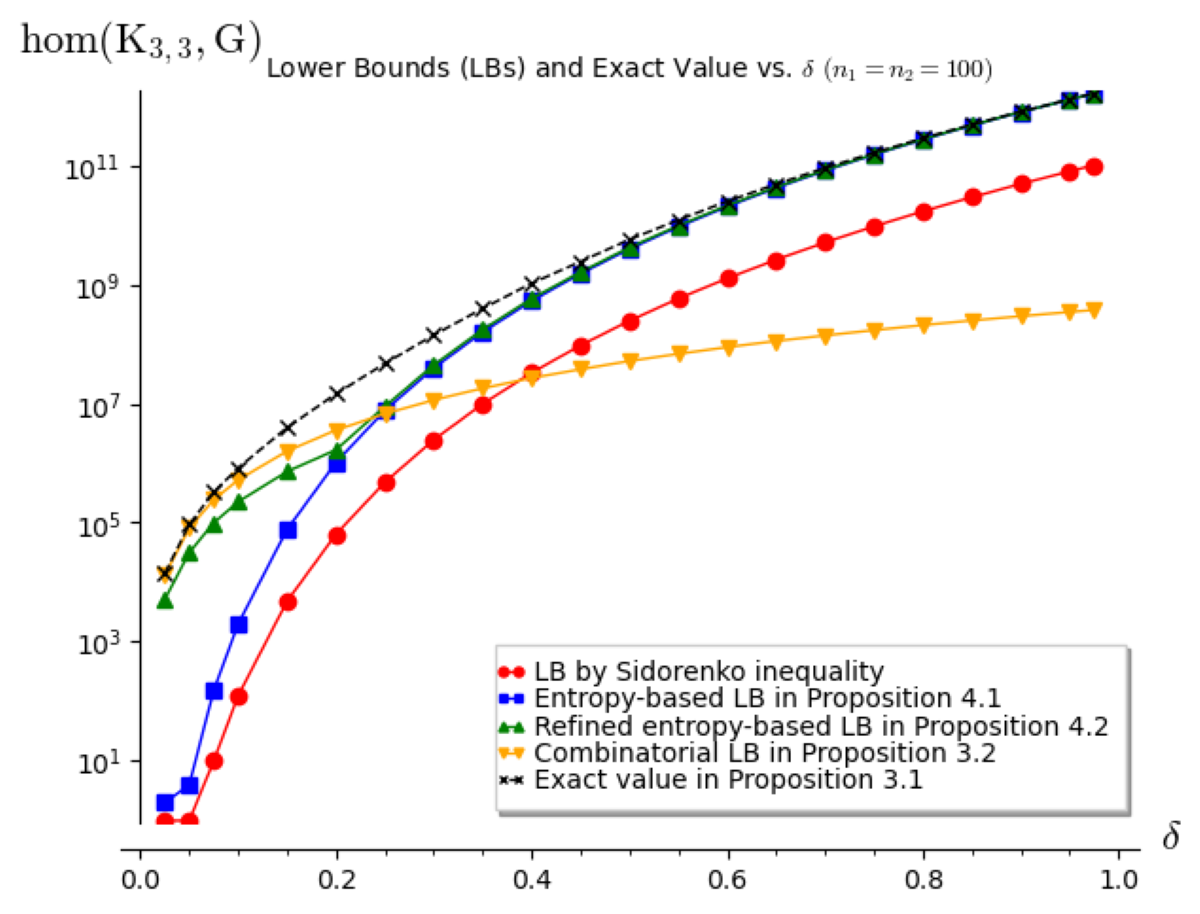}
  \caption{Exact values and lower bounds on the number of homomorphisms from the
  complete bipartite graph $\CoBG{3}{3}$ to bipartite graphs $\Gr{G}$ drawn
  uniformly at random with partite sets of sizes $n_1 = n_2 = 100$ and edge
  density $\delta\in\mathcal{A}_{n_1,n_2}$. For each value of $\delta$, a new random graph
  $\Gr{G}$ is generated independently. The exact values of $\homcount{\CoBG{3}{3}}{\Gr{G}}$
  and their lower bounds are plotted as a function of $\delta$.}
  \label{figure 1}
\end{figure}

\begin{figure}[htbp]
  \centering
  \begin{subfigure}[b]{0.8\textwidth}
    \includegraphics[width=\textwidth]{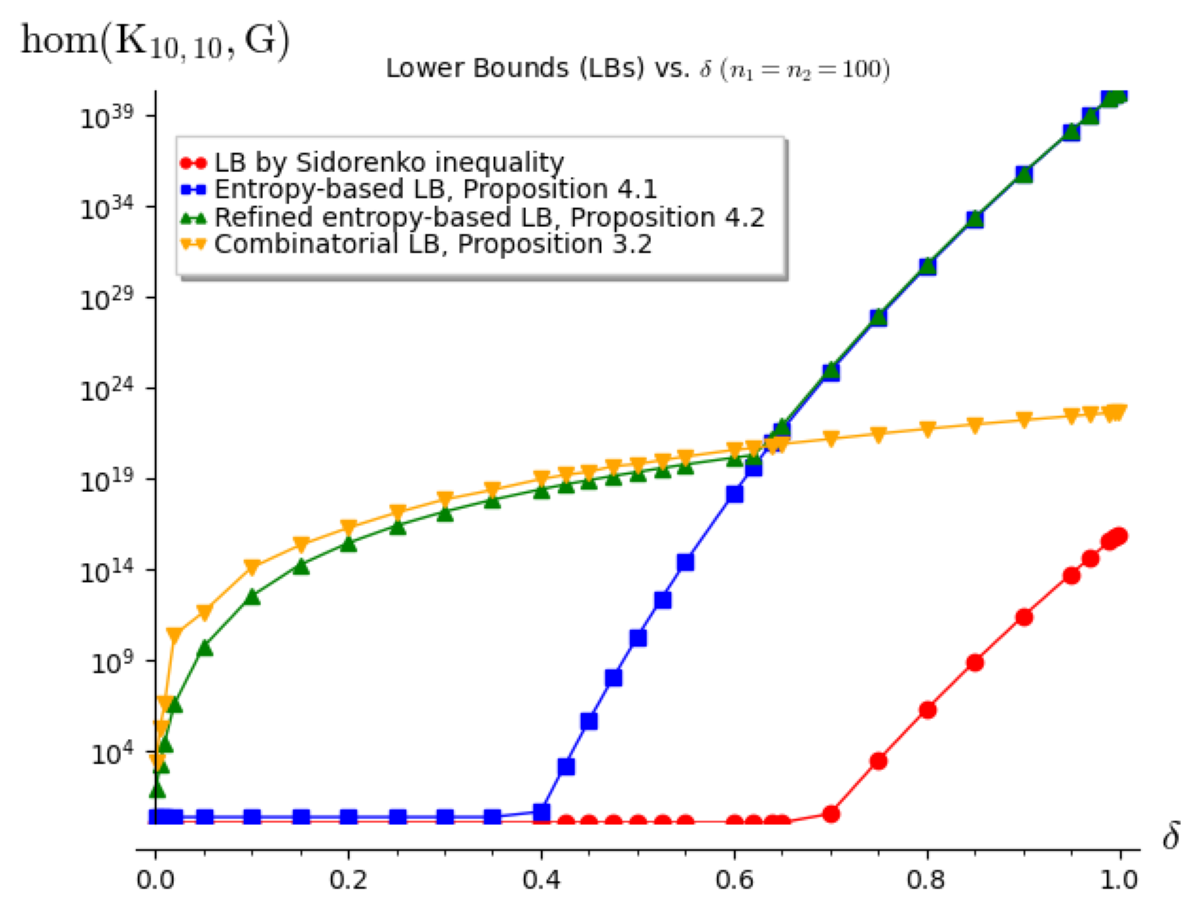}
    \label{fig:plot2a}
  \end{subfigure}
  \vskip 0.1cm
  \begin{subfigure}[b]{0.8\textwidth}
    \includegraphics[width=\textwidth]{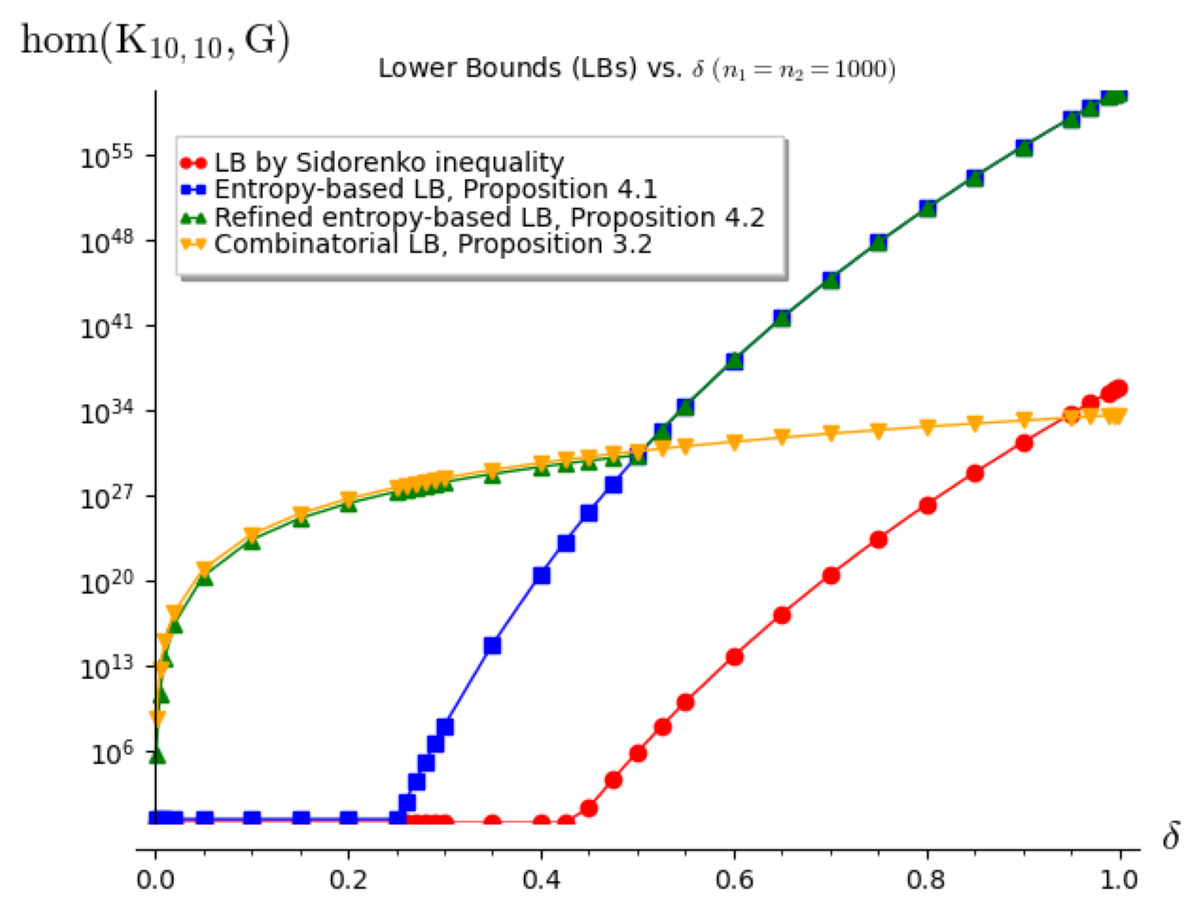}
    \label{fig:plot2b}
  \end{subfigure}
  \caption{Lower bounds on the number of homomorphisms from the complete bipartite graph
  $\CoBG{10}{10}$ to bipartite graphs $\Gr{G}$ drawn uniformly at random with partite sets
  of sizes $n_1$ and $n_2$, and edge density $\delta\in\mathcal{A}_{n_1,n_2}$. The upper plot corresponds
  to $n_1 = n_2 = 100$, and the lower plot to $n_1 = n_2 = 1000$. For each value of $\delta$,
  a new random graph $\Gr{G}$ is generated independently.
  The values of the lower bounds on $\homcount{\CoBG{10}{10}}{\Gr{G}}$ are plotted as a function
  of $\delta$.}
  \label{figure 2: two plots}
\end{figure}

We proceed to examine the upper and lower bounds on the number of homomorphisms from a non-complete
bipartite graph $\Gr{F}$ to a bipartite graph $\Gr{G}$. Figure~\ref{figure3} illustrates the setting
in which the source graph $\Gr{F}$, shown in the upper plot, is fixed, while the target graph $\Gr{G}$
is selected uniformly at random, subject to prescribed partite set sizes and a specified edge density
$\delta = \delta(\Gr{G})$. In this example, the source graph $\Gr{F}$ is a spanning subgraph of
$\CoBG{4}{4}$ with 8~edges (edge density $\tfrac{1}{2}$), and the target graph $\Gr{G}$ is taken to
have equal-sized partite sets ($n_1 = n_2 = \tfrac{n}{2}$), and the edge density is fixed to either
$\tfrac{1}{4}$ or $\tfrac{3}{4}$. This figure presents, as a function of the order $n$ of the target
graph $\Gr{G}$ and for these two edge densities, the exact value of $\homcount{\Gr{F}}{\Gr{G}}$ (when
computationally feasible) alongside the proposed easy-to-compute and analytical bounds. The upper three
curves in the lower plot correspond to $\delta = \tfrac{3}{4}$, while the lower three curves correspond
to $\delta = \tfrac{1}{4}$. For each value of $n$ and $\delta$, the plotted values (exact count, upper
bound, and lower bound) are computed as averages over 100~independent random instances. 
For each instance, the same target realization is used to compute all three quantities.

In Figure~\ref{figure3}, the curves for $\delta=\tfrac34$ lie above
the corresponding curves for $\delta=\tfrac14$. The ordering of the exact
counts is consistent with the following monotonicity property of their
expectations. Fix labeled target partite sets of sizes $n_1,n_2$, let
$\set{G}^{(\delta)}$ denote the graphs on these partite sets with exactly
$\delta n_1n_2$ edges, and define
\begin{align}
\label{eq1: 07.08.25}
g(\delta)\triangleq
\Expecwrt{\Gr{G}\sim U(\set{G}^{(\delta)})}
{\homcount{\Gr{F}}{\Gr{G}}},
\qquad \delta\in\mathcal{A}_{n_1,n_2}.
\end{align}
Here $\Gr{F}$ is fixed, and $g$ takes values in $[0,\infty)$.
To prove that $g$ is nondecreasing on its admissible grid, choose a
uniformly random ordering of the $n_1n_2$ possible target edges and let
$\Gr{G}_m$ contain the first $m$ edges. For each $m$, $\Gr{G}_m$
is uniform among all graphs with $m$ edges. Moreover,
$\Gr{G}_{m_1}$ is a spanning subgraph of $\Gr{G}_{m_2}$ whenever
$m_1\leq m_2$. Therefore
$\homcount{\Gr{F}}{\Gr{G}_{m_1}}
\leq\homcount{\Gr{F}}{\Gr{G}_{m_2}}$ almost surely, and taking
expectations proves the claim.
This argument concerns expected exact counts; it does not, by itself,
prove monotonicity of either bound or guarantee the ordering of finite
sample averages. The ordering of the plotted bounds is an empirical
observation.
\begin{figure}[htbp]
  \centering
  \begin{subfigure}[b]{0.35\textwidth}
    \includegraphics[width=\textwidth]{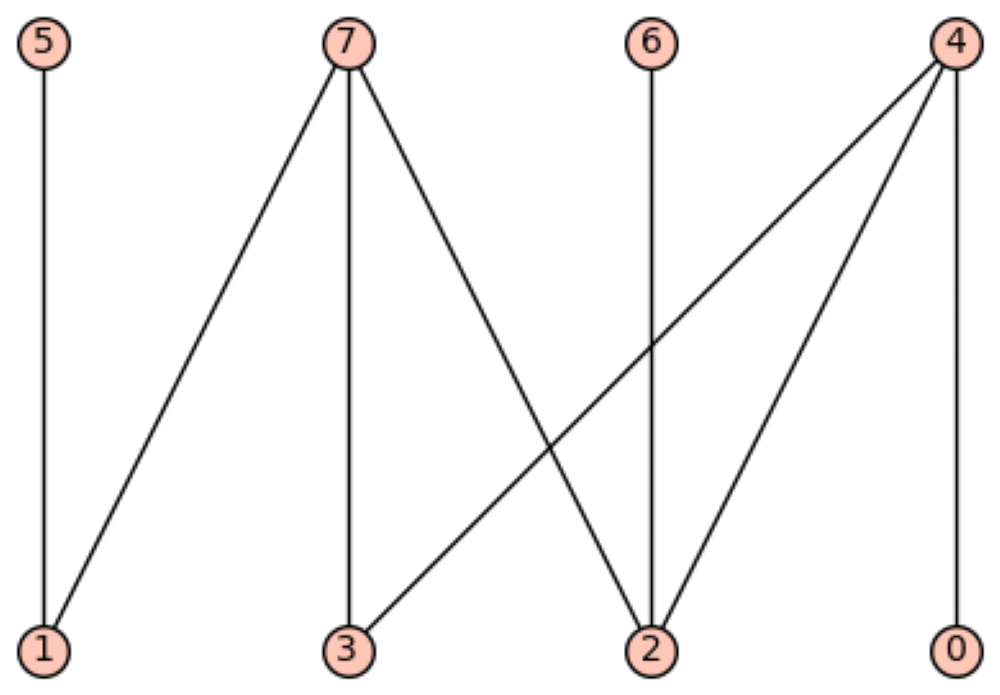}
    \label{fig:plot3_fixed_F}
  \end{subfigure}
  \vskip 0.2cm
  \begin{subfigure}[b]{1.0\textwidth}
    \includegraphics[width=\textwidth]{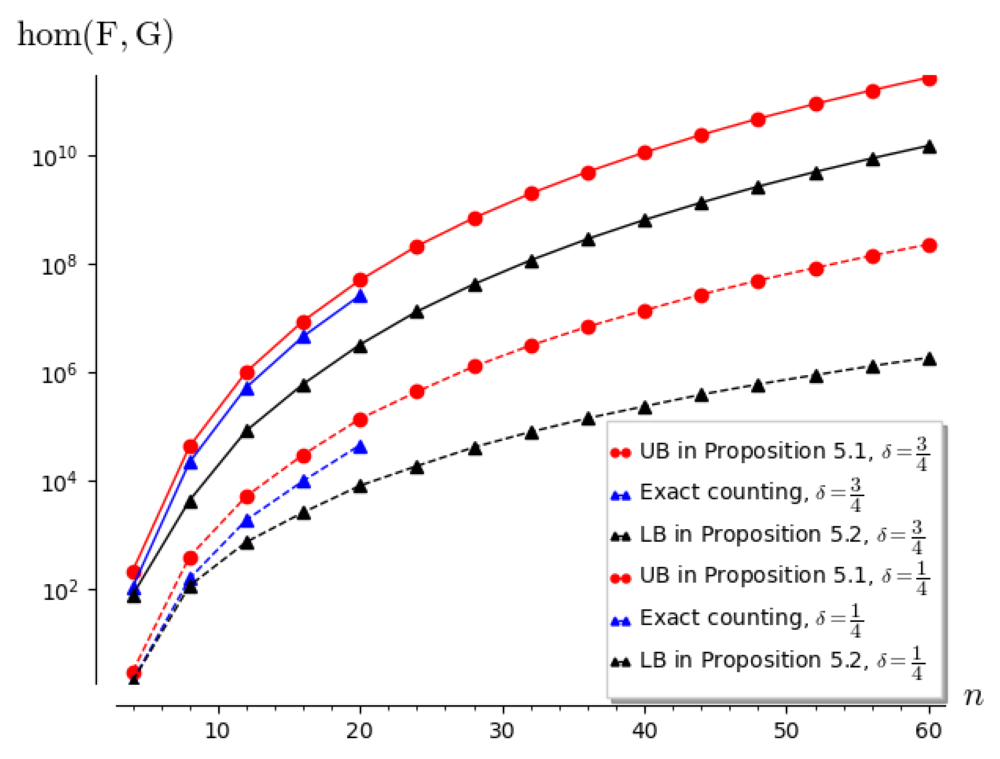}
    \label{fig:plot3_overall}
  \end{subfigure}
  \caption{Exact values and upper and lower bounds on the number of homomorphisms from a fixed noncomplete bipartite graph $\Gr{F}$,
  depicted in the upper plot, to a random bipartite graph $\Gr{G}$ on $n$ vertices, chosen uniformly at random among the bipartite graphs
  with equal-sized partite sets ($n_1 = n_2 = \tfrac{n}{2}$) and a fixed edge density $\delta = \delta(\Gr{G})$. The upper three curves
  in the lower plot correspond to $\delta = \tfrac{3}{4}$ (solid lines), while the lower three curves correspond to $\delta = \tfrac{1}{4}$
  (dashed lines). Each plotted value is an average over 100 independent target realizations, with the same samples used for all compared quantities.}
  \label{figure3}
\end{figure}

\begin{figure}[htbp]
  \centering
  \begin{subfigure}[b]{0.75\textwidth}
    \includegraphics[width=\textwidth]{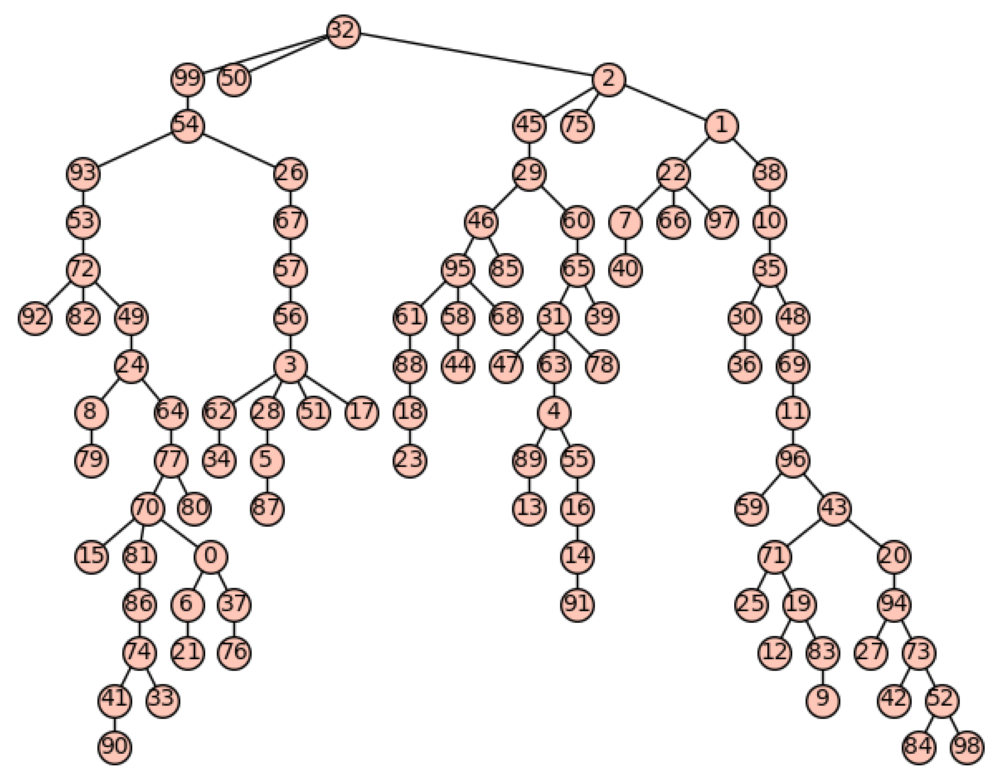}
    \label{fig:plot4-6_G}
  \end{subfigure}
  \begin{subfigure}[b]{0.8\textwidth}
    \includegraphics[width=\textwidth]{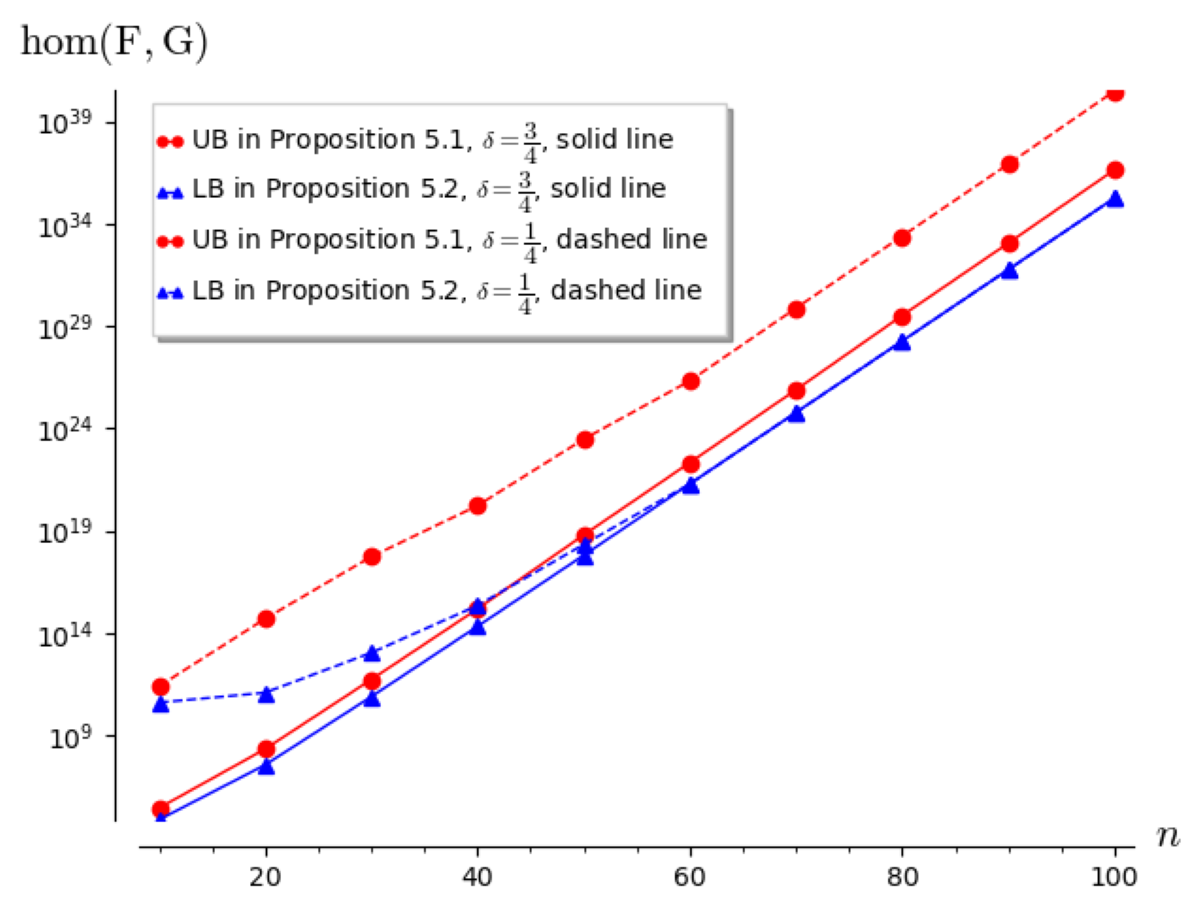}
    \label{fig:plot4}
  \end{subfigure}
  \caption{Average upper and lower bounds on the number of homomorphisms from a random bipartite graph $\Gr{F}$ to a fixed
  tree $\Gr{G}$ on 100 vertices, as depicted in the upper plot. For each experiment, the source graph $\Gr{F}$ is drawn
  uniformly at random among the bipartite graphs with $n$ vertices, equal-sized partite sets, and a fixed edge density
  $\delta = \delta(\Gr{F})$. The solid and dashed curves correspond to nominal densities $\delta = \tfrac{3}{4}$ and $\delta = \tfrac{1}{4}$,
  respectively. Each point averages 50 independent source realizations; the same samples are used for both bounds.}
  \label{figure4}
\end{figure}

Figure~\ref{figure4} illustrates the complementary setting in which the target graph $\Gr{G}$ is kept fixed,
while the source graph $\Gr{F}$ is selected uniformly at random, subject to prescribed partite set sizes and
a specified edge density $\delta \triangleq \delta(\Gr{F})$. It shows average upper and lower bounds on the
number of homomorphisms from a random bipartite graph $\Gr{F}$ to a fixed tree $\Gr{G}$. The target graph
$\Gr{G}$, depicted in the upper plot, is a fixed tree on 100~vertices. For each experiment, the source graph
$\Gr{F}$ is drawn uniformly at random from the set of bipartite graphs with $n$ vertices, equal-sized partite
sets, and a fixed edge density $\delta$. The two solid curves correspond to nominal density $\delta = \tfrac{3}{4}$ (upper and
lower bounds), while the two dashed curves correspond to nominal density $\delta = \tfrac{1}{4}$. Each point represents an
average over 50~independent random instances, with identical $\Gr{F}$ samples used to compute both bounds.
The horizontal axis represents the order $n$ of the source graph $\Gr{F}$, ranging from~10 to~100.

For equal-sized partite sets, $n$ must be even. Exact densities
$\tfrac14$ and $\tfrac34$ require $n$ to be divisible by~4.
For other even orders, these values cannot be realized exactly,
and the realized density must be specified as $m/(n^2/4)$,
where $m$ is the integer edge count actually sampled.
For $n=10$, the nominal edge densities $\tfrac14$ and $\tfrac34$
examined in Figure~\ref{figure4} correspond to realized densities
$\tfrac{6}{25}$ and $\tfrac{18}{25}$, respectively.
Each randomly selected bipartite graph $\Gr{F}$, with equal-sized
partite sets $n_1=n_2=5$, has exactly
$\bigl\lfloor \tfrac14 \cdot 5^2 \bigr\rfloor=6$ or
$\bigl\lfloor \tfrac34 \cdot 5^2 \bigr\rfloor=18$ edges, respectively.

In Figure~\ref{figure4}, the bounds at the larger nominal source density
lie below the corresponding bounds at the smaller nominal density.
For exact homomorphism counts, the expected ordering is reversed relative
to the random-target setting. Fix labeled source partite sets of sizes
$n_1,n_2$, write $\set{F}^{(\delta)}$ for the source graphs with exactly
$\delta n_1n_2$ edges, and define
\begin{align}
\label{eq2: 07.08.25}
f(\delta)\triangleq
\Expecwrt{\Gr{F}\sim U(\set{F}^{(\delta)})}
{\homcount{\Gr{F}}{\Gr{G}}},
\qquad \delta\in\mathcal{A}_{n_1,n_2},
\end{align}
where $\Gr{G}$ is fixed and $f$ takes values in $[0,\infty)$.
Choose a uniformly random ordering of all possible source edges and let
$\Gr{F}_m$ contain the first $m$ edges. Each $\Gr{F}_m$ has the
required uniform distribution. Since every homomorphism from
$\Gr{F}_{m_2}$ to $\Gr{G}$ is also a homomorphism from
$\Gr{F}_{m_1}$ to $\Gr{G}$ when $m_1\leq m_2$, taking expectations
shows that $f$ is nonincreasing. This conclusion concerns expected exact
counts and does not, by itself, establish monotonicity of the plotted bounds.
Figure~\ref{figure4} shows that the gap, on a logarithmic scale, between
the upper and lower bounds narrows as the edge density $\delta$ of
the source graph increases. This observed narrowing is consistent
with the exactness of the upper bound in
Proposition~\ref{proposition: UB on hom(BG, BG)} for complete bipartite
sources, although endpoint exactness alone does not establish
monotonic narrowing of the gap.
Furthermore, the lower-bound curves become visually indistinguishable
for large $n$. Below, we show that the relative separation
of the bounds $L_n(\delta;\Gr{G})$ tends to zero, consistently with
this observation, by combining Corollary~\ref{corollary: K_{p,q} --> tree},
Lemma~\ref{lemma: auxiliary result}, and Corollary~\ref{corollary: LB on homcount(F,G)}.
For each selection of a source graph $\Gr{F}$ from the set of
bipartite graphs on $n = v(\Gr{F})$ vertices with equal-sized partite
sets of size $\frac{n}{2}$ and edge density
$\delta = \delta(\Gr{F})$, the lower bound on
$\homcount{\Gr{F}}{\Gr{G}}$, where $\Gr{G}$ is a fixed tree on at
least three vertices and $n$ is even with $n\geq4$, is denoted by
$L_n(\delta;\Gr{G})$ and is given by
\begin{align}
\label{eq3: 07.08.25}
L_n(\delta;\Gr{G}) &= \homcount{\CoBG{\frac{n}{2}}{\frac{n}{2}}}{\Gr{G}} + \Bigl( \frac{n^2}{4} - \card{\E{\Gr{F}}} \Bigr) \, \eta_{\frac{n}{2}, \frac{n}{2}}(\Gr{G}) \\[0.1cm]
\label{eq4: 07.08.25}
&= \homcount{\CoBG{\frac{n}{2}}{\frac{n}{2}}}{\Gr{G}} + \tfrac14 \, \bigl(1-\delta(\Gr{F}) \bigr) \, \eta_{\frac{n}{2}, \frac{n}{2}}(\Gr{G}) \, n^2 \\[0.1cm]
\label{eq5: 07.08.25}
&= 2 \sum_{w \in \V{\Gr{G}}} d_{\Gr{G}}(w)^{\frac{n}{2}} - 2 \bigl(v(\Gr{G}) - 1) + \tfrac14 \, \bigl(1-\delta(\Gr{F}) \bigr) \, \eta_{\frac{n}{2}, \frac{n}{2}}(\Gr{G}) \, n^2,
\end{align}
where \eqref{eq3: 07.08.25} holds by Corollary~\ref{corollary: LB on homcount(F,G)} and the two partite sets of $\Gr{F}$ are of size $\frac{n}{2}$;
\eqref{eq4: 07.08.25} holds since the number of edges in $\Gr{F}$ equals $\tfrac14 n^2 \delta(\Gr{F})$,
and \eqref{eq5: 07.08.25} holds by applying Corollary~\ref{corollary: K_{p,q} --> tree} to provide a closed-form expression for the number of
homomorphisms from the complete bipartite graph $\CoBG{\frac{n}{2}}{\frac{n}{2}}$ to the tree $\Gr{G}$. We next show that for the tree $\Gr{G}$,
the nonnegative constant $\eta_{\frac{n}{2}, \frac{n}{2}}(\Gr{G})$ (see \eqref{eq6: 24.07.25}) is bounded above by a constant that does not depend
on the order $n$ of the graph $\Gr{F}$. Indeed, by the definition of the set $\set{D}$ in \eqref{eq5: 24.07.25}, for every $(u,v)\in\set{D}$
and $w \in \set{N}_{\Gr{G}}(u)$, since by definition $\{u,v\} \not\in \E{\Gr{G}}$, we have that $v \neq w$. Moreover, since $\Gr{G}$ is a tree
(i.e., a connected acyclic graph), the condition $\set{N}_{\Gr{G}}(v) \cap \set{N}_{\Gr{G}}(w) \neq \es$ implies
that $\bigcard{\set{N}_{\Gr{G}}(v) \cap \set{N}_{\Gr{G}}(w)} = 1$, as otherwise $\Gr{G}$ would have a 4-cycle. Consequently, by
\eqref{eq5: 24.07.25}, for a tree $\Gr{G}$ on $v(\Gr{G})$ vertices
\begin{align}
0 &\leq \eta_{\frac{n}{2}, \frac{n}{2}}(\Gr{G}) \nonumber \\
\label{eq6: 07.08.25}
&\leq \sum_{(u,v) \in \set{D}} d_{\Gr{G}}(u) \\
\label{eq7: 07.08.25}
&\leq v(\Gr{G}) \, \sum_{u \in \V{\Gr{G}}} d_{\Gr{G}}(u) \\
\label{eq8: 07.08.25}
&= 2 \, v(\Gr{G}) \, e(\Gr{G}) \\
\label{eq9: 07.08.25}
&= 2 \, v(\Gr{G}) \, \bigl(v(\Gr{G}) - 1 \bigr),
\end{align}
where \eqref{eq6: 07.08.25} holds by \eqref{eq5: 24.07.25} and \eqref{eq6: 24.07.25}, using the uniqueness of the
intersection of neighborhoods of distinct vertices in trees, and \eqref{eq9: 07.08.25} relies on the fact that a
tree has exactly one fewer edge than vertices.
The upper bound in \eqref{eq9: 07.08.25} depends only on the order of $\Gr{G}$, not on $n$.
For the fixed tree $\Gr{G}$ and even $n:=v(\Gr{F})\geq4$, Jensen's inequality gives
\begin{align}
\sum_{w \in \V{\Gr{G}}} d_{\Gr{G}}(w)^{\frac{n}{2}} & \geq
v(\Gr{G}) \, \left( \frac{1}{v(\Gr{G})} \sum_{w \in \V{\Gr{G}}} d_{\Gr{G}}(w) \right)^{\frac{n}{2}} \nonumber \\
&= v(\Gr{G}) \, \left(\frac{2 \, e(\Gr{G})}{v(\Gr{G})}\right)^{\frac{n}{2}} \nonumber \\
\label{eq: 20.05.2026}
&= v(\Gr{G}) \, \left(\frac{2 \bigl(v(\Gr{G}) - 1 \bigr)}{v(\Gr{G})} \right)^{\frac{n}{2}}.
\end{align}
Combining \eqref{eq5: 07.08.25}--\eqref{eq: 20.05.2026}, it follows that, for a fixed tree $\Gr{G}$ on at least three
vertices, the density-dependent term in \eqref{eq5: 07.08.25} grows at most quadratically in $n$,
whereas the density-independent term grows at least exponentially in $n$. Consequently, the relative contribution of
the density-dependent term tends to zero as $n \to \infty$. More precisely, writing
\[
A_n=2\sum_{w\in\V{\Gr{G}}}d_{\Gr{G}}(w)^{n/2}
      -2\bigl(v(\Gr{G})-1\bigr),
\qquad
c_n=\tfrac14 n^2 \, \eta_{n/2, \, n/2}(\Gr{G}),
\]
gives $L_n(\delta;\Gr{G})=A_n+(1-\delta)c_n$ and
$\frac{c_n}{A_n} \to 0$. Thus, for any two sequences of admissible realized
densities $\delta_1(n),\delta_2(n)\in[0,1]$,
\[
\frac{L_n(\delta_1(n);\Gr{G})}{L_n(\delta_2(n);\Gr{G})} \underset{n \to \infty}{\longrightarrow} 1.
\]
Their difference on a logarithmic scale consequently tends to zero. For the exact densities $\tfrac14$ 
and $\tfrac34$, this conclusion holds along orders $n$ divisible by~4. The absolute difference of
the bounds need not tend to zero. This asymptotic behavior of $L_n(\delta;\Gr{G})$ is consistent with
the observed convergence of the lower-bound curves on the logarithmic scale in Figure~\ref{figure4}.
In this figure, the target graph $\Gr{G}$ is a fixed tree on~100 vertices, and the source graph $\Gr{F}$ 
is sampled uniformly from the bipartite graphs with $n$ vertices, equal-sized partite sets, and prescribed 
realized edge density~$\delta$.

The numerical experiments illustrate the effectiveness of the computationally tractable bounds derived in this paper.
For complete bipartite source graphs, the entropy-based bounds improve on Sidorenko's lower bound, as established in
Discussion~\ref{discussion: Comparison of the 1st IT LB to Sidorenko's lower bound},
and become tight as the target's edge density approaches~1. The combinatorial bound remains competitive at low target
densities and is exact for $\CG{4}$-free targets. For non-complete bipartite source graphs, the plotted bounds
exhibit trends consistent with the monotonicity of the expected homomorphism counts and permit comparison with exact counts in the
instances where these are available. The experiments also illustrate the diminishing relative effect of the source's
edge density on the lower bound when the target is a fixed tree and the source order increases.

\section*{Acknowledgement}
The author thanks the reviewers for their detailed and constructive comments, which improved the presentation of the paper, and the editors 
for inviting him to contribute to the Special Issue on selected papers of EuroComb~2025 and for handling the manuscript.



\begin{thebibliography}{99}
\small

\bibitem{FederVardi98}
T. Feder and M. Y. Vardi, ``The computational structure of monotone monadic SNP and constraint satisfaction:
a study through datalog and group theory,'' {\em SIAM Journal on Computing}, vol.~28, no.~1, pp.~57--104, January 1998.
\url{https://doi.org/10.1137/S0097539794266766}

\bibitem{HellN26}
P. Hell and J. Ne\v{s}et\v{r}il, {\em Graphs and Homomorphisms}, second edition, Oxford University Press, 2026.
\url{https://doi.org/10.1093/oso/9780198708704.001.0001}

\bibitem{HellN21}
P. Hell and J. Ne\v{s}et\v{r}il, ``Graphs and homomorphisms,'' {\em Topics in Algorithmic Graph Theory},
{\em Cambridge University Press} (edited by L. W. Beineke, M. C. Golumbic and R.~J.~Wilson), pp.~262--293, 2021.
\url{https://doi.org/10.1017/9781108592376}

\bibitem{Lovasz12}
L. Lov\'{a}sz, {\em Large Networks and Graph Limits}, American Mathematical Society, 2012.
\url{https://doi.org/10.1090/coll/060}

\bibitem{Zhao17}
Y. Zhao, ``Extremal regular graphs: independent sets and graph homomorphisms,'' {\em The American Mathematical Monthly},
vol.~124, no.~9, pp.~ 827--843, November 2017. \url{https://doi.org/10.4169/amer.math.monthly.124.9.827}

\bibitem{GodsilRRSV19}
C. Godsil, D. E. Roberson, B. Rooney, R. \v{S}\'{a}mal, and A. Varvitsiotis, ``Graph homomorphisms via vector colorings,''
{\em European Journal of Combinatorics}, vol.~79, pp.~246--261, June 2019.
\url{https://doi.org/10.1016/j.ejc.2019.04.001}

\bibitem{AlbertsonC85}
M. O. Albertson and K. L. Collins, ``Homomorphisms of 3-chromatic graphs,'' {\em Discrete Mathematics}, vol.~54,
no.~2, pp.~127--132, April 1985. \url{https://doi.org/10.1016/0012-365X(85)90073-1}

\bibitem{NaserasrSZ21}
R. Naserasr, E. Sopena, and T. Zaslavsky, ``Homomorphisms of signed graphs: An update,'' {\em European Journal of Combinatorics},
vol.~91, pp.~1--20, January 2021. \url{https://doi.org/10.1016/j.ejc.2020.103222}

\bibitem{HellN90}
P. Hell and J. Ne\v{s}etril, ``On the complexity of H-coloring,'' {\em Journal of Combinatorial Theory, Series B}, vol.~48, no.~1,
pp.~92--110, February 1990. \url{https://doi.org/10.1016/0095-8956(90)90132-J}

\bibitem{DyerG00}
M. E. Dyer and C. S. Greenhill, ``The complexity of counting graph homomorphisms,'' {\em Random
Structure and Algorithms}, vol.~17, pp.~260--289, 2000.

\bibitem{DyerG04}
M. E. Dyer and C. S. Greenhill, ``Corrigendum: The complexity of counting graph homomorphisms,'' {\em Random Structures
and Algorithms}, vol~25, no.~3, pp.~346--352, October 2004. \url{https://doi.org/10.1002/rsa.20036}

\bibitem{Hell06}
P. Hell, ``From graph colouring to constraint satisfaction: there and back again,''
{\em Topics in Discrete Mathematics}, pp.~315--371, Springer, 2006.
\url{https://doi.org/10.1007/3-540-33700-8_20}

\bibitem{CaiCL15}
J. Y. Cai, X. Chen, and P. Lu, ``Complexity dichotomies for counting graph homomorphisms,'' {\em Encyclopedia of Algorithms},
pp.~1--5, Springer, 2015. \url{https://link.springer.com/rwe/10.1007/978-3-642-27848-8_747-1}

\bibitem{CurticapeanDM17}
R. Curticapean, H. DellHolger, and D. Marx, ``Homomorphisms are a good basis for counting small subgraphs,''
{\em Proceedings of the 49th Annual ACM SIGACT Symposium on Theory of Computing (STOC 2017)}, pp.~210--223,
Montreal, Canada, June 2017. \url{http://dx.doi.org/10.1145/3055399.3055502}

\bibitem{Lovasz67}
L. Lov\'{a}sz, ``Operations with structures,'' {\em Acta Mathematica Academiae Scientiarum Hungarica},
vol.~18, no.~3--4, pp.~321--328, 1967. \url{https://doi.org/10.1007/BF02280291}

\bibitem{Sernau18}
L. Sernau, ``Graph operations and upper bounds on graph homomorphism counts,'' {\em Journal of Graph Theory},
vol.~87, no.~2, pp.~149--163, February 2018. \url{https://doi.org/10.1002/jgt.22148}

\bibitem{GarijoGN11}
D. Garijo, A. Goodall, and J. Ne\v{s}et\v{r}il, ``Distinguishing graphs by their left and right homomorphism profiles,''
{\em European Journal of Combinatorics}, vol.~32, no.~7, pp.~1025--1053, October 2011. \url{https://doi.org/10.1016/j.ejc.2011.03.012}

\bibitem{Dvorak10}
Z. Dvo\v{r}\'{a}k, ``On recognizing graphs by numbers of homomorphisms,'' {\em Journal of Graph Theory}, vol.~64, no.~4,
pp.~330--342, June 2010. \url{https://doi.org/10.1002/jgt.20461}

\bibitem{Borgs06}
C. Borgs, J. Chayes, L. Lov\'{a}sz, V. T. S\'{o}s, and K. Vesztergombi, ``Counting graph homomorphisms,''
{\em Topics in Discrete Mathematics}, pp.~315--371, Springer, 2006.
\url{https://doi.org/10.1007/3-540-33700-8_18}

\bibitem{KoppartyR11}
S. Kopparty and B. Rossman, ``The homomorphism domination exponent,'' {\em European Journal of Combinatorics},
vol.~32, no.~7, pp.~1097--1114, October 2011. \url{https://doi.org/10.1016/j.ejc.2011.03.009}

\bibitem{GarijoNR09}
D. Garijo, J. Ne\v{s}et\v{r}il, M. P. Revuelta, ``Homomorphisms and polynomial invariants of graphs,''
{\em European Journal of Combinatorics}, vol.~30, no.~7, pp.~1659--1675, October 2009.
\url{https://doi.org/10.1016/j.ejc.2009.03.016}

\bibitem{BrightwellW99}
G. Brightwell and P. Winkler, ``Graph homomorphisms and phase transitions,'' {\em Journal
of Combinatorial Theory, Series B}, vol.~77, no.~2, pp.~221--262, November 1999.
\url{https://doi.org/10.1006/jctb.1999.1899}

\bibitem{Kahn02}
J. Kahn, ``Entropy, independent sets and antichains: a new approach to Dedekind's problem,''
{\em Proceedings of the American Mathematical Society}, vol.~130, no.~2, pp.~371--378, June 2001.
\url{https://doi.org/10.1090/S0002-9939-01-06058-0}

\bibitem{GalvinT04}
D. Galvin and P. Tetali, ``On weighted graph homomorphisms,'' {\em DIMACS Series in Discrete Mathematics
and Theoretical Computer Science}, Graphs, Morphisms and Statistical Physics (J. Ne\v{s}etril and P. Winkler Editors),
American Mathematical Society, vol.~63, pp.~97--104, 2004. \url{https://doi.org/10.48550/arXiv.1206.3160}

\bibitem{CsikvariRS22}
P. Csikv\'{a}ri, N. Ruozzi, and S. Shams, ``Markov random fields, homomorphism counting, and Sidorenko’s
conjecture,'' {\em IEEE Transactions on Information Theory}, vol.~68, no.~9, pp.~6052--6062, September~2022.
\url{https://doi.org/10.1109/TIT.2022.3169487}

\bibitem{ShamsRC19}
S. Shams, N. Ruozzi, and P. Csikv\'{a}ri, ``Counting homomorphisms in bipartite graphs,'' {\em Proceedings
of the 2019 IEEE International Symposium on Information Theory}, pp.~1487--1491, Paris, France, July 2019.
\url{https://doi.org/10.1109/ISIT.2019.8849389}

\bibitem{BaoJBCL25}
L. Bao, E. Jin, M. M. Bronstein, I. I. Ceylan, and M. Lanzinger, ``Homomorphism counts as structural encodings
for graph learning,'' {\em Proceedings of the Thirteenth International Conference on Learning Representations
(ICLR 2025)}, pp.~1--29, Singapore, April 2025. \url{https://doi.org/10.48550/arXiv.2410.18676}

\bibitem{Sidorenko94}
A. Sidorenko, ``A partially ordered set of functionals corresponding to graphs,'' {\em Discrete Mathematics}, vol.~131,
no.~1--3, pp.~263--277, August 1994. \url{https://doi.org/10.1016/0012-365X(94)90388-3}

\bibitem{CsikvariL14}
P. Csikv\'{a}ri and Z. Lin, ``Graph homomorphisms between trees,'' {\em The Electronic Journal of Combinatorics},
vol.~21, no.~4, article P4.9, pp.~1--38, October 2014. \url{https://doi.org/10.37236/4096}

\bibitem{LevinP17}
D. A. Levin and Y. Peres, ``Counting walks and graph homomorphisms via Markov chains and importance sampling,''
{\em The American Mathematical Monthly}, vol.~124, no.~7, pp.~637--641, August-September 2017.
\url{https://doi.org/10.4169/amer.math.monthly.124.7.637}

\bibitem{GopalanSW16}
P. Gopalan, R. A. Servedio, and A. Wigderson, ``Degree and sensitivity: tails of two distributions,'' {\em Proceedings
of the 31st Conference on Computational Complexity}, article no.~13, pp.~1--23, May 2016.
\url{https://drops.dagstuhl.de/storage/00lipics/lipics-vol050-ccc2016/LIPIcs.CCC.2016.13/LIPIcs.CCC.2016.13.pdf}

\bibitem{Zhao23}
Y. Zhao, {\em Graph Theory and Additive Combinatorics: Exploring Structure and Randomness}, Cambridge University
Press, 2023. \url{https://doi.org/10.1017/9781009310956}

\bibitem{AignerZ18}
M. Aigner and G. M. Ziegler, {\em Proofs from THE BOOK}, Sixth Edition,
Springer, Berlin, 2018. \url{https://doi.org/10.1007/978-3-662-57265-8}

\bibitem{Jukna11}
S. Jukna, {\em Extremal Combinatorics with Applications in Computer Science}, second edition,
Springer, 2011. \url{https://doi.org/10.1007/978-3-642-17364-6}

\bibitem{Szegedy15a}
B. Szegedy, ``An information theoretic approach to Sidorenko’s conjecture,'' preprint, January 2015.
\url{https://arxiv.org/abs/1406.6738v3}

\bibitem{Szegedy15b}
B. Szegedy, ``Sparse graph limits, entropy maximization and transitive graphs,'' preprint, April 2015.
\url{https://arxiv.org/abs/1504.00858v1}

\bibitem{ChungGFS86}
F. R. K. Chung, L. R. Graham, P. Frankl, and J. B. Shearer, ``Some intersection
theorems for ordered sets and graphs," {\em Journal of Combinatorial Theory,
Series~A}, vol.~43, no.~1, pp.~23--37, 1986. \url{https://doi.org/10.1016/0097-3165(86)90019-1}

\bibitem{ConlonKLL18}
D. Conlon, J. H. Kim, C. Lee, and J. Lee, ``Some advances on Sidorenko’s conjecture,'' {\em Journal of the London
Mathematical Society}, vol.~98, no.~3, pp.~593--608, December 2018. \url{https://doi.org/10.1112/jlms.12142}

\bibitem{FriedgutK98}
E. Friedgut and J. Kahn, ``On the number of copies of one hypergraph in another,'' {\em Israel Journal of Mathematics},
vol.~105, pp.~251--256, 1998. \url{https://doi.org/10.1007/BF02780332}

\bibitem{Friedgut04}
E. Friedgut, ``Hypergraphs, entropy, and inequalities,'' {\em The American Mathematical Monthly},
vol.~111, no.~9, pp.~749--760, November 2004. \url{https://doi.org/10.2307/4145187}

\bibitem{HatamiJS18}
H. Hatami, S. Janson, and B. Szegedy, ``Graph properties, graph limits, and entropy,''
{\em Journal of Graph Theory}, vol.~87, no.~2, pp.~208--229, February 2018.
\url{https://doi.org/10.1002/jgt.22152}

\bibitem{Galvin14}
D. Galvin, ``Three tutorial lectures on entropy and counting,'' {\em Proceedings of the 1st Lake
Michigan Workshop on Combinatorics and Graph Theory}, Kalamazoo, MI, USA, March 2014.
\url{https://doi.org/10.48550/arXiv.1406.7872}

\bibitem{Kahn01}
J. Kahn, ``An entropy approach to the hard-core model on bipartite graphs,'' {\em Combinatorics,
Probability and Computing}, vol.~10, no.~3, pp.~219--237, May 2001.
\url{https://doi.org/10.1017/S0963548301004631}

\bibitem{MadimanT_IT10}
M. Madiman and P. Tetali, ``Information inequalities for joint distributions,
interpretations and applications," {\em IEEE Transactions on Information Theory},
vol.~56, no.~6, pp.~2699--2713, June 2010.
\url{https://doi.org/10.1109/TIT.2010.2046253}

\bibitem{Radhakrishnan97}
J. Radhakrishnan, ``An entropy proof of Bregman's theorem,'' {\em Journal of Combinatorial Theory,
Series~A}, Elsevier Science, vol.~77, no.~1, pp.~161--164, January 1997.
\url{https://doi.org/10.1006/jcta.1996.2727}

\bibitem{Radhakrishnan01}
J. Radhakrishnan, ``Entropy and counting,'' {\em Proceedings of the IIT Kharagpur, Golden Jubilee Volume
on Computational Mathematics, Modelling and Algorithms}, Narosa Publishers, New Delhi, India, pp.~1--25, 2001.
\url{https://www.tcs.tifr.res.in/~jaikumar/Papers/EntropyAndCounting.pdf}

\bibitem{Sason21}
I. Sason, ``A generalized information-theoretic approach for bounding the number of independent
sets in bipartite graphs,'' {\em Entropy}, vol.~23, no.~3, paper~270, pp.~1--14, 2021.
\url{https://doi.org/10.3390/e23030270}

\bibitem{Sason21b}
I. Sason, ``Entropy-based proofs of combinatorial results on bipartite graphs,''
{\em Proceedings of the 2021 IEEE International Symposium on Information Theory},
pp.~3225--3230, Melbourne, Australia, July 12--20, 2021.
\url{https://doi.org/10.1109/ISIT45174.2021.9518068}

\bibitem{Sason25}
I. Sason, ``On H-intersecting graph families and counting of homomorphisms,''
{\em AIMS Mathematics}, vol.~10, no.~3, paper~290, pp.~6355--6378, March 2025.
\url{https://doi.org/10.3934/math.2025290}

\bibitem{WangTL23}
Z. Wang, J. Tu, and R. Lang, ``Entropy, graph homomorphisms, and dissociation sets,'' {\em Entropy}, vol.~25,
paper.~163, pp.~1--11, January 2023. \url{https://www.mdpi.com/1099-4300/25/1/163}

\bibitem{SahSSZ20}
A. Sah, M. Sawhney, D. Stoner, and Y. Zhao, ``A reverse Sidorenko inequality,'' {\em Inventiones Mathematicae}, vol.~221, pp.~665--711,
August 2020. \url{https://doi.org/10.1007/s00222-020-00956-9}

\bibitem{GareyJ79}
M. R. Garey and D. S. Johnson, {\em Computers and Intractability: A Guide to the Theory of NP-Completeness},
W. H. Freeman and Company, San Francisco, 1979.

\bibitem{CoverT06}
T. M. Cover and J. A. Thomas, {\em Elements of Information Theory}, second edition,
John Wiley \& Sons, 2006. \url{https://doi.org/10.1002/047174882X}

\bibitem{GrahamKP94}
R. L. Graham, D. E. Knuth, and O. Patashnik, {\em Concrete Mathematics: A Foundation for Computer Science},
second edition, Addison–Wesley, 1994.

\bibitem{Sidorenko93}
A. Sidorenko, ``A correlation inequality for bipartite graphs,'' {\em Graphs and Combinatorics}, vol.~9, pp.~201--204,
June 1993. \url{https://doi.org/10.1007/BF02988307}

\bibitem{Simonovits84}
M. Simonovits, ``Extremal graph problems, degenerate extremal problems and super-saturated graphs,''
Progress in graph theory, pp.~419--437, Academic Press, Toronto, Ontario, Canada, 1984.
\url{https://users.renyi.hu/~miki/waterloo.pdf}

\bibitem{ConlonFS10}
D. Conlon, J. Fox, and B. Sudakov, ``An approximate version of Sidorenko’s conjecture,'' {\em  Geometric
and Functional Analysis}, vol.~20, pp.~1354--1366, October 2010. \url{https://doi.org/10.1007/s00039-010-0097-0}

\bibitem{ConlonFS10b}
D. Conlon, J. Fox, and B. Sudakov, ``Sidorenko's conjecture for a class of graphs: an exposition,'' {\em unpublished note}, 2010.
\url{https://arxiv.org/abs/1209.0184}

\end{thebibliography}
\end{document}